\documentclass[11pt]{article}

\usepackage[left=1in,top=1in,right=1in,bottom=1in,head=.1in]{geometry}

\usepackage{pratik}

\newcommand{\titletext}{Optimal Ridge Regularization for Out-of-Distribution Prediction}

\author{
    Pratik Patil\footremember{berkeleystats}{Department of Statistics, University of California, Berkeley, CA 94720, USA.} \\ {\small \texttt{pratikpatil@berkeley.edu}} 
    \and
    Jin-Hong Du\footremember{cmustats}{Department of Statistics and Data Science, Carnegie Mellon University, Pittsburgh, PA 15213, USA.}\footremember{cmumld}{Machine Learning Department, Carnegie Mellon University, Pittsburgh, PA 15213, USA.} \\ {\small \texttt{jinhongd@andrew.cmu.edu}} 
    \and
    Ryan {J.} Tibshirani\footrecall{berkeleystats} \\ {\small \texttt{ryantibs@berkeley.edu}}
}

\date{\vspace{-20pt}}

\begin{document}

\maketitle

\begin{abstract}
We study the behavior of optimal ridge regularization and optimal ridge risk for out-of-distribution prediction, where the test distribution deviates arbitrarily from the train distribution. We establish general conditions that determine the sign of the optimal regularization level under covariate and regression shifts. These conditions capture the alignment between the covariance and signal structures in the train and test data and reveal stark differences compared to the in-distribution setting. For example, a negative regularization level can be optimal under covariate shift or regression shift, even when the training features are isotropic or the design is underparameterized. Furthermore, we prove that the optimally tuned risk is monotonic in the data aspect ratio, even in the out-of-distribution setting and when optimizing over negative regularization levels. In general, our results do not make any modeling assumptions for the train or the test distributions, except for moment bounds, and allow for arbitrary shifts and the widest possible range of (negative) regularization levels.
\end{abstract}

\section{Introduction}
\label{sec:introduction}

Regularization plays a crucial role in statistical modeling and is commonly incorporated into optimization-based models through a regularization term. 
Its effectiveness relies on properly scaling the regularization term, which is controlled by a penalty parameter that the data scientist needs to tune.
Recent work in machine learning (precise references given shortly) has shed light on some rather surprising behavior exhibited by the optimal regularization level in overparameterized prediction models, which can be zero or even negative in certain problems with moderate signal-to-noise ratio and high dimensionality. 
This stands in contrast to the ``typical'' behavior from classical low-dimensional learning theory.

With this motivation, our paper focuses on two key questions for high-dimensional ridge regression:

\begin{enumerate}
    \item[(\Qsnoparen{1})] What is the behavior of the \emph{optimal ridge penalty}, as a function of parameters such as signal-to-noise ratio, data aspect ratio, feature correlations, and signal structure?
    \item[(\Qsnoparen{2})] What is the behavior of the \emph{optimally tuned ridge risk}, as a function of these same problem parameters?
\end{enumerate}

To set the notation, let $(\bx_i, y_i)$ for $i \in [n]$ be i.i.d.\ observations in $\RR^{p} \times \RR$ representing the feature vector and response. 
Denote the data matrix as $\bX=[\bx_1,\dots,\bx_n]^{\top} \in \RR^{n \times p}$ and the associated response vector as $\by=[y_1,\dots,y_n]^{\top} \in \RR^n$.
Given a ridge penalty $\lambda > 0$, recall the ridge regression fits:
\begin{equation}
    \label{eq:ridge_optprob}
    \hat{\bbeta}^\lambda = \argmin_{\bb \in \RR^p} \,
    \|\by - \bX \bb\|_2^2 / n + \lambda \|\bb\|_2^2.
\end{equation}
Ridge regression \citep{hoerl_kennard_1970_1,hoerl_kennard_1970_2} has received considerable recent attention, particularly in settings involving overparameterization, such as double descent \citep[see, e.g.,][and references therein]{belkin_hsu_xu_2020, hastie2022surprises, muthukumar_vodrahalli_subramanian_sahai_2020} and benign overfitting \citep{bartlett_long_lugosi_tsigler_2020,koehler2021uniform,mallinar2022benign}.
This interest in ridge regression, especially its ``ridgeless'' limit, where $\lambda \to 0^+$, owes to its peculiar double/multiple descent risk behavior in overparameterized regimes, which (on the surface) challenges the conventional understanding of the role of regularization \citep{hastie2020ridge}.
By defining the ridge estimator as:
\begin{equation}
    \label{eq:ridge}
    \hat{\bbeta}^\lambda 
    = (\bX^\top \bX / n + \lambda \bI_p)^{\dagger} \bX^\top \by / n
    = \bX^\top (\bX \bX^\top / n + \lambda \bI_n)^{\dagger} \by / n, 
\end{equation}
where $\bA^{\dagger}$ denotes the Moore-Penrose pseudoinverse of $\bA$, we simultaneously accommodate the case of $\lambda > 0$, in which case \eqref{eq:ridge} reduces to the unique ridge predictor obtained using \eqref{eq:ridge_optprob}, and the case of $\lambda= 0$, in which case \eqref{eq:ridge} becomes the minimum $\ell_2$-norm interpolator among many solutions to problem \eqref{eq:ridge_optprob} when $p \ge \rank({\bX}) = n$.
Note that the above definition \eqref{eq:ridge} is well defined even when $\lambda < 0$.
For more background on the formulation of ridge regression with $\lambda < 0$, see \Cref{sec:introduction-details}.

Partial answers to questions \Qs{1} and \Qs{2} are known for the \emph{in-distribution} squared prediction risk, defined as:
\begin{equation}
    \label{eq:ridge_prederr}
    R(\hbeta^\lambda)
    = \EE_{\bx_0, y_0} [ ( \bx_0^\top \hat{\bbeta}^\lambda - y_0 )^2 \mid \bX, \by ],
\end{equation}
where $(\bx_0, y_0)$ is a test point sampled independently from the \emph{same} distribution $P_{\bx, y}$ as the training data.
Note that the prediction risk \eqref{eq:ridge_prederr} is conditional on the training data and is a function of $(\bX, \by)$ and the properties of $P_{\bx,\by}$.
Arguably, the cleanest answers to \Qs{1} and \Qs{2} are obtained under proportional asymptotics where the sample size $n$ and the feature size $p$ diverge proportionally to an \emph{aspect ratio} $\phi\in(0,\infty)$.
Under certain additional assumptions, the risk \eqref{eq:ridge_prederr} then almost surely converges to a limit $\sR(\lambda, \phi)$ that depends only on coarse properties of $P_{\bx, y}$.
Analyzing the behavior of the optimal ridge penalty $\lambda^*$ (which minimizes $\sR(\lambda, \phi)$ over $\lambda$) and the optimal risk $\sR(\lambda^{*}, \phi)$ then consequently allows us to answer questions \Qs{1} and \Qs{2}, respectively.

For \Qs{1}, consider a well-specified linear model $\by = \bX \bbeta + \bvarepsilon$, where the noise $\bepsilon \sim (\bm{0}_n, \sigma^2 \bI)$ is independent of $\bX$, and the signal is random with $\bbeta \sim(\zero_p, (\alpha^2/p) \bI)$. 
Remarkably, despite the lack of a closed-form expression for $\sR(\lambda^{*}, \phi)$, \citet{dobriban_wager_2018} show that $\lambda^{*} = \phi/\SNR > 0$, where $\SNR = \alpha^2/\sigma^2$ is the signal-to-noise ratio.
However, in real-data experiments, it has been observed that a negative ridge penalty can be optimal \citep{kobak_lomond_sanchez_2020}.
Motivated by this, \citet{wu_xu_2020,richards_mourtada_rosasco_2021} analyze the sign behavior of $\lambda^*$ beyond random isotropic signals and establish sufficient conditions for when $\lambda^{*}<0$ or $\lambda^{*}=0$. 

For \Qs{2}, again remarkably, it follows from the results of \citet{dobriban_wager_2018} that for isotropic features and random isotropic signals, the risk $\sR(\lambda, \phi)$ increases monotonically with the data aspect ratio $\phi$.
Recent work by \citet[Theorem 6]{patil2023generalized} extends this result to anisotropic features and deterministic signals (with arbitrary response distributions of bounded moments), demonstrating that optimal ridge regression exhibits a monotonic risk profile and avoids double and multiple descents.
In a certain sense, these results imply the sample-wise monotonicity of the optimal expected conditional risk (over the training data set).

We work to answer \Qs{1}-\Qs{2} more comprehensively across essentially all possible in-distribution ridge prediction problems. 
Furthermore, we will generalize this by also considering the \emph{out-of-distribution} (OOD) setting, where the test point $(\bx_0, y_0)$ in \eqref{eq:ridge_prederr} has a different distribution $P_{x_0,y_0}$ than the train distribution $P_{\bx,y} = P_{\bx} \cdot P_{y \mid \bx}$. 
Distribution shift occurs in many practical machine learning applications and has gained increasing attention in the learning community.
We focus on two common types of distribution shifts:
\begin{itemize}[nosep]
    \item[(i)] 
    \emph{Covariate shift}: where $P_{\bx_0} \neq P_{\bx}$ but $P_{y_0 \mid \bx_0} = P_{y \mid \bx}$.
    \item[(ii)]
    \emph{Regression shift}: where $P_{y_0 \mid \bx_0} \neq P_{y \mid \bx}$ but $P_{\bx_0} = P_{\bx}$.
\end{itemize}
Thus, we also study following generalizations of \Qs{1} and \Qs{2}:

\begin{enumerate}
    \item[(\Qsnoparen{1'})]
    \emph{How does distribution shift alter optimal regularization?}
    Specifically, under what types of shift is the optimal penalty $\lambda^*$ in the OOD setting notably different (typically smaller) compared to the in-distribution setting?
    \item[(\Qsnoparen{2'})]
    \emph{How does distribution shift alter optimal risk behavior?}
    Specifically, is $\sR(\lambda^{*}, \phi)$ still monotonic in $\phi$ when $\lambda^{*}$ changes due to the distribution shift?
    Conversely, is optimal regularization \emph{necessary} to ensure monotonic risk?
\end{enumerate}

\subsection{Summary of Results and Paper Outline}

\textbf{Extended risk characterization.}
In \Cref{sec:preliminaries}, we extend the scope of risk characterization for ridge regression for the out-of-distribution setting (\Cref{prop:ood-risk-asymptotics}) that: 
(i) does not assume any model for the train or the test distribution, apart from certain moment bounds on the train and test response distributions,
(ii) does not assume that the spectrums of feature covariance or signal projections converge to fixed distributions, 
and (iii) allows for the widest possible range of (negative) regularization level $\lambda_{\min}$ (see \Cref{def:lower_bound_negative_regularization}).

\textbf{Properties of optimal regularization.}
In \Cref{sec:optimal_regularization_signs}, we characterize the conditions that determine the sign of $\lambda^*$ under covariate shift (\Cref{thm:stationary-point-cov-shift}) and regression shift (\Cref{thm:stationary-point-sig-shift}). 
These conditions capture the general alignment between the signal and the covariance spectrum, which isolates the cases where the sign of $\lambda^*$ under the OOD setting changes from the in-distribution setting.
Our work subsumes and extends previously known results on optimal ridge regularization to the best of our knowledge (see \Cref{tab:opt-reg} for precise comparisons).

\textbf{Properties of optimal risk.}
In \Cref{sec:optimal_risk_monotonicity}, we show that the OOD risk of the optimally tuned ridge $\sR(\lambda^{*}, \phi)$ is monotonic in the data aspect ratio $\phi$ and $\SNR$ (\Cref{thm:monotonicty}).
Furthermore, we establish a partial converse (\Cref{thm:nonmonotonicity}) that shows risk non-monotonicity under suboptimal regularization.  
To prove our results, we exploit the equivalences between subsampling and ridge equivalences of \cite{patil2023generalized} to the OOD setting and also allow for tuning over negative regularization (\Cref{thm:negative-lam}).

\subsection{Related Work and Comparisons}
\label{sec:related_work}

\begin{table*}[!t]
    \centering
    \caption{
    \textbf{Optimal regularization landscape in ridge regression.}
    Here, `\iso' indicates either an isotropic feature or signal covariance, and `\noniso' indicates anisotropic features or signal covariance.
    For the data aspect ratio $\phi$, `all' indicates $\phi\in(0,\infty)$, `under' indicates $\phi\in(0,1)$ for the underparameterized regime, and `over' indicates $\phi\in(1,\infty)$ for the overparameterized regime.
    For the minimum penalty $\lambda_{\min}$, `neg' and `\negg' respectively indicate the naive (loose) and exact lower bound on the negative values (see \Cref{def:lower_bound_negative_regularization}).
    For the optimal penalty $\lambda^*$, \smash{\setlength{\fboxsep}{0pt}\colorbox{green!10}{green\vphantom{d}{}}} and \smash{\setlength{\fboxsep}{0pt}\colorbox{red!10}{red\vphantom{g}{}}} contrast the cases where the sign changes.
    `Arb. Mod.', `Arb. SNR.', and `Arb. Spec.' indicate allowing for arbitrary response model, signal-to-noise ratio, and feature covariance spectrum, respectively.
    (See \Cref{tab:ref} for the reference abbreviation key.)
    }\scriptsize%
    \label{tab:opt-reg}
    \begin{tabularx}{\textwidth}{C{0.25cm}C{0.25cm}C{0.25cm}C{0.25cm}C{0.75cm}C{0.5cm}C{0.45cm}C{0.45cm}C{0.45cm}L{4.5cm}C{0.18cm}C{2.95cm}}
        \toprule
        $\bSigma$
        & $\bbeta$ & $\bSigma_0$
        & $\bbeta_0$
        & $\phi \lessgtr 1$ & $\lambda_{\min}$ & \textbf{Arb. Mod.} & \textbf{Arb. SNR} & \textbf{Arb. Spec.} & \textbf{Additional Specific Data Geometry Conditions}
        & \textbf{$\lambda^*$}
        &  \textbf{Reference} \\
        \midrule
        \multicolumn{4}{c}{In-distribution} &&&&&&&&\\ 
         \noniso
        & \iso
        & $\bSigma$
        & $\bbeta$
        & all & zero & \xmark & \cmark & \xmark & \cellcolor{lightgray!15} & + &
        \colorrefold[\hyperlink{ref:DW}{DW}, Thm. 2.1]%
        \\ 
        \colorgray\iso & \colorgray\noniso & \colorgray$\bSigma$ & \colorgray$\bbeta$ & \cellcolor{green!10}all & zero & \xmark & \cmark & \xmark &  \cellcolor{lightgray!15} & \cellcolor{green!10}+ &\colorrefold[\hyperlink{ref:HMRT}{HMRT}, Cor. 5] \\
        \arrayrulecolor{black!10} \cmidrule(lr){5-12}
         &&&& \cellcolor{green!10}under & neg & \xmark & \cmark & \xmark & \cellcolor{lightgray!15} & \cellcolor{green!10}+ &
        \colorrefold[\hyperlink{ref:WX}{WX}, Prop. 6] \\
         &  &  &  & over & neg & \xmark & \xmark & \xmark & Strict misalignment of $(\bSigma, \bbeta)$ & + &
        \colorrefold[\hyperlink{ref:WX}{WX}, Thm. 4] \\
         &  &  &  & over & neg& \xmark & \xmark & \xmark & \multirow{2}{4.cm}{Strict alignment of $(\bSigma, \bbeta)$ and/or special feature model} & $-$ & \colorrefold[\hyperlink{ref:WX}{WX}, Thm. 4,\,Prop. 7]
        \\ 
         &  &  &  & over & zero & \xmark & \xmark & \xmark & & $0$ & \colorrefold[\hyperlink{ref:RMR}{RMR}, Cor. 2] 
        \\ 
        \arrayrulecolor{black!10} \cmidrule(lr){5-12}
        &  &  &  & under & \negg & \cmark & \cmark & \cmark & \cellcolor{lightgray!15} & + & \colorrefnew \textbf{\Cref{thm:stationary-point-no-shift}~\eqref{thm:stationary-point-no-shift-item-underparameterized}}\\
        \multirow{-6}{*}{\noniso} & 
        \multirow{-6}{*}{\noniso} & 
        \multirow{-6}{*}{$\bSigma$} & 
        \multirow{-6}{*}{$\bbeta$}
        & over & \negg & \cmark & \cmark & \cmark & General alignment of $(\bSigma,\bbeta,\sigma^2)$ & $-$  & \colorrefnew \textbf{\Cref{thm:stationary-point-no-shift}~\eqref{thm:stationary-point-no-shift-item-overparameterized} }\\
        
        \arrayrulecolor{black!50}\midrule
        \arrayrulecolor{black}\multicolumn{4}{c}{Out-of-distribution} \\
        \noniso & \iso & $\bSigma_0$
        & $\bbeta$
        & all & \negg & \cmark & \cmark & \cmark & \cellcolor{lightgray!15} & + &
        \colorrefnew \textbf{\Cref{thm:stationary-point-cov-shift-iso}} \\
        \arrayrulecolor{black!10} \cmidrule(lr){5-12}
        \colorgray\noniso & \colorgray\noniso & \colorgray$\bSigma_0$ & \colorgray$\bbeta$ & under & \negg & \cmark & \cmark & \cmark &  \cellcolor{lightgray!15} & + & \colorrefnew \textbf{\Cref{thm:stationary-point-cov-shift}~\eqref{thm:stationary-point-cov-shift-item-underparameterized}} \\
        \noniso & \noniso & $\bI$ & $\bbeta$ & over & \negg & \cmark & \cmark & \cmark & \cellcolor{lightgray!15} & $+$ &  \colorrefnew \textbf{\Cref{thm:stationary-point-cov-shift}~\eqref{thm:stationary-point-cov-shift-item-overparameterized-estimation-risk}} \\
        \colorgray\iso & \colorgray\noniso & \colorgray$\bSigma_0$ & \colorgray$\bbeta$ & \cellcolor{red!10}over & \negg & \cmark & \cmark & \cmark & General alignment of ($\bSigma_0,\bbeta,\sigma^2)$ & \cellcolor{red!10}
        $-$ & \colorrefnew \textbf{\Cref{thm:stationary-point-cov-shift}~\eqref{thm:stationary-point-cov-shift-item-overparameterized}} \\
        \arrayrulecolor{black!10} \cmidrule(lr){5-12}
        \multirow{3}{*}{\noniso}
        & \multirow{3}{*}{\noniso}
        & \multirow{3}{*}{$\bSigma$}
        & \multirow{3}{*}{$\bbeta_0$}
        & \cellcolor{red!10}under & \negg & \cmark & \cmark & \cmark & General alignment of ($\bSigma, \bbeta,\bbeta_0$) & \cellcolor{red!10} $-$  & \colorrefnew\textbf{\Cref{thm:stationary-point-sig-shift}~\eqref{thm:stationary-point-sig-shift-item-underparameterized}}, \eqref{eq:sig-shift-gen-cond}\\
        & & & & under & \negg & \cmark & \cmark & \cmark & General misalignment of ($\bSigma, \bbeta,\bbeta_0$)  & $+$ & \colorrefnew\textbf{\Cref{thm:stationary-point-sig-shift}~\eqref{thm:stationary-point-sig-shift-item-underparameterized}}, \eqref{eq:sig-shift-gen-cond}\\
        & & & & over & \negg & \cmark & \cmark & \cmark & General alignment of ($\bSigma, \bbeta,\bbeta_0,\sigma^2$) & $-$ & \colorrefnew\textbf{\Cref{thm:stationary-point-sig-shift}~\eqref{thm:stationary-point-sig-shift-item-overparameterized}}\\
        \arrayrulecolor{black} \bottomrule
    \end{tabularx}
    \label{tab:opt_reg}
\end{table*}

\textbf{Ridge risk characterization.}
The asymptotic risk of ridge regression has been extensively studied in the literature under proportional asymptotics when $p/n \to \phi \in (0, \infty)$, as $n,p \to \infty$ using tools from random matrix theory and statistical physics.
For well-specified linear models, expressions of risk asymptotics for the in-distribution setting are obtained by \citet{dobriban_wager_2018,hastie2022surprises}, among others. 
Historically, heuristic derivations of these expressions have also been derived for Gaussian process regression by \citet{sollich2001gaussian}.
Additionally, several works have explored risk asymptotics and its implications in different variants of ridge regression \citep{wei_hu_steinhardt,mel21arbitrary,loureiro21learning,jacot2020kernel,simon2021eigenlearning, zhou2023optimistic,bach2023high,pesce2023gaussian}.
The risk asymptotics for the OOD setting are obtained by \citet{canatar2021out}, \citet[Section E.5]{d2022underspecification}, \citet[Section S.6.5]{patil2022mitigating}, \citet{tripuraneni2021covariate}.
However, these works assume either random Gaussian features or a well-specified linear model or restrict to only the positive range of regularization.
Our work extends this literature by allowing for general response models and the widest possible range of negative regularization.

\textbf{Behavior of optimal regularization.}
Under random signals with isotropic covariance, \citet{dobriban_wager_2018} show that the asymptotic risk over the positive range of regularization $\lambda>0$ is minimized at $\lambda^\star = \phi /\SNR$.
Here $\SNR = \alpha^2/\sigma^2$, where $\sigma^2$ is the noise energy and $\alpha^2$ is the signal energy.
Remarkably, this result is invariant of the feature covariance.
Similar results under Gaussian assumptions are derived by \citet{dicker2016ridge}, and \citet{han2023distribution} extend the result to most signals in the unit ball with high probability.
However, \citet{kobak_lomond_sanchez_2020} demonstrate that optimal regularization can be negative for certain signal and covariance structures in real datasets.
Motivated by these curious experiments, \citet{wu_xu_2020} provide sufficient conditions for optimal regularization of the Bayes risk under anisotropic feature covariance and random signal, assuming a limiting spectrum distribution of the covariance matrix $\bSigma$ and alignment conditions between the eigenvalues of $\bSigma$ and the projections of the signal $\bbeta$ onto the eigenspace of $\bSigma$.
Furthermore, \citet{richards_mourtada_rosasco_2021} consider strict alignment conditions for a special feature model but do not explicitly consider negative regularization.
We refer to these conditions as \emph{strict (mis)alignment} conditions in this paper.

Our paper extends the scope of the aforementioned results for the OOD setting with general response models and allows for the widest possible range of negative regularization.
The main differences to \citet{wu_xu_2020} are the assumptions (linear models and limiting spectrum distribution) and the hypothesis of the theorem (aligned/misaligned). 
Their analysis only considers high \SNR regimes and analyzes the behavior of optimal regularization for bias when $\phi > 1$, although they also provide an upper bound for the noise level under special cases.
We provide generic sufficient conditions.
We call these \emph{general (mis)alignment} conditions in the paper.
\Cref{tab:opt-reg} provides a detailed comparison summary.

\textbf{Behavior of optimal risk.}
Under a random isotropic signal, \citet{dobriban_wager_2018} obtain the expression for limiting optimal risk, which can be shown to be monotonic in $\phi$.
\citet{patil2023generalized} extends these theoretical guarantees to features with an arbitrary covariance matrix and a general moment-bounded response. 
However, their analysis is limited to in-distribution prediction risk and positive regularization.
Recent works have also explored other aspects of the monotonicity of optimal risk; see, for example, \citet{nakkiran2021optimal,simon2023better,yang2023dropout}.

We extend the monotonicity result in \citet{patil2023generalized} to the OOD setting and allow for negative regularization.
In addition, we also show the monotonicity of the optimal risk in $\SNR$.
The proof technique for risk monotonicity leverages the equivalences between subsampling and ridge regularization established by \citet{du2023gcv,patil2023generalized}. 
However, the equivalences in these works only consider cases where the ridge penalty is non-negative.
We extend their analyses to accommodate negative regularization, which requires extending the properties of the parameters that appear in certain fixed-point equations under negative regularization (see \Cref{sec:analytic-properties-fp-sols}).

\section{Out-of-Distribution Risk Asymptotics}
\label{sec:preliminaries}

Before describing the properties of $\lambda^*$ in \Cref{sec:optimal_regularization_signs} and the behavior of optimal risk $\sR(\hat{\bbeta}^{\lambda^*})$ in \Cref{sec:optimal_risk_monotonicity}, we provide the risk asymptotics in this section.
For the reader's convenience, we summarize all our notation in \Cref{sec:organization_notation_supp}.
We state assumptions on the train and test distributions below.

\subsection{Data Assumptions}

We first define a general feature and response distribution, which we will use in our subsequent assumption shortly.

\begin{definition}
    [General feature and response distribution]
    \label{def:dist}
    For $\bx \sim P_{\bx}$, it can be decomposed as $\bx = \bSigma^{1/2}\bz$, where $\bz \in \RR^{p}$ contains i.i.d. entries with mean $0$, variance $1$, and $(4+\mu)$-th moment uniformly bounded for some $\mu > 0$.
    Here $\bSigma \in \RR^{p \times p}$ is deterministic and symmetric with eigenvalues uniformly bounded away from $0$ and $+\infty$. 
    For $y \sim P_{y}$, it has mean $0$ and $(4+\nu)$-th moment uniformly bounded for some $\nu > 0$. 
    The $L_2$ linear projection parameter of $y$ onto $\bx$ is denoted by $\bbeta = \EE[\bx \bx^{\top}]^{-1} \EE[\bx y]$, and the variance of the conditional distribution $P_{y \mid \bx}$ is denoted by $\sigma^2$. 
    The joint distribution $P_{\bx,y}$ is parameterized by $(\bSigma, \bbeta, \sigma^2)$.
\end{definition}

\Cref{def:dist} imposes weak moment assumptions on covariates and responses, which are commonly used in random matrix theory and overparameterized risk analysis \citep{hastie2022surprises,bartlett_montanari_rakhlin_2021}. 
These assumptions encode a wide class of distributions over $\RR^{p+1}$.
By decomposing $\rd P_y = \rd P_{\bx} \cdot \rd P_{y\mid\bx}$, we can express the response as $y=\bx^{\top}\bbeta + \varepsilon$, where $\varepsilon$ is uncorrelated with $\bx$ and $\EE[\varepsilon^2]=\sigma^2$.
Note that the $L_2$ projection parameter $\bbeta$ minimizes the linear regression error \citep{gyorfi_kohler_krzyzak_walk_2006,buja2019models_1,buja2019models_2}\footnote{Technically, our results can accommodate the conditional variance $\sigma^2$ depending on $\bx$ with suitable regularity conditions, but for simplicity, we do not consider this variation in the current paper.}.
Also note that this formulation does not impose any specific structure on the conditional distribution $P_{y \mid \bx}$ and does not imply that $(\bx,y)$ follows a linear model, as $\varepsilon$ is also a function of $\bx$.
It is possible to further relax the assumption on the feature vector $x$ to only require an appropriate convex concentration (that proves versions of the Marchenko-Pastur law) \citep{louart2022sharp,cheng2022dimension} or even certain infinitesimal asymptotic freeness between the population covariance matrix $\bSigma$ and the sample covariance matrix $\bX^\top \bX / n$ \citep{lejeune2022asymptotics,patil2023asymptotically}.
We do not consider such relaxations here.

Under \Cref{def:dist}, the joint distribution $P_{\bx,y}=P_{\bx,y}(\bSigma, \bbeta, \sigma^2)$ is parameterized by $(\bSigma, \bbeta, \sigma^2)$.
We next state assumptions on the train and test distributions in terms of these distributions, allowing for different sets of parameters.

\begin{assumption}[Train and test distributions]\label{asm:train-test}     
    Assume that $P_{x, y}$ and $P_{x_0,y_0}$ are distributed according to \Cref{def:dist}, parameterized by $(\bSigma, \bbeta, \sigma^2)$ and $(\bSigma_0, \bbeta_0, \sigma_0^2)$, respectively.
\end{assumption}

In this paper, we consider the following types of shifts:
\begin{itemize}[nosep]
    \item[(i)] 
    \emph{Covariate shift}: where $\bSigma \neq \bSigma_0$ but $(\bbeta,\sigma)=(\bbeta_0,\sigma_0)$.
    \item[(ii)]
    \emph{Regression shift}: where $\bSigma = \bSigma_0$ but $(\bbeta,\sigma) \neq (\bbeta_0,\sigma_0)$.
    \item[(iii)] 
    \emph{Joint shift}: where $\bSigma \neq \bSigma_0$ and $(\bbeta,\sigma) \neq (\bbeta_0,\sigma_0)$.
\end{itemize}
Observe that this framework also encompasses various risk notions (even for the in-distribution setting), including the \emph{estimation risk}, which arises when $(\bSigma_0,\bbeta_0,\sigma_0^2)=(\bI,\bbeta,0)$.

\subsection{Out-of-Distribution Risk Asymptotics}

In this section, we obtain the asymptotic risk of ridge regression in the OOD setting.
Most of the papers on ridge regression consider the range of regularization $\lambda \geq 0$.
Motivated by empirical findings in \citet{kobak_lomond_sanchez_2020} that negative regularization can be optimal in real datasets, some recent works consider negative regularization; see, e.g., \citet{wu_xu_2020,patil2021uniform,patil2022mitigating}, using a naive lower bound of $- r_{\min} (1 - \sqrt{\phi})^2$ for $\lambda$.
A tighter lower bound can be obtained from Theorem 3.1 of \citet{lejeune2022asymptotics}, which provides an explicit characterization of the smallest nonzero eigenvalue of Wishart-type matrices.
This bound is derived by explicitly identifying the analytic continuation to the real line of a unique solution to a certain fixed-point equation over the (upper) complex half-plane \citep{silverstein1995analysis,dobriban2015efficient}.
The new bound can significantly outperform the previous naive lower bound (see Figure 1 of \cite{lejeune2022asymptotics}).

\begin{definition}[Lower bound on negative regularization]
    \label{def:lower_bound_negative_regularization}
  Let $\mu_{\min} \in \RR$ be the unique solution, satisfying $\mu_{\min} > -r_{\min}$, to the equation:
\begin{equation}
    1 = \phi \otr[\bSigma^2 (\bSigma + \mu_{\min} \bI)^{-2}],
\end{equation}
and let $\lambda_{\min}(\phi)$ be given by:
\begin{equation}\label{eq:lam_min}
    \lambda_{\min}(\phi) 
    = \mu_{\min} - \phi \otr[\bSigma (\bSigma + \mu_{\min} \bI)^{-1}].
\end{equation}
  
\end{definition}
This enables feasible risk estimation over $\lambda\in (\lambda_{\min}(\phi),\infty)$.
Here $\otr[\bA]$ denotes the average trace $\tr[\bA] / p$ of a matrix $\bA \in \RR^{p \times p}$.
To reiterate, the difference between the bound \eqref{eq:lam_min} and the naive bound used in \citet{wu_xu_2020,patil2021uniform} can be significant, as seen in \Cref{fig:negative_optimal_our_condition} as well.

To characterize the asymptotic out-of-distribution (OOD) risk in \Cref{prop:ood-risk-asymptotics}, we first define the non-negative constants $\mu=\mu(\lambda, \phi)$ and $\tv=\tv(\lambda, \phi; \bSigma_0)$ as solutions of the following fixed-point equations:
{
\begin{align*}
    \mu
    = \lambda + \phi\otr[\mu\bSigma (\bSigma + \mu\bI)^{-1}],
    \quad
    \text{and}
    \quad
    \tv
    = \frac{\phi \otr[\bSigma_0\bSigma(\bSigma+\mu\bI)^{-2}]}{1-\phi \otr[\bSigma^2(\bSigma + \mu\bI)^{-2}]}.
\end{align*}
}
One can interpret $\mu$ as the level of implicit regularization ``self-induced'' by the data \citep{bartlett_montanari_rakhlin_2021,misiakiewicz2023six}.
Alternatively, it is also common to parameterize the equations using its inverse $v(\lambda, \phi) = \mu^{-1}(\lambda, \phi)$, which corresponds to the Stieltjes transform of the sample Gram matrix in the limit. With this notation in place, we can now extend the result in Eq. (11) of \citet{patil2023generalized} to the OOD setting as formalized below.

\begin{restatable}
    [Deterministic equivalents for OOD risk]
    {proposition}
    {PropOodRiskAsymptotics}
    \label{prop:ood-risk-asymptotics}
    Under \Cref{asm:train-test}, as $n,p\rightarrow\infty$ such that $p/n\rightarrow\phi\in(0,\infty)$ and $\lambda\in(\lambda_{\min}(\phi),\infty)$, the prediction risk $R(\hbeta^\lambda)$ defined in \eqref{eq:ridge_prederr} admits a deterministic equivalent $R(\hat{\bbeta}^{\lambda}) \asympequi \sR(\lambda, \phi)$, where the equivalent additively decomposes into:
    \begin{align}
         \sR(\lambda, \phi) :=\sB(\lambda, \phi) + \sV(\lambda, \phi)  + \sE(\lambda, \phi) + \kappa^2, \label{eq:risk-det-equiv}
    \end{align}
    with the following deterministic equivalents for the bias, variance, regression shift bias, and irreducible error:
    \begin{align*}
        \sB &= \mu^2\cdot \bbeta^{\top}(\bSigma+\mu\bI)^{-1}(\tv\bSigma+\bSigma_0) (\bSigma+\mu\bI)^{-1}\bbeta,
        \\
        \sV &= \sigma^2\tv, 
        \\
        \sS &= 2\mu \cdot \bbeta^\top (\bSigma + \mu\bI)^{-1} \bSigma_0 ( {\bbeta}_0 - \bbeta ), 
        \\
        \kappa^2 &= (\bbeta_0 - \bbeta)^\top \bSigma_0 (\bbeta_0 - {\bbeta}) + \sigma_0^2. 
    \end{align*}
\end{restatable}

Note that the deterministic equivalents presented in \Cref{prop:ood-risk-asymptotics} depend not only on the regularization parameters $(\lambda, \phi)$, but also on the problem parameters $(\bSigma, \bbeta, \sigma^2)$ and $(\bSigma_0, \bbeta_0, \sigma_0^2)$, which we have omitted for notational brevity.
Since the risk depends additively on $\sigma_0^2$, we focus mainly on the effect of $(\bSigma_0, \bbeta_0)$ in our analysis.
Extending this result to finite samples is possible by imposing additional distributional assumptions on the features and response.
Techniques in \citet{knowles2017anisotropic,cheng2022dimension,louart2022sharp}, among others, can be used to obtain non-asymptotic analogs of \Cref{prop:ood-risk-asymptotics}.
In this paper, we will focus only on the deterministic equivalents, which capture the first-order information (akin to expectation) of interest for our goals.

\section{Properties of Optimal Regularization}
\label{sec:optimal_regularization_signs}

In this section, we focus on the optimal ridge penalty $\lambda^*$ for the asymptotic out-of-distribution (OOD) risk, defined as\footnote{
Over the extended reals, there is at least one solution to \eqref{eq:opt-lam}.
In case there are multiple solutions $\lambda^*$ to the problem \eqref{eq:opt-lam}, the subsequent guarantees stated in the paper hold for any solution $\lambda^*$.}:
\begin{align}
    \lambda^* &\in \argmin_{\lambda\geq \lambda_{\min}(\phi)}\sR(\lambda,\phi) \label{eq:opt-lam}.
\end{align}
As discussed in \Cref{sec:related_work}, previous studies have explored the properties of $\lambda^*$ for ridge regression summarized in \Cref{tab:opt-reg}. 
However, these studies predominantly focus on specific scenarios, such as isotropic signals or features, and do not consider the full range of negative penalty values.
Furthermore, their investigations are restricted mainly to the in-distribution setting when $(\bSigma_0,\bbeta_0)=(\bSigma,\bbeta)$.
We broaden the scope of these results, considering more general scenarios, including anisotropic signals, the full range of (negative) regularization, and both in-distribution and OOD settings.

\subsection{In-Distribution Optimal Regularization}

We present our initial result for the in-distribution setting, which encompasses and extends the scope of previous works. 
Based on \Cref{prop:ood-risk-asymptotics}, we can characterize the properties of the optimal ridge penalty $\lambda^*$ defined in \eqref{eq:opt-lam} as follows.

\begin{restatable}
    [Optimal regularization sign for in-distribution risk]
    {theorem}
    {ThmStationaryPointNoShift}
    \label{thm:stationary-point-no-shift}
     Assume the setup of \Cref{prop:ood-risk-asymptotics} with $(\bSigma_0,\bbeta_0)=(\bSigma,\bbeta)$.
    \begin{enumerate}[leftmargin=4mm,itemsep=0pt,font=\normalfont]
        \item
        \label{thm:stationary-point-no-shift-item-underparameterized}
        \emph{(Underparameterized)} When $\phi<1$, we have $\lambda^*\geq 0$.

        \item
        \label{thm:stationary-point-no-shift-item-overparameterized}
        \emph{(Overparameterized)}
        When $\phi>1$, if for all $v < 1/\mu(0, \phi)$, the following general alignment holds:
        \begin{align}
            \frac{\otr[\bB\bSigma (v \bSigma + \bI)^{-2}] + \sigma^2 }{\otr[\bB\bSigma (v \bSigma + \bI)^{-3} ] + \sigma^2}
            &>
            \frac{\otr[\bSigma (v \bSigma + \bI)^{-2}]}{\otr[\bSigma (v \bSigma + \bI)^{-3}]},
            \label{eq:align-cond}
        \end{align}
        where $\bB=\bbeta\bbeta^{\top}$, then we have $\lambda^*< 0$.
    \end{enumerate}
\end{restatable}

It is worth mentioning that although we state our results for general deterministic signals, our analysis can also incorporate random signals. 
In such cases, when $\bbeta$ is random, one can simply replace $\bB$ in the conclusion with its expectation $\EE[\bB]$.
Next, we highlight some special cases of \Cref{thm:stationary-point-no-shift} and compare them with previously known results.

When $\bSigma = \bI$ or $\EE[\bB] = (\alpha^2/p) \bI$, it is easy to verify that the general alignment condition \eqref{eq:align-cond} does not hold (see \Cref{rmk:thm:stationary-point-no-shift-iso}).
This corresponds to the special cases studied by \citet{dicker2016ridge,dobriban_wager_2018}, where $\lambda^{*} \ge 0$.

The general alignment condition \eqref{eq:align-cond} in \Cref{thm:stationary-point-no-shift} encompasses the strict alignment conditions in \citet{wu_xu_2020,richards_mourtada_rosasco_2021}.
Under strict alignment conditions, \citet{wu_xu_2020} demonstrate that the optimal ridge penalty is negative in the overparameterized and noiseless setting.
They assume perfect alignment or misalignment between the signal distribution and the spectrum distribution of the covariance.
This also includes the strong and weak features models considered in \citet{richards_mourtada_rosasco_2021}.
When the signal is strictly aligned with the spectrum of $\bSigma$ and $\sigma^2=0$, it can be shown that \eqref{eq:align-cond} holds for all $\mu>0$ (see \Cref{prop:thm:stationary-point-no-shift-strict}).

The general alignment condition \eqref{eq:align-cond} allows for a broader range of signal and covariance structures.
For example, in scenarios where the signal is the average of the largest and smallest eigenvectors of $\Sigma$, the strict alignment condition does not hold.
However, these scenarios can still satisfy the general alignment conditions, as we will illustrate shortly.

In the noiseless setting, when $\sigma^2 = 0$, the alignment condition \eqref{eq:align-cond} can be expressed succinctly as: 
\begin{equation}
    \frac{\partial h(\mu, \bSigma_{\bbeta})}{\partial \mu} < \frac{h(\mu, \bSigma)}{\partial \mu}, \label{eq:align-cond-zero-noise}
\end{equation}
by defining the function $h(\cdot, \bSigma_{\bbeta}) \colon \mu \mapsto \log \otr[\bB \bSigma (\bSigma+\mu\bI)^{-2}]$ and $\bSigma_{\bbeta}=\bSigma\bbeta\bbeta^{\top}$.
At a high level, these alignment conditions capture how aligned the signal vector is with the feature covariance matrix. 
When the alignment is strong, it indicates that the problem is effectively low-dimensional. 
In such cases, less regularization is needed if the signal energy in this effective direction is sufficiently large. 

In the noisy setting, when $\sigma^2 \neq 0$, the condition \eqref{eq:align-cond} explicitly trades off the alignment of the signal and the noise level.
While \citet[Proposition 5]{wu_xu_2020} and \citet[Corollary 2]{richards_mourtada_rosasco_2021} provide upper bounds on $\sigma^2$ for optimal negative regularization under restricted data models, \eqref{eq:align-cond} applies to a wider class of data models.
In this sense, \Cref{thm:stationary-point-no-shift} extends the previous results on the occurrence of optimal negative regularization to a more general setting.

\begin{figure}[!t]
    \centering
    \includegraphics[width=0.99\textwidth]{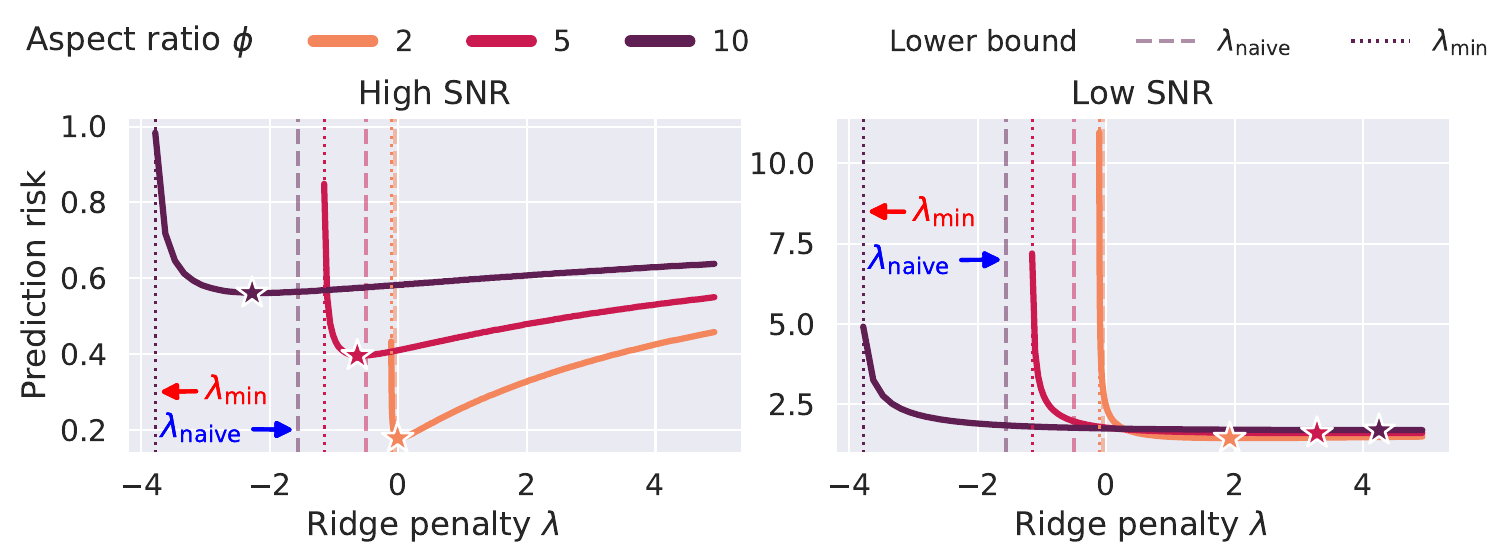}
    \caption{
    \textbf{Illustration of negative or positive optimal regularization under general alignment.}
    We plot the in-distribution risk of ridge regression against the penalty $\lambda$ for varying data aspect ratios $\phi$ in the overparameterized regime.
    The left and right panels correspond to scenarios when \SNR is high ($\sigma^2=0.01$) and low ($\sigma^2=1$), respectively.
    The data model has a covariance matrix $(\bSigma_{\mathrm{ar1}})_{ij}:= \rhoar^{|i-j|}$ with parameter $\rhoar=0.5$, and a coefficient $\bbeta:=\frac{1}{2}(\bw_{(1)} + \bw_{(p)})$, where $\bw_{(j)}$ is the $j$th eigenvector of $\bSigma_{\mathrm{ar1}}$.
    }
    \label{fig:negative_optimal_our_condition}
\end{figure}

In \Cref{fig:negative_optimal_our_condition}, we illustrate our theoretical result of \Cref{thm:stationary-point-no-shift}.
When \SNR is high, we observe that the optimal ridge penalty can be negative in the overparameterized regime.
In particular, for $\phi=10$, the negative optimal ridge penalty exceeds the bound used in \citet{wu_xu_2020}.
It is worth noting that the scenario depicted in \Cref{fig:negative_optimal_our_condition} does not satisfy the strict alignment condition, but is incorporated by our characterization in \Cref{thm:stationary-point-no-shift}.

\subsection{Out-of-Distribution Optimal Regularization}

We now investigate the behavior of the optimal ridge penalty under covariate shift or regression shift.
In the in-distribution setting, when $\bSigma_0=\bSigma$ and the signals are isotropic ($\bbeta\bbeta^{\top}\asympequi (\alpha^2/p) \bI$), previous works \citep{dobriban_wager_2018,han2023distribution} show that the optimal ridge penalty is $\lambda^*=\phi/\SNR\geq 0$.
Interestingly, even when allowing for negative regularization and covariate shift ($\bSigma_0\neq\bSigma$), it is easy to check that the optimal $\lambda$ remains positive in data models with isotropic signals for any $\phi \in (0, \infty)$.

\begin{restatable}
    [Optimal regularization under covariate shift and random signal]
    {proposition}
    {ThmStationaryPointCovShiftIso}
    \label{thm:stationary-point-cov-shift-iso}    
    When $\bSigma_0\neq\bSigma$ and $\bbeta_0=\bbeta$, assuming isotropic signals $\EE[\bbeta\bbeta^{\top}]= (\alpha^2 / p) \bI$, we have $\lambda^*(\phi) = \phi /\SNR=\argmin_{\lambda}\EE_{\bbeta}[\sR(\lambda,\phi)]$.
    Furthermore, for $\lambda_{\min}<0$ such that $\bX^{\top}\bX/n \succ \lambda_{\min}\bI$, even the non-asymptotic OOD risk \eqref{eq:ridge_prederr} is minimized at $\lambda^*_p = \phi_p /\SNR$, where $\phi_p = p/n$.
\end{restatable}

We remark that \Cref{thm:stationary-point-cov-shift-iso} can also be seen as a result of a Bayes optimality argument.
However, we provide a more direct argument in \Cref{sec:thm:stationary-point-cov-shift-iso-proof}.
A result of this flavor for random features is also obtained by \citet[Proposition 6.1]{tripuraneni2021covariate}.
An interesting observation from \Cref{thm:stationary-point-cov-shift-iso} is that the optimal ridge penalty does not depend on $\bSigma_0$.
This implies that when the signal is isotropic, one does not need to worry about the covariate shift when tuning the penalty.
Therefore, generalized cross-validation for in-distribution risks \citep{patil2021uniform} can still yield optimal penalties even for OOD risks under isotropic signals!

However, it is important to note that this result holds specifically for random isotropic signals, which may not be realistic in practice. 
In scenarios where (near) random isotropic signals are not present, the question of data-dependent regularization tuning in the OOD setting is not as straightforward.
We do not consider data-dependent tuning in the current paper and instead focus on the theoretical properties of the (oracle) optimal regularization (and the corresponding risk).

\Cref{thm:stationary-point-cov-shift-iso} suggests that the optimal penalty $\lambda^*$ remains invariant for OOD risks under isotropic signals. 
Although random isotropic signals make the theory more tractable, they generally are not realistic in practice. 
The following result examines the optimal ridge penalty under deterministic signals, with similar conclusions holding for random anisotropic signals.

\begin{restatable}
    [Optimal regularization under covariate shift and deterministic signal]
    {theorem}
    {ThmStationaryPointCovShift}
    \label{thm:stationary-point-cov-shift}
    Assume the setting of \Cref{prop:ood-risk-asymptotics} with $\bSigma_0\neq\bSigma$, $\bbeta_0=\bbeta$.
    \begin{enumerate}[leftmargin=4mm,itemsep=0pt,font=\normalfont]
        \item
        \label{thm:stationary-point-cov-shift-item-underparameterized}
        \emph{(Underparameterized)} When $\phi<1$, we have $\lambda^*\geq 0$.

        \item
        \label{thm:stationary-point-cov-shift-item-overparameterized-estimation-risk}
        \emph{(Overparameterized)} When $\phi>1$, if $\bSigma_0=\bI$ (corresponding to the estimation risk), then we have $\lambda^* \ge 0$.

        \item
        \label{thm:stationary-point-cov-shift-item-overparameterized}
        \emph{(Overparameterized)} When $\phi>1$, if $\bSigma=\bI$ and
        {
        \begin{align}
            \otr[\bSigma_0\bB] > \otr[\bSigma_0]\left(\otr[\bB] + \frac{(1+\mu(0,\phi))^3}{\mu(0,\phi)^3}\sigma^2\right), \label{eq:align-cov}
        \end{align}
        }
        where $\bB=\bbeta\bbeta^{\top}$, then we have $\lambda^* < 0$.
    \end{enumerate}
\end{restatable}

When $\phi<1$, Part \eqref{thm:stationary-point-cov-shift-item-underparameterized} of \Cref{thm:stationary-point-cov-shift} suggests that the optimal ridge penalty $\lambda^*$ remains non-negative even under covariate shift.
Similarly, when $\phi > 1$, Part \eqref{thm:stationary-point-cov-shift-item-overparameterized-estimation-risk} of \Cref{thm:stationary-point-cov-shift} (for the estimation risk) also guarantees a non-negative $\lambda^*$.
However, it is quite surprising that in the overparameterized regime, even with deterministic features, the optimal ridge penalty $\lambda^*$ could be negative in Part \eqref{thm:stationary-point-cov-shift-item-overparameterized}.
In particular, in noiseless setting ($\sigma^2=0$), the condition \eqref{eq:align-cov} reduces to the strict alignment condition on $(\bSigma_0,\bbeta)$; see \Cref{prop:thm:stationary-point-no-shift-strict}.

While \Cref{thm:stationary-point-cov-shift} restricts $\bSigma = \bI$ to simplify the condition when $\phi>1$, we provide the condition with general $\bSigma$ in \Cref{eq:sig-shift-gen-cond}.
However, taking $\bSigma = \bI$ suffices to highlight this rather surprising sign-reversal phenomenon.

\begin{restatable}
    [Optimal regularization under regression shift]
    {theorem}
    {ThmStationaryPointSigShift}
    \label{thm:stationary-point-sig-shift}
    Assume the setup of \Cref{prop:ood-risk-asymptotics} with $\bSigma_0=\bSigma$, $\bbeta_0\neq\bbeta$.
    \begin{enumerate}[leftmargin=4mm,itemsep=0pt,font=\normalfont]
        \item
        \label{thm:stationary-point-sig-shift-item-underparameterized}
        \emph{(Underparameterized)} When $\phi<1$, if $\sigma^2=o(1)$ and for all $\mu\geq 0$, the following general alignment holds:
        \begin{align}
            \label{eq:sig_shift_cond}
            \otr[\bB_0\bSigma^2(\bSigma + \mu\bI)^{-2}] > \otr[\bB\bSigma^2(\bSigma + \mu\bI)^{-2}],
        \end{align}
        where $\bB = \bbeta \bbeta^\top$ and $\bB_0=\bbeta_0\bbeta^{\top}$, then we have $\lambda^* < 0$.        
        \item 
        \label{thm:stationary-point-sig-shift-item-overparameterized}
        \emph{(Overparameterized)} When $\phi>1$, if the general alignment conditions \eqref{eq:align-cond} and \eqref{eq:sig_shift_cond} hold, then we have $\lambda^*< 0$.
    \end{enumerate}
\end{restatable}

Similarly to Part \eqref{thm:stationary-point-cov-shift-item-overparameterized} of \Cref{thm:stationary-point-cov-shift}, Part \eqref{thm:stationary-point-sig-shift-item-underparameterized} of \Cref{thm:stationary-point-sig-shift} is rather surprising.
As shown in \Cref{tab:opt-reg}, $\lambda^*$ is always positive for the in-distribution setting when $\phi < 1$.
However, \Cref{thm:stationary-point-sig-shift} suggests that $\lambda^*$ can be negative, even for isotropic signals, when there is some alignment or misalignment between $\bbeta^\top \bSigma$ and $(\bbeta - \bbeta_0)^\top \bSigma$.
In fact, when $\bSigma=\bSigma_0=\bI$ and $\langle \bbeta, \bbeta_0\rangle \geq \|\bbeta\|_2^2$, the alignment condition $\bbeta^\top \bSigma^2(\bSigma + \mu\bI)^{-2}  ({\bbeta}_0-\bbeta) \geq 0$ in \Cref{thm:stationary-point-sig-shift} always holds for all $\mu>0$.
It is worth noting that we assume $\sigma^2 = o(1)$ in \Cref{thm:stationary-point-sig-shift} for simplicity, but a more general balance condition that holds for any $\sigma^2 > 0$ is provided in \eqref{eq:sig-shift-gen-cond}.

\begin{figure}[!t]
    \centering
    \includegraphics[width=0.99\textwidth]{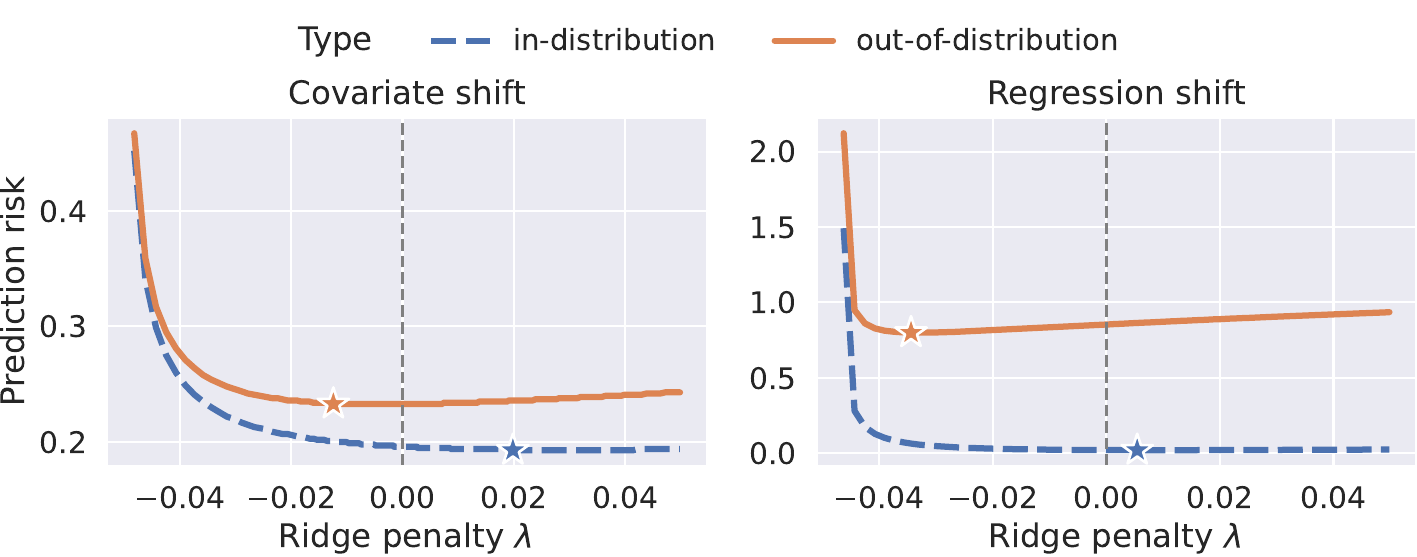}
    \caption{
    \textbf{Covariate and regression shift can lead to negative optimal regularization in both underparameterized and overparameterized regimes.}
    The plot shows the in-distribution and OOD risks against $\lambda$ in the high \SNR setting ($\sigma^2=0.01$ and $\sigma_0^2=0$).
    The left panel shows the overparameterized regime ($\phi=1.5$) where the optimal ridge penalty $\lambda^*$ is negative under covariate shift, when $\bSigma=\bI$, $\bSigma_0=\bSigma_{\mathrm{ar1}}$, and $\bbeta=\bbeta_0=\frac{1}{2}(\bw_{(1)} + \bw_{(p)})$.
    The right panel shows the underparameterized regime ($\phi=0.5$) where the optimal ridge penalty $\lambda^*$ is negative under regression shift, when $\bSigma=\bSigma_0=\bSigma_{\mathrm{ar1}}$, $\bbeta=\frac{1}{2}(\bw_{(1)} + \bw_{(p)})$, and $\bbeta_0=2\bbeta$.    
    }
    \label{fig:negative_optimal_our_condition_ood}
\end{figure}

The numerical illustrations in \Cref{fig:negative_optimal_our_condition_ood} demonstrate the results of \Cref{thm:stationary-point-cov-shift,thm:stationary-point-sig-shift}.
As shown, while the optimal ridge penalties $\lambda^*$ for the in-distribution prediction risks are positive, the OOD prediction risk can be negative and approach its lower limit.
Similar observations also occur in real-world MNIST datasets (see \Cref{tab:MNIST} and \Cref{sec:real_data_illustration} for the experimental details).
This phenomenon arises due to distribution shift, which effectively aligns $\bbeta_0$ and $\bSigma_0$.
In some cases, it may provide a possible explanation for the success of interpolators in practice, as the optimal ridge penalty $\lambda^*$ can become negative under distribution shift, e.g., with random features regression \citep{tripuraneni2021covariate}.

Intuitively, negative regularization may be optimal in certain cases due to the implicit bias of the overparameterized ridge estimator, even when $\lambda = 0$. 
When the signal energy is sufficiently high, it can be beneficial to ``subtract'' some of this bias at the expense of increased variance. 
Negative regularization effectively reduces this bias and can, therefore, be the optimal choice. 
In a broader context, when there is implicit regularization, such as self-regularization resulting from the data structure, and we desire the overall regularization to be smaller than this inherent amount, negative ``external'' regularization can help counterbalance it.

\begin{table}[!t]
    \caption{
    \textbf{Optimal ridge penalty for OOD risks on MNIST gradually becomes more negative with increasing distribution shift.}
    We gradually shift the test distribution by excluding samples with specific labels.
    For more details, please refer to \Cref{sec:real_data_illustration}.
    }
    \label{tab:MNIST}
    \centering
    \begin{tabular}{ccccccc}
        \toprule
         Excl. labels: & $\varnothing$ & $\{4\}$ & $\{3,4\}$ & $\{2,3,4\}$ & $\{1,2,3,4\}$ \\\cmidrule(lr){1-6}
         $\lambda^*$ &1.03& 0.48&0.00&-0.48 & 1.44\\
         \bottomrule
    \end{tabular}
\end{table}

\section{Properties of Optimal Risk}
\label{sec:optimal_risk_monotonicity}

Under the general covariance structure $\bSigma$, Theorem 6 in \citet{patil2023generalized} shows that optimal ridge regression exhibits a monotonic risk profile in the in-distribution setting ($\bSigma_0=\bSigma$) and effectively avoids the phenomena of double and multiple descents. 
This observation motivates a broader investigation into the monotonic behavior of regularization optimization in the out-of-distribution setting. 
In this section, we investigate the monotonicity of the optimal OOD risk and converse for the suboptimal OOD risk.

\subsection{Optimal Risk Monotonicity}

To begin with, we examine the case of isotropic signals under covariate shift.
A direct consequence of \Cref{thm:stationary-point-cov-shift-iso} is the monotonicity property of the risk in this scenario.

\begin{restatable}
    [Optimal risk under isotropic signals]
    {proposition}
    {LemOptimalRiskIsoSignal}
    \label{lem:optimal-risk-iso-signal}
    When $\bSigma_0\neq\bSigma$ and $\bbeta=\bbeta_0$, assuming isotropic signals $\EE[\bbeta\bbeta^{\top}]= (\alpha^2 / p) \bI$ the optimal risk obtained at $\lambda^*(\phi)=\phi/\SNR$ is given by:
    \begin{equation}
        \EE_{\bbeta}[\sR(\lambda^*,\phi)] = \alpha^2 \mu^*\otr[\bSigma_0(\bSigma+\mu^*\bI)^{-1}] + \sigma_0^2,
        \label{eq:optimal-risk-iso-signal}
    \end{equation}
    where $\mu^*=\mu(\lambda^*,\phi)$.
    Furthermore, the left side of \eqref{eq:optimal-risk-iso-signal} is strictly increasing in $\phi$ if $\SNR \in (0, \infty)$ and $\sigma_0^2$ are fixed and strictly increasing in $\SNR$ if $\phi$, $\sigma^2$, and $\sigma_0^2$ are fixed.
\end{restatable}

\Cref{lem:optimal-risk-iso-signal} shows that the optimal OOD risk is a monotonic function of $\phi$ and \SNR. 
This is intuitive because one would expect that having more data (smaller $\phi$) or larger \SNR would result in a lower prediction risk. 
In contrast, the ridge or ridgeless predictor computed on the full data does not exhibit this property \citep[Figure 2]{hastie2022surprises}. 
It is worth noting that the optimal penalty $\lambda^*$ is also monotonically increasing in the data aspect ratio $\phi$ when \SNR is kept fixed. 

However, under anisotropic signals, the optimal regularization penalty generally depends on the specific OOD risk being considered. 
In such cases, it is difficult to obtain analytical formulas for the optimal ridge penalty and the optimal risk.
Nonetheless, we can still show that the optimal ODD risk monotonically increases in $\phi$ and $\SNR$.
We formalize this in the following result, which generalizes the aforementioned in-distribution result in \citet[Theorem 6]{patil2023generalized} to the OOD setting, allowing risk optimization over the possible range of negative regularization (with the lower limit as given in \Cref{def:lower_bound_negative_regularization}).

\begin{restatable}
    [Monotonicity of optimally tuned OOD risk]
    {theorem}
    {ThmMonotonicty}
    \label{thm:monotonicty}
For $\lambda\geq \lambda_{\min}(\phi)$ where $\lambda_{\min}(\phi)$ is as in \eqref{eq:lam_min}, for all $\epsilon> 0$ small enough, the risk of optimal ridge predictor satisfies:
\begin{equation}
    \min_{\lambda\geq \lambda_{\min}(\phi) + \epsilon} R(\hbeta^\lambda) ~~ \asympequi \min_{\lambda\geq \lambda_{\min}(\phi)}\sR(\lambda, \phi), \label{eq:R-lam-phi-opt}
\end{equation}
and right side of \eqref{eq:R-lam-phi-opt} is monotonically increasing in $\phi$ if $\SNR$ and $\sigma_0^2$ are fixed. 
In addition, when $\bbeta=\bbeta_0$ it is monotonically increasing in $\SNR$ if $\phi$, $\sigma^2$, and $\sigma_0^2$ are fixed.
\end{restatable}

\begin{figure}[!t]
    \centering
    \includegraphics[width=0.99\textwidth]{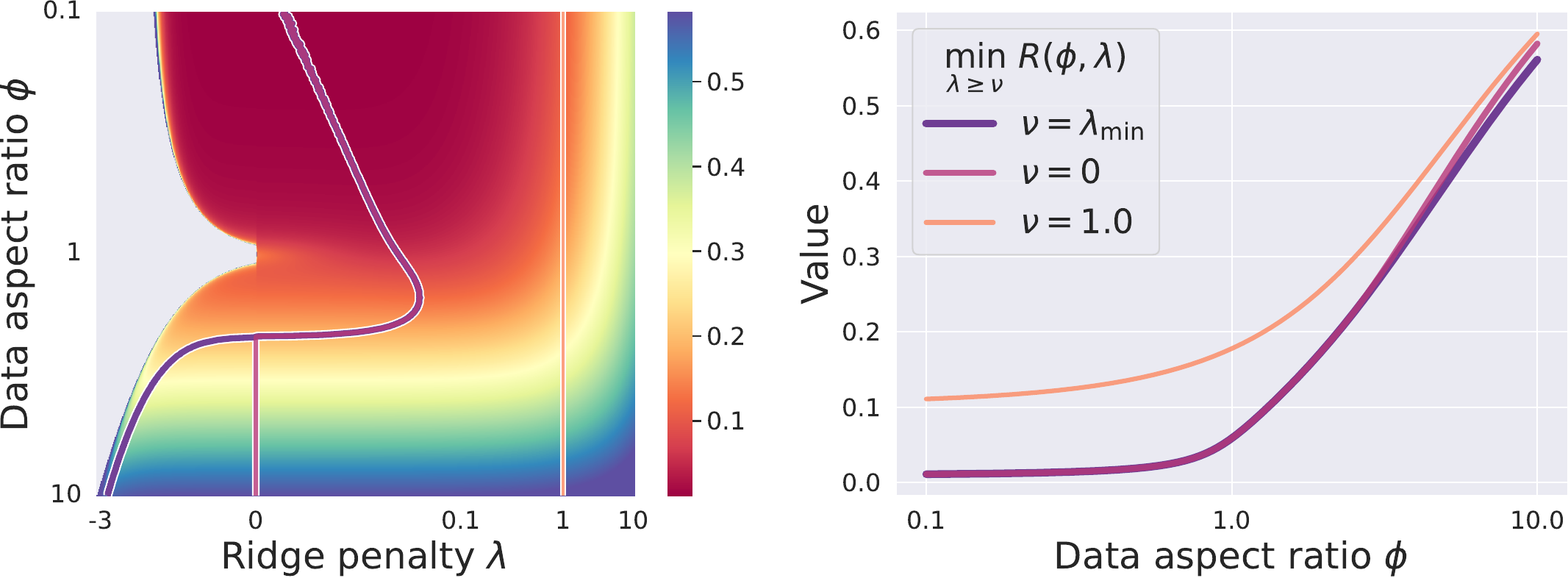}
    \caption{
    \textbf{Ridge regression optimized over $\lambda\geq\nu$ for different thresholds $\nu$ has monotonic risk profile.}
    We showcase the prediction risk of optimal ridge regression under the same data model as in \Cref{fig:negative_optimal_our_condition}, with $\sigma^2=0.01$.
    The left panel shows the heatmap of the risks $\sR(\lambda, \phi)$ of ridge regression for different ridge penalties $\lambda$ and data aspect ratios $\phi$.
    The lines indicate the optimized ridge risks $\min_{\lambda\geq \nu}\sR(\lambda, \phi)$ at different thresholds $\nu$.
    The right panel shows the optimized risk $\min_{\lambda\geq \nu}\sR(\lambda, \phi)$ as a function of $\phi$.
    }
    \label{fig:optimal-risk}
\end{figure}

We find the monotonicity of the optimal risk in $\phi$ remarkable because it even holds under \emph{arbitrary} covariate and regression shift!
The monotonicity in $\SNR$ under regression shift can be similarly analyzed but requires fixing more parameters.
When considering the in-distribution prediction risk, \citet[Theorem 6]{patil2023generalized} demonstrate that the optimally tuned ridge over $\lambda\in(0, \infty)$ exhibits a monotonic risk profile in the data aspect ratio $\phi$.
It is somewhat surprising that optimizing over both $(0, \infty)$ and $(\lambda_{\min}(\phi), \infty)$ yields monotonic behavior, as numerically verified in \Cref{fig:optimal-risk}.

We also illustrate the optimal risk monotonicity behavior on MNIST in \Cref{fig:MNIST_risk_monotoncity}.
Further, the minimum risk over $(\lambda_{\min}(\phi), \infty)$ can be significantly lower than the minimum risk over $(0, \infty)$, particularly in the overparameterized regime.
Although most software packages only consider positive ridge regularization for ridge tuning, \Cref{fig:optimal-risk} suggests that allowing negative regularization can lead to significant improvements.

\begin{figure*}[!ht]
    \centering
    \includegraphics[width=0.65\textwidth]{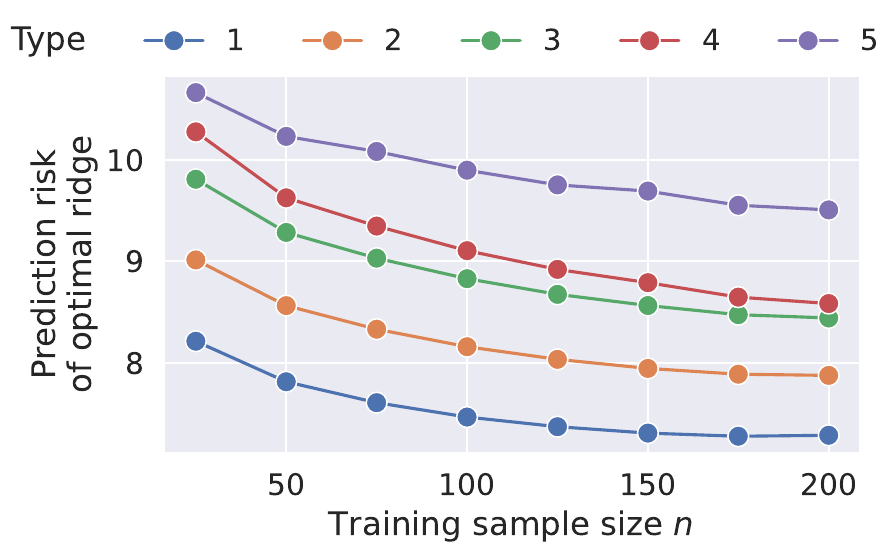}
    \caption{
        \textbf{Effect of distribution shift on the risk monotonicity behavior of optimal ridge on MNIST.}
        The figure illustrates the risk profile (against the training sample size) of optimal ridge regression on the MNIST dataset when subjected to different types of distribution shifts. 
        We follow the same setup as for \Cref{tab:MNIST} (see \Cref{sec:real_data_illustration} for more details) and vary the number of training sample size $n$ from 25 to 200, and inspect the OOD prediction risk of the optimal ridge predictor.
        Different colors represent different types of shift from less severe (\texttt{Type} 1) to more severe (\texttt{Type} 5).
        The y-axis represents the out-of-distribution prediction risk for the task of accurately predicting the digit value for unseen images. 
        The figure shows a clear pattern where the optimal ridge exhibits a monotonically decreasing risk in the training sample size $n$.
        }
    \label{fig:MNIST_risk_monotoncity}
\end{figure*}

Finally, we consider the converse question: is optimal regularization \emph{necessary} for achieving monotonic risk? 
When considering the excess in-distribution prediction risk under isotropic design (i.e., $\bSigma=\bI$), the following theorem demonstrates that the prediction risk of ridge regression is generally non-monotonic in the data aspect ratio $\phi \in(0,\infty)$.

\begin{restatable}
    [Non-monotonicity of suboptimally tuned risk]
    {theorem}
    {ThmNonmonotonicity}
    \label{thm:nonmonotonicity}
    When $(\bSigma_0,\bbeta_0)=(\bSigma,\bbeta)$ and $\bSigma=\bI$, the risk component equivalents defined in \eqref{eq:risk-det-equiv} have the following properties:
    \begin{enumerate}[leftmargin=4mm,itemsep=0pt,font=\normalfont]
        \item \emph{(Bias component)} For all $\lambda> 0$, $\sB(\lambda, \phi)$ is strictly increasing over $\phi\in(0,\lambda+1)$ and strictly decreasing over $\phi\in(\lambda+1,\infty)$.
        
        \item \emph{(Variance component)} For all $\lambda> 0$, $\sV(\lambda, \phi)$ is strictly increasing over $\phi\in(0,\infty)$.

        \item \emph{(Risk)} When $\|\bbeta\|_2^2> 0$, for all $\lambda>0$ and $\epsilon>0$, there exist $\sigma^2,\phi\in(0,\infty)$, such that $\partial \sR(\lambda, \phi)/\partial \phi \leq -\epsilon$, i.e.,
        \begin{equation}
            \max_{\sigma^2,\phi\in(0,\infty)} \min_{\lambda\geq \lambda_{\min}(\phi)} {\partial \sR(\lambda, \phi)}/{\partial \phi} \leq -\epsilon.
            \label{eq:nonmonotonicity-risk}
        \end{equation}
    \end{enumerate}
\end{restatable}

When $\sigma^2$ is small, \Cref{thm:nonmonotonicity} implies that the risk at any fixed $\lambda$ is non-monotonic, even when $\bSigma = \bI$.
It also explains the lack of risk monotonicity in ridgeless regression.
This result extends known results on non-monotonic behavior of variance of ridgeless regression \citep{yang2020rethinking}.

\subsection{Connection to Subsampling and Ensembling}
\label{sec:proof_through_subsample_ridge_equivalences}

We now briefly discuss the connection between subsampling and ridge regularization \citep{patil2023generalized}, used to prove the OOD risk monotonicity results in \Cref{sec:optimal_risk_monotonicity}.
To incorporate negative regularization and OOD risks, we extend these equivalences using tools from \citet{lejeune2022asymptotics}.

Before presenting the extended equivalence results, we introduce several quantities related to the subsampled ridge ensembles.
For an index set $\cI \subseteq [n]$ of size $k$, let $\bL_{\cI} \in \RR^{n \times n}$ be a diagonal matrix with $i$-th diagonal $1$ if $i \in \cI$ and $0$ otherwise. 
The feature matrix and the response vector associated with a subsampled dataset $\{(\bx_i, y_i): i \in \cI \}$ are $\bL_{\cI} \bX$ and $\bL_{\cI} \by$, respectively.
Given a ridge penalty $\lambda$, let $\hbeta^{\lambda}_{k}(\cI)$ denote the ridge estimator fitted on the subsample $(\bL_{\cI} \bX, \bL_{\cI} \by)$, consisting of $k$ samples.

When we aggregate the estimators fitted on all subsampled datasets of size $k$, we obtain the so-called \emph{full-ensemble} estimators 
$\hbeta^{\lambda}_{k,\infty},
$
which is almost surely $\EE[\hbeta^{\lambda}_{k}(\cI)\mid \bX,\by]$, if we draw $\cI$ independently from the set of index sets of size $k$.
As $k,n,p\rightarrow\infty$ such that $p/n\rightarrow\phi\in(0,\infty)$ and $p/k\rightarrow\psi\in[\phi,\infty]$, \Cref{thm:monotonicty-gen} implies the OOD risk equivalence:
$
R(\hbeta^{\lambda}_{k,\infty}) \asympequi \sR(\lambda, \phi,\psi)
$ as in \eqref{eq:R_p}.
When $\psi=\phi$, the equivalent $\sR(\lambda, \phi,\phi)$ reduces to $\sR(\lambda, \phi)$ defined in \eqref{eq:risk-det-equiv} and analyzed in previous sections.
We are now ready to present our main result in this section.

\begin{restatable}
    [Optimal ensemble versus ridge regression with negative regularization]
    {theorem}
    {ThmNegativeLam}
    \label{thm:negative-lam}
    Let $\sR^{**} := \min_{\psi \ge \phi, \lambda \ge \lambda_{\min}(\phi)}\sR(\lambda, \phi,\psi)$.
    Then the following statements hold:
    \begin{enumerate}[leftmargin=4mm,itemsep=0pt,font=\normalfont]
        \item \emph{(Underparameterized)} When $\phi < 1$ and $\bbeta_0=\bbeta$, $\lambda^*\geq 0$, 
        {\begin{align*}
            \sR^{**}
            =
            \min_{\lambda \ge 0} \sR(\lambda, \phi,\phi)
            =
            \min_{\psi \ge \phi}
            \sR(0;\phi,\psi).
        \end{align*}}
        
        \item \emph{(Overparameterized)} When $\phi \geq 1$, $\lambda^*\geq \lambda_{\min}(\phi)$,
        {\[
            \sR^{**}
            =
            \min_{\lambda \ge \lambda_{\min}(\phi)} \sR(\lambda, \phi,\phi)
            =
            \min_{\psi \ge \phi}
            \sR(\lambda_{\min}(\phi);\phi,\psi).
        \]}
    \end{enumerate}
\end{restatable}

\Cref{thm:negative-lam} establishes the OOD risk equivalences between ridge predictors and full-ensemble ridgeless predictors.
It shows that in the underparameterized regime, ridgeless ensembles are sufficient to achieve optimal in-distribution risk, which is also supported by \citet{du2023gcv} when considering optimization over $\lambda\geq 0$.
However, in the overparameterized regime, ridgeless ensembles alone are not enough to achieve optimal risk over $\lambda\geq \lambda_{\min}(\phi)$.
Explicit negative regularization is required to obtain the optimal risk.
In other words, the implicit regularization provided by subsampling is always positive, and under certain data geometries, using $\lambda<0$ can improve predictive performance.

\begin{figure}[!t]
    \centering
    \includegraphics[width=0.99\textwidth]{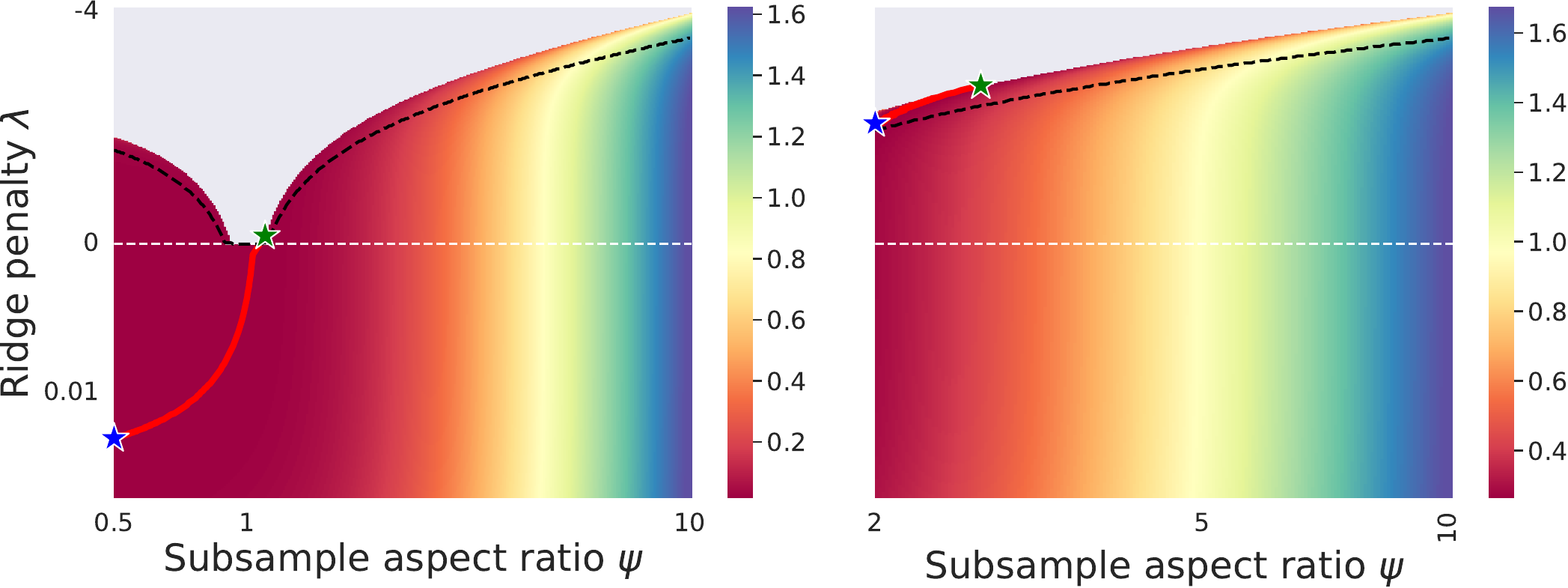}
    \caption{
    \textbf{Negative regularization can help achieve optimal risk in both underparameterized and overparameterized regimes.}
    The heatmap illustrates the prediction risks for ridge regression as a function of the ridge penalty $\lambda$ and subsample aspect ratio $\psi$ in the full ensemble. 
    We use the same data model as \Cref{fig:optimal-risk} with $\sigma^2=0.01$. 
    The left and right panels show the underparameterized ($\phi=0.5$) and overparameterized regimes ($\phi=2$), respectively. 
    The red paths represent the optimal risks, while the blue and green stars indicate the optimal ridge predictor and the optimal full-ensemble ridge with the largest subsample aspect ratio.
    }
    \label{fig:advantage_ensembling}
\end{figure}

The proof of \Cref{thm:negative-lam} establishes the risk equivalence between the ridge predictor and the full-ensemble ridgeless predictor.
As shown in \Cref{fig:advantage_ensembling}, the risk profile of the full-ensemble ridge predictors demonstrates that negative regularization can help achieve optimal risk, particularly in the overparameterized regime when the subsample aspect ratio $\psi$ is greater than $1$.

\section{Discussion}
\label{sec:discussion}

This paper investigates the optimal regularization and the optimal risk of ridge regression in the OOD setting.
The analysis shows the differences between the in-distribution and OOD settings. 
However, in both cases, the optimal risk is monotonic in the data aspect ratio. 
The main takeaway is that negative regularization can improve predictive performance in certain data geometries, even more so than the in-distribution setting.

We close the paper by mentioning two future directions.
The scope of this paper is limited to standard ridge regression.
Under proportional asymptotics, similar conclusions are expected to hold for kernel ridge regression, as the gram matrix linearizes in the sense of asymptotic equivalence \citep{liang2020just, bartlett_montanari_rakhlin_2021, sahraee2022kernel,misiakiewicz2023six}.
Beyond ridge variants, it is of interest to perform a similar analysis for the lasso.
Preliminary investigations in the literature suggest that, similar to optimal ridge regression, the optimal lasso also exhibits monotonic risk in the overparameterization ratio \citep{patil2022mitigating}. 
For a comparative illustration, see \Cref{fig:ridge-monotonicity,fig:lasso-monotonicity}.
It is also of interest to understand the monotonicity properties of the optimal risk for other convex penalized M-estimators in general \citep{thrampoulidis2018precise}.
Empirical investigations indicate that these estimators also exhibit monotonic risk in the data aspect ratio under optimal regularization \citep{patil2023generalized}. 
Making these observations rigorous is a promising direction on which we hope to report in the near future.

\section*{Acknowledgments}
We thank Edgar Dobriban, Avi Feller, Qiyang Han, Arun Kuchibhotla, Dmitry Kobak, Daniel LeJeune, Jaouad Mourtada, Alessandro Rinaldo, Alexander Wei, and Yuting Wei for helpful conversations surrounding this work.
We thank the anonymous reviewers for their valuable feedback and suggestions that have improved the paper (in particular, the addition of \Cref{sec:introduction-details}).
We also thank the computing support provided by the ACCESS allocation MTH230020 for some of the experiments performed on the Bridges2 system at the Pittsburgh Supercomputing Center.
PP and RJT were supported by ONR grant N00014-20-1-2787 during the course of this work.


\counterwithin{theorem}{section}
\counterwithin{remark}{section}
\counterwithin{proposition}{section}
\counterwithin{lemma}{section}
\counterwithin{fact}{section}
\renewcommand{\thefigure}{\thesection\arabic{figure}}

\newpage
\appendix

\newgeometry{left=0.5in,top=0.25in,right=0.5in,bottom=0.25in,head=.1in,foot=0.1in}

\begin{center}
\Large
{\bf \framebox{Supplement}}
\end{center}

\bigskip

This supplement serves as a companion to the paper titled ``\titletext''.
In the first section, we provide an outline of the supplement in \Cref{tab:outline-supplement}.
We also include a summary of the general and specific notation used in the main paper and the supplement in \Cref{tab:notation_general,tab:notation_specific}, respectively.
In addition, in \Cref{tab:ref}, we provide an explanation of and pointers to the abbreviations used in the references mentioned in \Cref{tab:opt_reg}.

\section{Organization and Notation}
\label{sec:organization_notation_supp}

\subsection{Organization}
\label{sec:organization_supp}

\begin{table}[!ht]
\centering
\caption{Outline of the supplement.}
\label{tab:outline-supplement}
\begin{tabularx}{\textwidth}{L{1.75cm}L{2cm}L{20cm}}
\toprule
\textbf{Section} & \textbf{Subsection} & \textbf{Purpose} \\
\midrule \addlinespace[1ex]
\multicolumn{2}{c}{Overview of supplement} \\ \arrayrulecolor{black!20} \cmidrule(lr){1-2} \addlinespace[0.5ex]
\multirow{3}{*}{\Cref{sec:organization_notation_supp}} & \Cref{sec:organization_supp} & Organization \\
& \Cref{sec:general_notation_supp} & General notation \\
& \Cref{sec:specific_notation_supp} & Specific notation \\
& \Cref{sec:reference_key} & Reference key \\
\arrayrulecolor{black!50} \midrule \addlinespace[1ex]
\multicolumn{2}{c}{Further details in \Cref{sec:introduction}} \\ \arrayrulecolor{black!20} \cmidrule(lr){1-2} \addlinespace[0.5ex]
\Cref{sec:introduction-details} & \Cref{sec:background-ridge-regression-formulation} & Background details on ridge regression with negative regularization \\ 
\arrayrulecolor{black!50}
\midrule
\multicolumn{2}{c}{Proofs in \Cref{sec:preliminaries}} \\ \arrayrulecolor{black!20} \cmidrule(lr){1-2} \addlinespace[0.5ex]
\Cref{sec:preliminaries_appendix} & \Cref{sec:prop:ood-risk-asymptotics-proof} & Proof of \Cref{prop:ood-risk-asymptotics} (out-of-distribution risk asymptotics) \\ 
\arrayrulecolor{black!50}
\midrule 
\multicolumn{2}{c}{Proofs in \Cref{sec:optimal_regularization_signs}} \\ \arrayrulecolor{black!20} \cmidrule(lr){1-2}
\addlinespace[0.5ex]
\multirow{5}{*}{\Cref{sec:proofs-sec:optimal_regularization_signs}}
& \Cref{sec:thm:stationary-point-no-shift-proof}
& Proof of \Cref{thm:stationary-point-no-shift} (sign of optimal regularization for in-distribution) 
\\
& \Cref{sec:alignment-special-cases} & Special cases of the general alignment condition in \Cref{thm:stationary-point-no-shift} \\
& \Cref{sec:thm:stationary-point-cov-shift-iso-proof} & Proof of \Cref{thm:stationary-point-cov-shift-iso} (optimal regularization with isotropic signals) \\
& \Cref{sec:thm:stationary-point-cov-shift-proof} & Proof of \Cref{thm:stationary-point-cov-shift} (sign of optimal regularization with covariate shift) \\
& \Cref{sec:thm:stationary-point-sig-shift-proof} & Proof of \Cref{thm:stationary-point-sig-shift} (sign of optimal regularization with regression shift) \\
& \Cref{sec:proofs-sec:optimal_regularization_signs-helper-lemmas} & Helper lemmas (derivatives of components of risk deterministic approximation) \\
\arrayrulecolor{black!50}
\midrule 
\multicolumn{2}{c}{Proofs in \Cref{sec:optimal_risk_monotonicity}} \\ \arrayrulecolor{black!20} \cmidrule(lr){1-2}
\addlinespace[0.5ex]
\multirow{5}{*}{\Cref{sec:proofs-sec:optimal_risk_monotonicity}}
& \Cref{sec:lem:optimal-risk-iso-signal-proof}
& 
Proof of \Cref{lem:optimal-risk-iso-signal} (optimal risk under isotropic signals)
\\
& \Cref{sec:thm:monotonicity-proof} & Proof of \Cref{thm:monotonicty} (risk monotonicity with optimal regularization) \\
& \Cref{sec:thm:nonmonotonicity-proof} & Proof of \Cref{thm:nonmonotonicity} (risk non-monotonicity with suboptimal regularization) \\
& {\Cref{sec:thm:negative-lam-proof}}
& 
Proof of \Cref{thm:negative-lam} (optimal subsample versus ridge with negative regularization)
\\
& \Cref{sec:thm:negative-lam-proof-helper-lemmas} & Helper lemmas (optimal risk monotonicities, full ensemble OOD risk equivalences) \\
\arrayrulecolor{black!50}
\midrule 
\multicolumn{2}{c}{Some technical lemmas} \\ \arrayrulecolor{black!20} \cmidrule(lr){1-2}
\addlinespace[0.5ex]
\multirow{3}{*}{\Cref{sec:analytic-properties-fp-sols}} 
& \Cref{subsec:analytic-properties-fp-sols-min-lam} & Properties of minimum limiting negative ridge penalty \\
& \Cref{subsec:analytic-properties-fp-sols-v} & Analytic properties of certain fixed-point solutions under negative regularization \\
& \Cref{subsec:analytic-properties-fp-sols-contour} & Contour of fixed-point solutions under negative regularization \\
\arrayrulecolor{black!50}
\midrule 
\multicolumn{2}{c}{Additional experiments} \\ 
\arrayrulecolor{black!20} \cmidrule(lr){1-2}
\addlinespace[0.5ex]
\multirow{3}{*}{\Cref{sec:additional_numerical_illustrations}}
& \Cref{sec:additional-illustrations-sec:optimal_regularization_signs}
& Additional illustrations for \Cref{sec:optimal_regularization_signs} \\
& \Cref{sec:additional_illustrations-sec:discussion} & Additional illustrations for \Cref{sec:discussion} \\
\arrayrulecolor{black}
\bottomrule
\end{tabularx}
\end{table}

\clearpage

\subsection{General Notation}
\label{sec:general_notation_supp}

\begin{table}[!ht]
\centering
\caption{A summary of general notation used in the paper and the supplement.}
\label{tab:notation_general}
\begin{tabularx}{\textwidth}{L{4cm}L{16cm}}
\toprule
\textbf{Notation} & \textbf{Description} \\
\midrule
\multicolumn{1}{c}{Fonts} & \\ \arrayrulecolor{black!20} 
\cmidrule(lr){1-1} \addlinespace[0.5ex]
Non-bold lower case & Denotes vectors (e.g., $a$, $b$, $c$) \\
Non-bold upper case & Denotes matrices (e.g., $\bA$, $\bB$, $\bC$) \\
Calligraphic font & Denotes sets (e.g., $\cA$, $\cB$, $\cC$) \\
Script font & Denotes certain limiting functions (e.g., $\sA$, $\sB$, $\sC$) \\
\arrayrulecolor{black!50}\midrule
\multicolumn{1}{c}{Sets} & \\ \arrayrulecolor{black!20} 
\cmidrule(lr){1-1} \addlinespace[0.5ex]
$\NN$ & Set of natural numbers \\
$(a, b, c)$ & (Ordered) tuple of elements $a$, $b$, $c$ \\
$\{a, b, c\}$ & Set of elements $a$, $b$, $c$ \\
$[n]$ & Set $\{1, \dots, n\}$ for a natural number $n$ \\
$\mathbb{R}$, $\mathbb{R}_{\ge 0}$ & Set of real and non-negative real numbers \\
$\CC$, $\CC^{+}$, $\CC^{-}$ & Set of complex numbers and upper and lower complex half-planes \\
\arrayrulecolor{black!50}\midrule
\multicolumn{1}{c}{Operations} & \\ \arrayrulecolor{black!20} 
\cmidrule(lr){1-1} \addlinespace[0.5ex]
$\tr[\bA]$, $\otr[\bA]$, $\det(\bA)$ & Trace, average trace ($\tr[\bA] / p$), and determinant of a square matrix $\bA \in \RR^{p \times p}$ \\
$\bB^{-1}$ & Inverse of an invertible square matrix $\bB \in \RR^{p \times p}$ \\
$\bC^{\dagger}$ & Moore-Penrose inverse of a general rectangular matrix $\bC \in \RR^{n \times p}$ \\
$\mathrm{rank}(\bC)$, $\mathrm{null}(\bC)$ & Rank and nullity of a general rectangular matrix $\bC \in \RR^{n \times p}$ \\
$\bD^{1/2}$ & Principal square root of a positive semidefinite matrix $\bD$ \\
$f(\bD)$ & Matrix obtained by applying a function $f : \RR \to \RR$ to a positive semidefinite matrix $\bD$ \\
$\bI$, $\bm{1}$, $\bm{0}$ & The identity matrix, the all-ones vector, the all-zeros vector \\
\arrayrulecolor{black!50}\midrule
\multicolumn{1}{c}{Norms} & \\ \arrayrulecolor{black!20} 
\cmidrule(lr){1-1} \addlinespace[0.5ex]
$\| \bu \|_{q}$ & The $\ell_q$ norm of a vector $\bu$ for $q \ge 1$ \\
$\| \bu \|_{\bA}$ & The induced $\ell_2$ norm of a vector $\bu$ with respect to a positive semidefinite matrix $\bA$ \\
$\|f\|_{L_q}$ & The $L_q$ norm of a function $f$ for $q \ge 1$ \\
$\| \bA \|_{\mathrm{op}}$ & Operator (or spectral) norm of a real matrix $\bA$ \\
$\| \bA \|_{\mathrm{tr}}$ & Trace (or nuclear) norm of a real matrix $\bA$ \\
$\| \bA \|_F$ & Frobenius norm of a real matrix $\bA$ \\
\arrayrulecolor{black!50}\midrule
\multicolumn{1}{c}{Orders} & \\ \arrayrulecolor{black!20} 
\cmidrule(lr){1-1} \addlinespace[0.5ex]
$X = \cO_\upsilon(Y)$ & $| X | \le C_\upsilon Y$ for some constant $C_\upsilon$ that may depend on the ambient parameter $\upsilon$  \\
$X \lesssim_\upsilon Y$ & $ X \le C_\upsilon Y$ for some constant $C_\upsilon$ that may depend on the ambient parameter $\upsilon$ \\
$\bu \le \bv$ & Lexicographic ordering for vectors $\bu$ and $\bv$  \\
$\bA \preceq \bB$ & Loewner ordering for symmetric matrices $\bA$ and $\bB$ \\
\arrayrulecolor{black!50}\midrule
\multicolumn{1}{c}{Asymptotics} & \\ \arrayrulecolor{black!20} 
\cmidrule(lr){1-1} \addlinespace[0.5ex]
$\Op$, $\op$ & Probabilistic big-O and little-o notation \\
$\bC \asympequi \bD$ & Asymptotic equivalence of matrices $\bC$ and $\bD$ (see \Cref{sec:preliminaries} for more details) \\
$\dto$, $\pto$, $\asto$ & Denotes convergence in distribution, probability, and almost sure convergence \\
\arrayrulecolor{black}\bottomrule
\end{tabularx}
\end{table}

Some additional conventions used throughout the supplement:
\begin{itemize}[nosep]
    \item 
    If a proof of a statement is separated from the statement, the statement is restated (while keeping the original numbering) along with the proof for convenience.
    \item
    If no subscript is present for the norm $\| \bu \|$ of a vector $\bu$, it is assumed to be the $\ell_2$ norm of $\bu$.
\end{itemize}

\clearpage
\subsection{Specific Notation}
\label{sec:specific_notation_supp}

\begin{table}[!ht]
\centering
\caption{A summary of specific notation used in the paper and the supplement.}
\label{tab:notation_specific}
\begin{tabularx}{\textwidth}{L{3.cm}L{16cm}}
\toprule
\textbf{Notation} & \textbf{Description} \\
\midrule
\multicolumn{1}{c}{In-distribution} & \\ \arrayrulecolor{black!20} 
\cmidrule(lr){1-1} \addlinespace[0.5ex]
$P_{\bx}$, $P_{y \mid \bx}$, $P_{\bx,y}$ & Distribution of the train features supported on $\RR^{p}$, conditional distribution of the train response supported on $\RR$, and the joint distribution  \\
$\bSigma$ & Covariance matrix in $\RR^{p \times p}$ of the train feature distribution \\
$r_{\min}$ ($r_{\max}$) & Minimum (maximum) eigenvalue of the train covariance matrix \\
$\sigma^2,\alpha^2$ & Conditional variance and linearized signal energy of the train response distribution \\
$\SNR$ & In-distribution signal-to-noise ratio $\alpha^2/\sigma^2$ \\
$\bbeta$ & Coefficients of the (population) linear projection of $y$ onto $\bx$ \\
$\{ (\bx_i,y_i) \}_{i=1}^{n}$ & Train samples of size $n$ in $\RR^{p} \times \RR$ sampled i.i.d.\ from $P_{\bx,y}$ \\
$\bX,\by$ & Train feature matrix in $\RR^{n\times p}$ and the corresponding response vector in $\RR^n$ \\
$\hbeta^{\lambda}$ & Ridge estimator fitted on train data $(\bX, \by)$ at regularization level $\lambda$ \\
$\hat{\bbeta}^\lambda_{k}(\cI)$ & Ridge estimator fitted on subsampled data $(\bL_{\cI} \bX, \bL_{\cI} \by)$ at the regularization level $\lambda$ \\
$\hat{\bbeta}^{\lambda}_{k,\infty}$ & Full-ensemble ridge estimator at regularization level $\lambda$ \\
\arrayrulecolor{black!50}\midrule
\multicolumn{1}{c}{Out-of-distribution} & \\ \arrayrulecolor{black!20} 
\cmidrule(lr){1-1} \addlinespace[0.5ex]
$P_{\bx_0}$, $P_{y_0 \mid \bx_0}$, $P_{\bx_0, y_0}$ & Distribution of the test features supported on $\RR^{p}$, conditional distribution of the test response supported on $\RR$, and the joint distribution \\
$(\bx_0, y_0)$ & Test sample in $\RR^{p} \times \RR$ sampled from $P_{\bx_0, y_0}$ independently of the train data $(\bX, \by)$ \\
$\bSigma_0$ & Covariance matrix in $\RR^{p \times p}$ of the test feature distribution \\
$\sigma_0^2$ & Variance of the conditional test response distribution \\
$\bbeta_0$ & Coefficients of the $L_2$ linear projection of $y_0$ onto $\bx_0$: $\EE[\bx_0 \bx_0^{\top}]^{-1} \EE[\bx_0 y_0]$ \\
$R(\hbeta^{\lambda})$ & Squared test prediction risk of ridge estimator at level $\lambda$ \\
\arrayrulecolor{black!50}\midrule
\multicolumn{1}{c}{Risk parameters} & \\ \arrayrulecolor{black!20} 
\cmidrule(lr){1-1} \addlinespace[0.5ex]
$\phi$, $\psi$ & Limiting data and subsample aspect ratios \\
$\lambda^*$ & Optimal ridge penalty \\
$\lambda_{\min}(\phi)$, $v_{\min}$ & Minimum value of ridge regularization allowed at $\phi$ and the associated fixed-point parameter \\
$v_p(\lambda, \phi)$ & A fixed-point parameter \\
$\mu(\lambda, \phi)$ & Induced amount of implicit regularization at ridge regularization $\lambda$ and aspect ratio $\phi$ \\
$\tv_p(\lambda, \phi; \bSigma)$ & Function of a fixed-point parameter \\
$\sR(\lambda, \phi)$, $\sB(\lambda, \phi)$, $\sV(\lambda, \phi)$, $\sE(\lambda, \phi)$ & Deterministic approximation to the squared risk, the bias, the variance, and the regression shift \\
$\kappa^2$ & Inflated out-of-distribution irreducible error \\
$\bSigma_{\bbeta}$ & A certain matrix capturing the alignment of feature covariance and signal vector $\bSigma\bbeta\bbeta^{\top}$\\
$\sR(\lambda, \phi, \psi)$ & Deterministic approximation to the risk of the full-ensemble estimator \\
\arrayrulecolor{black} \bottomrule
\end{tabularx}
\end{table}

\subsection{Reference Abbreviation Key}
\label{sec:reference_key}

\begin{table}[!ht]
    \centering
    \caption{Links to references mentioned in \Cref{tab:opt-reg}.}
    \begin{tabular}{cc}
        \toprule
        Abbreviation & Reference \\\midrule
            \hypertarget{ref:HMRT}{HMRT}
            & \citet*{hastie2022surprises}\\
         \hypertarget{ref:DW}{DW} & \citet*{dobriban_wager_2018} \\
         \hypertarget{ref:RMR}{RMR} &  \citet*{richards_mourtada_rosasco_2021}\\
         \hypertarget{ref:WX}{WX} & \citet*{wu_xu_2020}\\
         \hypertarget{ref:D}{D} & \citet*{dicker2016ridge}\\
         \bottomrule
    \end{tabular}
    \label{tab:ref}
\end{table}

\restoregeometry

\clearpage
\section{Further Details in \Cref{sec:introduction}}
\label{sec:introduction-details}

\subsection{Background Details on Ridge Regression with Negative Regularization}
\label{sec:background-ridge-regression-formulation}

    Recall that the ridge regression estimator $\widehat{\beta}_\lambda \in \mathbb{R}^p$, based on the training data $X,y$, is defined as the solution to the following regularized least squares problem:
    \begin{equation}
        \label{eq:ridge-minimization-problem-primal}
        \mathop{\mathrm{minimize}}_{\beta\in \mathbb{R}^p} \,
        \underbrace{\frac{1}{n} \|y - X\beta\|_2^2}_{\substack{\cL(\beta)}} + \lambda \underbrace{\vphantom{\frac{1}{n} \|y - X\beta\|_2^2} \|\beta\|_2^2}_{\substack{\cR(\beta)}}. \tag{P}
    \end{equation}
    Here, $\lambda$ is a regularization parameter.
    For ease of reference in the following, denote the least squares loss by $\cL$ and $\ell_2$ regularization by $\cR$.
    This is referred to as the primal form of ridge regression, denoted by \eqref{eq:ridge-minimization-problem-primal}.
    One can write a Lagrangian dual problem of \eqref{eq:ridge-minimization-problem-primal} that optimizes weights $\theta \in \mathbb{R}^{n}$ over the data points as:
    \begin{equation}
        \label{eq:ridge-minimization-problem-dual}
        \mathop{\mathrm{minimize}}_{\theta \in \mathbb{R}^n} \;
        \theta^\top (X X^\top / n + \lambda I_n) \theta / 2 - \lambda \theta^\top y / n. \tag{D}
    \end{equation}
    This is referred to as the dual form of ridge regression, denoted by \eqref{eq:ridge-minimization-problem-dual}.
    Let us denote the solution of this problem by $\widehat{\theta}_\lambda$.
    The estimator $\widehat{\beta}$ is linked to the dual weights as $\lambda \widehat{\beta}_\lambda = X^\top \widehat{\theta}_\lambda$.

    Now, there are three cases depending on whether $\lambda$ is positive, zero, or negative.

    \begin{itemize}
        \item When $\lambda > 0$, regardless of whether $n \ge p$ or $p > n$, both problems \eqref{eq:ridge-minimization-problem-primal} and \eqref{eq:ridge-minimization-problem-dual} are strictly convex and lead to the following unique ridge estimator (expressed in primal and dual forms, respectively):
        \begin{equation}
            \label{eq:ridge-primal-dual-solutions}
            \widehat{\beta}_\lambda 
            = (X^\top X / n + \lambda I_p)^{-1} X^\top y / n 
            = X^\top (X X^\top / n + \lambda I_n)^{-1} y / n.
        \end{equation}

        \item When $\lambda = 0$ (more precisely, defined through analytic continuation as $\lambda \to 0^{+}$) and $X^\top X$ is full rank (which is the case when $n \ge p$ and the design $X$ has independent columns), the optimization problem \eqref{eq:ridge-minimization-problem-primal} is still strictly convex and leads to the following unique solution:
        \[
            \widehat{\beta}_0 = (X^\top X / n)^{-1} X^\top y / n.
        \]

        \item When $\lambda = 0$ (again, defined using analytic continuation as $\lambda \to 0^{+}$) and $X X^\top$ is full rank ($p > n$ and the design $X$ has independent rows), the optimization problem \eqref{eq:ridge-minimization-problem-dual} is still strictly convex and leads to the following unique solution:
        \[
            \widehat{\beta}_0 = X^\top (X X^\top / n)^{-1} y / n.
        \]

        \item More generally, when $n \ge p$ and $\lambda \ge - s_{\min}$, where $s_{\min}$ is the smallest eigenvalue of $X^\top X / n$, the optimization problem \eqref{eq:ridge-minimization-problem-primal} is still strictly convex and leads to the primal solution in \eqref{eq:ridge-primal-dual-solutions}.

        \item Similarly, when $p > n$ and $\lambda \ge - s_{\min}$, where $s_{\min}$ is the smallest eigenvalue of $X X^\top / n$, the dual to the optimization problem \eqref{eq:ridge-minimization-problem-dual} is strictly convex and leads to the dual solution in \eqref{eq:ridge-primal-dual-solutions}.
    \end{itemize}
    
    By defining the ridge estimator as:
    \[
    \widehat{\beta}_\lambda 
    = (X^\top X / n + \lambda I_p)^{\dagger} X^\top y / n
    = X^\top (X X^\top / n + \lambda I_n)^{\dagger} y / n,
    \]
    where $A^{\dagger}$ denotes the Moore-Penrose pseudoinverse of the matrix $A$, we can handle all the cases mentioned above simultaneously.
    To gain intuition, we visually compare the ridge solutions with positive and negative regularization in \Cref{fig:negative_ridge_visualize_underparam,fig:negative_ridge_visualize_overparam} for the under- and over-parameterized regimes, respectively.
    See \cite{hua1983generalized,bjorkstrom1999generalized} also for some related discussion.

\clearpage

\newgeometry{left=1in,top=0.75in,right=1in,bottom=0.5in,head=.1in,foot=0.1in}

\subsubsection{Visualizing ridge regression solutions for positive versus negative regularization (underparameterized regime)}

\bigskip

\begin{figure*}[!ht]
    \centering
    \includegraphics[width=0.49\textwidth]{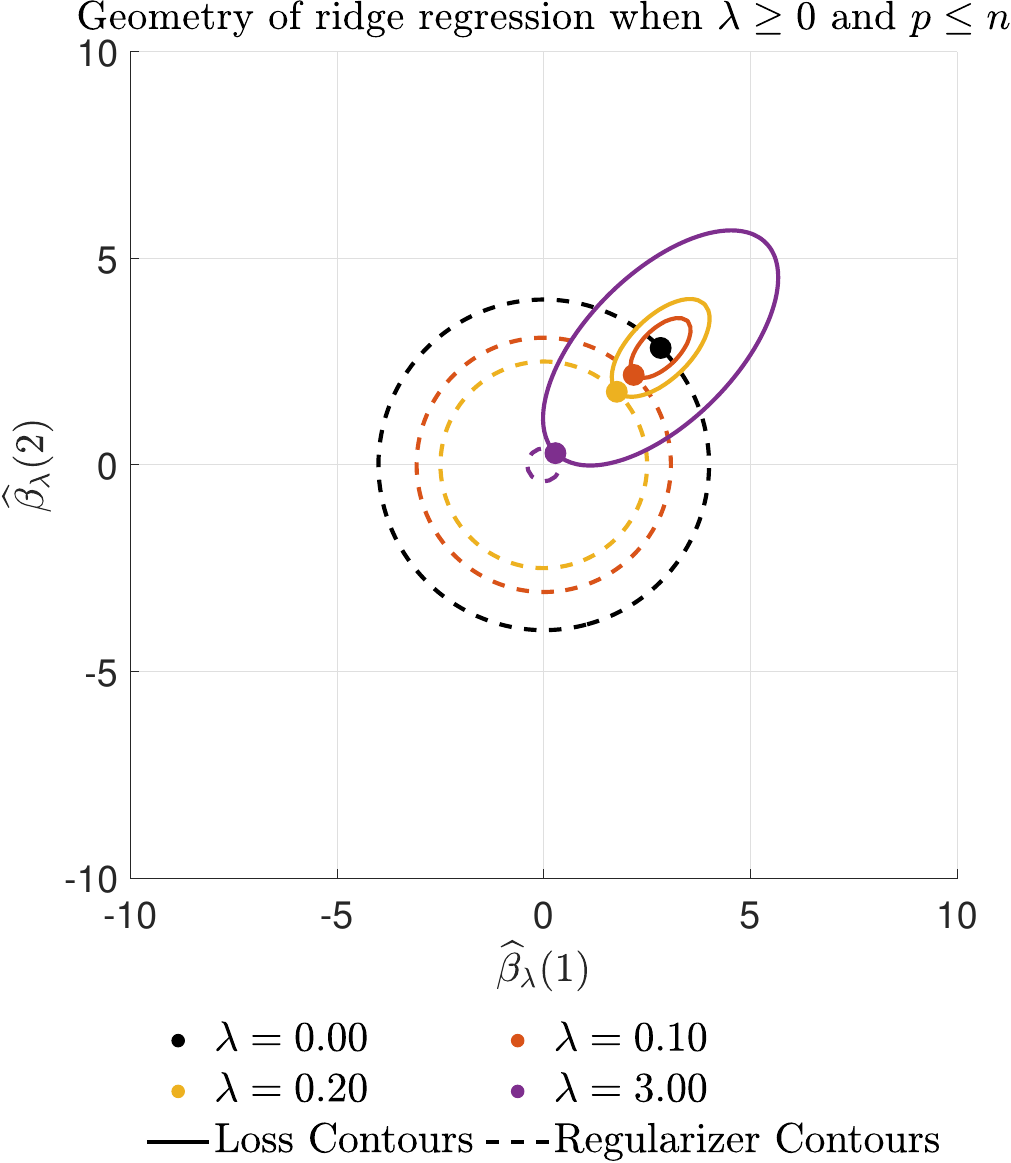}
    \includegraphics[width=0.49\textwidth]{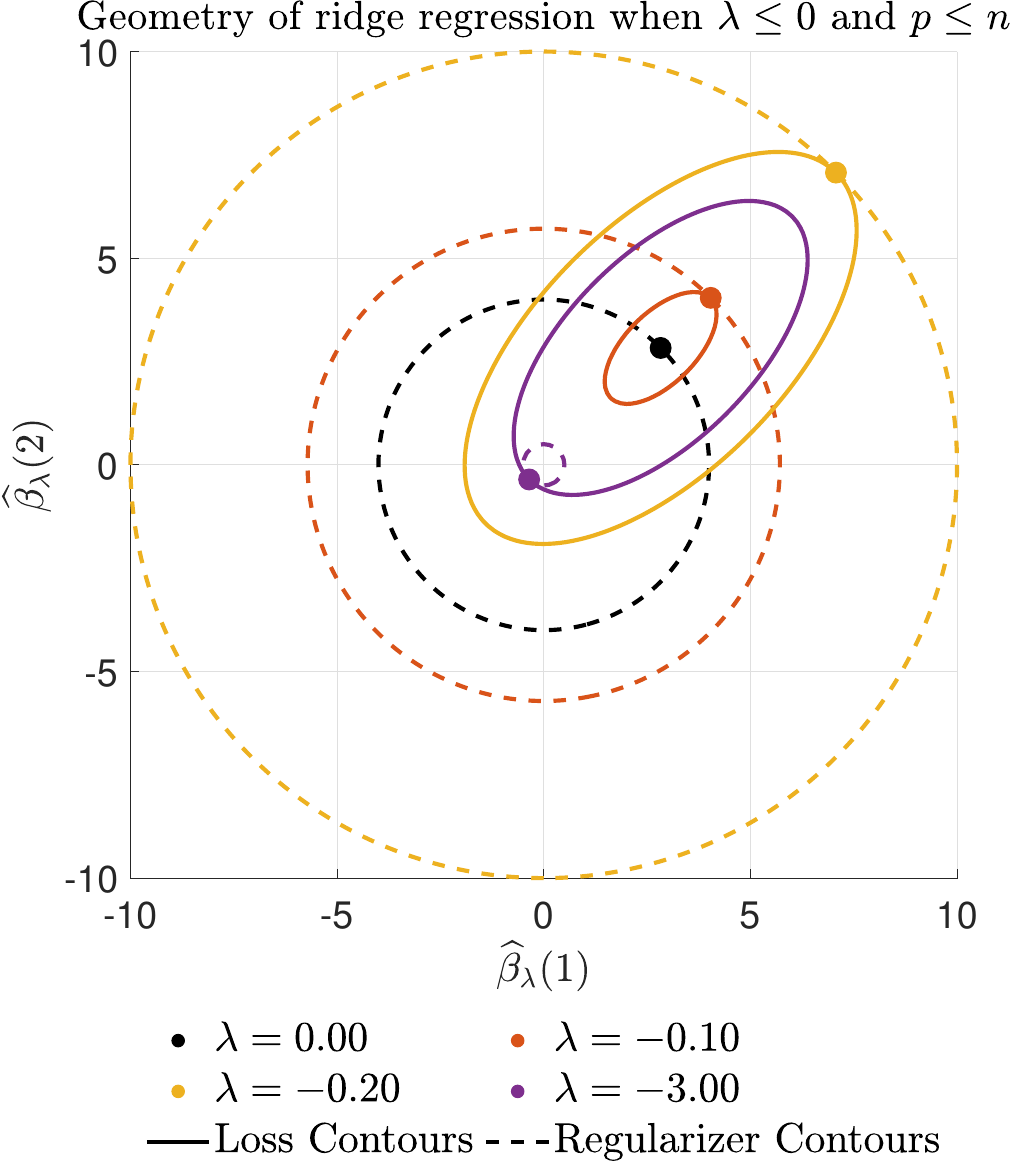}
    \caption{
        \textbf{Comparing the geometry of ridge solutions for positive versus negative regularization levels $\lambda$ in the underparameterized regime when $p \le n$.} 
        We consider a two-dimensional underparameterized problem with a design matrix $X$ having the smallest eigenvalue of $X^\top X / n$ as $s_{\min} = 4/3$ and the largest eigenvalue as $s_{\max} = 1/3$. 
        We plot the contours of the $\ell_2$ square loss $\cL$ and the $\ell_2$ regularizer $\cR$ in the problem \eqref{eq:ridge-minimization-problem-primal} for both positive and negative values of $\lambda$. 
        For positive values of $\lambda \in [0, \infty]$, the loss contours touch the constraint contours from the \emph{outside}, while for negative values of $\lambda \in (-\infty, -s_{\max}) \cup (-s_{\min}, 0)$, they touch from the \emph{inside}.
        To better understand the trend, it helps to think of ``tying'' $+\infty$ to $-\infty$ on the real line, making it a ``projective real line''.
        The values of $\lambda$ are plotted in the following order: $\framebox{{\color{black} 0.0}} \to {\color{myredorange} 0.1} \to {\color{myyelloworange} 0.2} \to {\color{mypurple} 3.0} \to {\color{mypurple} -3.0} \to {\color{myyelloworange} -0.2} \to {\color{myredorange} -0.1} \to \framebox{{\color{black} 0.0}}$.
    }
    \label{fig:negative_ridge_visualize_underparam}
\end{figure*}

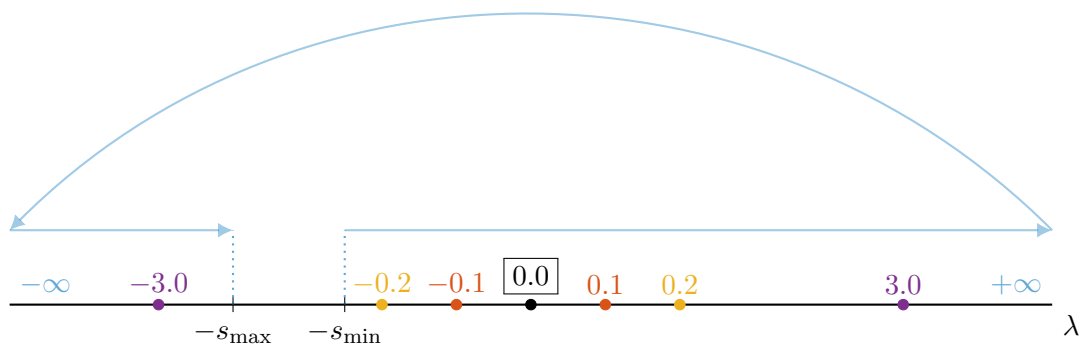
\begin{figure}[!ht]
    \centering
    \begin{tikzpicture}[scale=0.99]

    \draw[thick] (-7,0) -- (7,0) node[anchor=north west] {$\lambda$};
    
    \filldraw[black] (0,0) circle (2pt) node[anchor=south] {\framebox{$0.0$}};
    
    \draw[myblue, thick, ->, opacity=0.375] (7,1) to[out=135,in=45] (-7,1);
    \node at (7,0) [anchor=south east, myblue, opacity=0.625] {$+\infty$};
    \node at (-7,0) [anchor=south west, myblue, opacity=0.625] {$-\infty$};
    
    \draw (-4,0.1) -- (-4,-0.1) node[anchor=north] {$-s_{\max}$};
    \draw (-2.5,0.1) -- (-2.5,-0.1) node[anchor=north] {$-s_{\min}$};
    
    \filldraw[mypurple] (-5,0) circle (2pt) node[anchor=south] {$-3.0$};
    \filldraw[myyelloworange] (-2,0) circle (2pt) node[anchor=south] {$-0.2$};
    \filldraw[myredorange] (-1,0) circle (2pt) node[anchor=south] {$-0.1$};
    \filldraw[myredorange] (1,0) circle (2pt) node[anchor=south] {$0.1$};
    \filldraw[myyelloworange] (2,0) circle (2pt) node[anchor=south] {$0.2$};
    \filldraw[mypurple] (5,0) circle (2pt) node[anchor=south] {$3.0$};
    
    \draw[myblue, thick, ->, opacity=0.375] (-2.5,1) -- (7,1);
    \draw[myblue, thick, ->, opacity=0.375] (-7,1) -- (-4,1);

    \draw[myblue, dotted, thick, opacity=0.625] (-4,0) -- (-4,1);
    \draw[myblue, dotted, thick, opacity=0.625] (-2.5,0) -- (-2.5,1);

    \end{tikzpicture}
    \caption{Illustration of the sequence of ridge regularization levels $\lambda$ used in \Cref{fig:negative_ridge_visualize_underparam} and the corresponding real projective line. (Note: The $\lambda$ values are not necessarily drawn to scale.)}
    \label{fig:negative_ridge_visualize_underparam_sequence}
\end{figure}

\clearpage
\subsubsection{Visualizing ridge regression solutions for positive versus negative regularization (overparameterized regime)}

\bigskip

\begin{figure*}[!ht]
    \centering
    \includegraphics[width=0.49\textwidth]{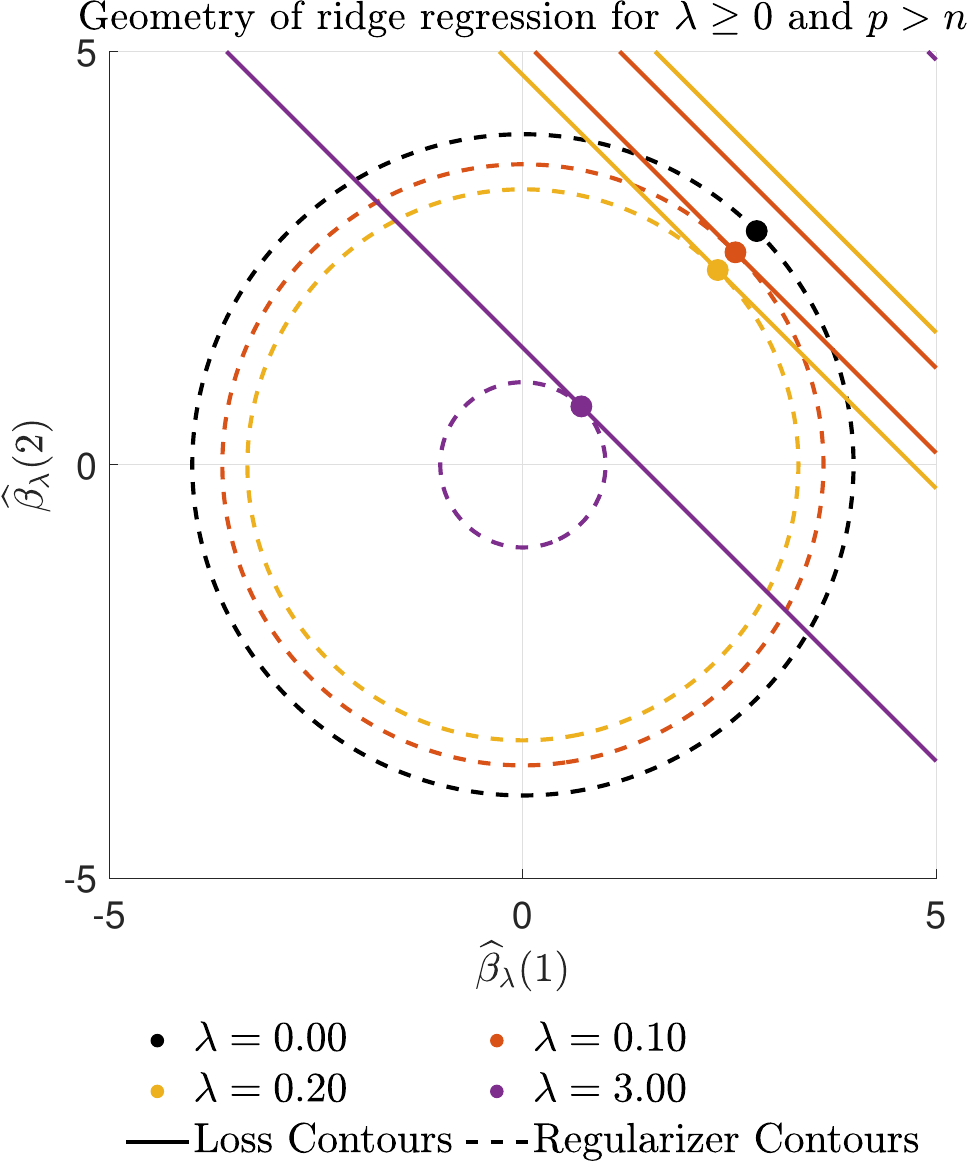}
    \includegraphics[width=0.49\textwidth]{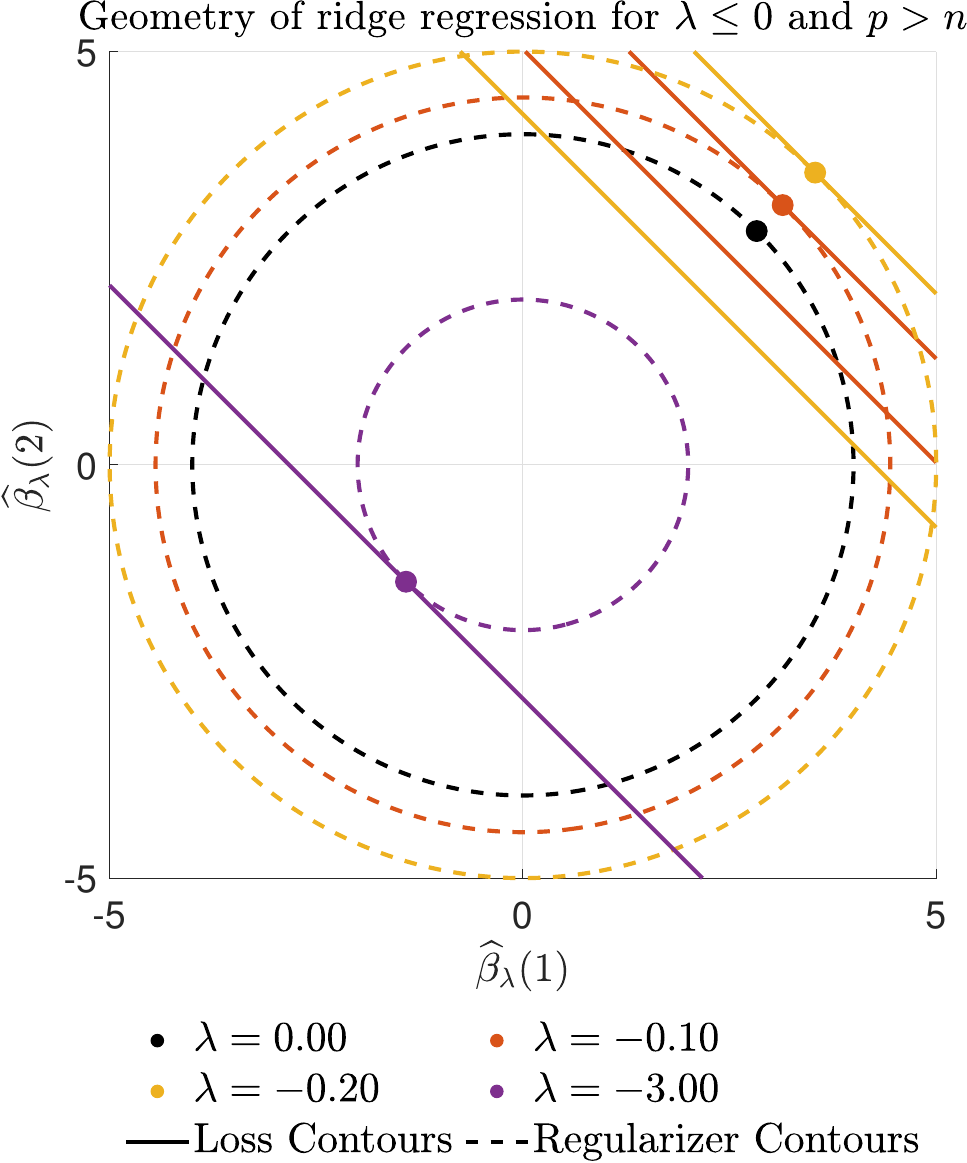}
    \caption{
        \textbf{Comparing the geometry of ridge solutions for positive versus negative regularization levels $\lambda$ in the overparameterized regime when $p > n$.} 
        We consider a two-dimensional overparameterized problem with the design matrix $X$ having the smallest \emph{non-zero} eigenvalue of $X^\top X / n$ equal to $s_{\min}^{+} = 1$. 
        Similarly to \Cref{fig:negative_ridge_visualize_underparam}, we plot contours of the $\ell_2$ squares loss $\cL$ and the $\ell_2$ regularizer $\cR$ (from problem \eqref{eq:ridge-minimization-problem-primal}) for both positive and negative values of $\lambda$. 
        As before, note that for positive values of $\lambda$, the loss contours touch the constraint contours from the \emph{outside}, while for negative values of $\lambda$, they touch from the \emph{inside}.
        To reiterate, it helps to think of ``tying'' $\infty$ to $-\infty$ on the real line to better understand the trend. 
        The values of $\lambda$ are plotted in the following order: $\framebox{{\color{black} 0.0}} \to {\color{myredorange} 0.1} \to {\color{myyelloworange} 0.2} \to {\color{mypurple} 3.0} \to {\color{mypurple} -3.0} \to {\color{myyelloworange} -0.2} \to {\color{myredorange} -0.1} \to \framebox{{\color{black} 0.0}}$.
    }
    \label{fig:negative_ridge_visualize_overparam}
\end{figure*}

\begin{figure}[!ht]
    \centering
    \begin{tikzpicture}[scale=0.99]

    \draw[thick] (-7,0) -- (7,0) node[anchor=north west] {$\lambda$};
    
    \filldraw[black] (0,0) circle (2pt) node[anchor=south] {\framebox{$0.0$}};
    
    \draw[myblue, thick, ->, opacity=0.375] (7,1) to[out=135,in=45] (-7,1);
    \node at (7,0) [anchor=south east, myblue, opacity=0.625] {$+\infty$};
    \node at (-7,0) [anchor=south west, myblue, opacity=0.625] {$-\infty$};
    
    \draw (-3.75,0.1) -- (-3.75,-0.1) node[anchor=north] {$-s_{\max}$};
    \draw (-2.5,0.1) -- (-2.5,-0.1) node[anchor=north] {$-s_{\min}^{+}$};
    
    \filldraw[mypurple] (-5,0) circle (2pt) node[anchor=south] {$-3.0$};
    \filldraw[myyelloworange] (-2,0) circle (2pt) node[anchor=south] {$-0.2$};
    \filldraw[myredorange] (-1,0) circle (2pt) node[anchor=south] {$-0.1$};
    \filldraw[myredorange] (1,0) circle (2pt) node[anchor=south] {$0.1$};
    \filldraw[myyelloworange] (2,0) circle (2pt) node[anchor=south] {$0.2$};
    \filldraw[mypurple] (5,0) circle (2pt) node[anchor=south] {$3.0$};
    
    \draw[myblue, thick, ->, opacity=0.375] (-2.5,1) -- (7,1);
    \draw[myblue, thick, ->, opacity=0.375] (-7,1) -- (-3.75,1);

    \draw[dotted, thick, myblue, opacity=0.625] (-3.75,0) -- (-3.75,1);
    \draw[dotted, thick, myblue, opacity=0.625] (-2.5,0) -- (-2.5,1);

    \end{tikzpicture}
    \caption{Illustration of the sequence of ridge regularization levels $\lambda$ used in \Cref{fig:negative_ridge_visualize_overparam} and the corresponding real projective line.
    (Note: Here $s_{\min}^{+}$ is the smallest \emph{non-zero} eigenvalue of $X^\top X / n$.)}
    \label{fig:negative_ridge_visualize_overparam_sequence}
\end{figure}
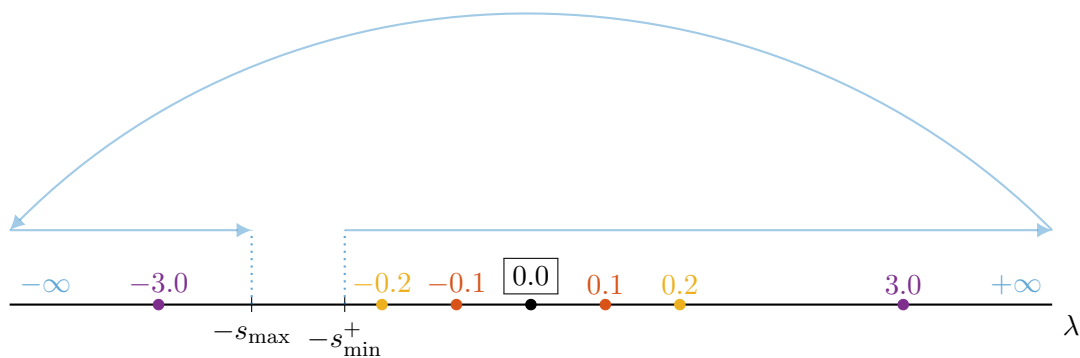

\restoregeometry

\clearpage
\section{Proofs in \Cref{sec:preliminaries}}
\label{sec:preliminaries_appendix}

\subsection{Proof of \Cref{prop:ood-risk-asymptotics}}
\label{sec:prop:ood-risk-asymptotics-proof}

\bigskip

\PropOodRiskAsymptotics*

\begin{proof}
The main workhorse of the proof is \Cref{thm:monotonicty-gen}.
Recall that
\begin{align}
    R(\hat{\bbeta}^{\lambda}) &= (\hbeta^\lambda - \bbeta_0)^\top \bSigma_0 (\hbeta^\lambda - {\bbeta}_0)
    + \sigma_0^2 \notag\\
    &= (\hbeta^\lambda - \bbeta)^\top \bSigma_0 (\hbeta^\lambda - {\bbeta}) +
    2(\hbeta^\lambda - \bbeta)^\top \bSigma_0 (\bbeta - {\bbeta}_0) + 
    \{ (\bbeta_0 - \bbeta)^\top \bSigma_0 (\bbeta_0 - {\bbeta}) 
    + \sigma_0^2 \}. \label{eq:R-decom}
\end{align}
Note that the last quadratic term of \eqref{eq:R-decom} is just a constant.
From \Cref{thm:monotonicty-gen}, we know that the first term of \eqref{eq:R-decom} has the following deterministic equivalent:
\begin{align}
    (\hbeta^\lambda - \bbeta)^\top \bSigma_0 (\hbeta^\lambda - {\bbeta}) &\asympequi \sQ_p(\lambda, \phi,\phi), \label{eq:R-decom-term-1}
\end{align}
where
\begin{align*}
    \sQ_p(\lambda, \phi,\psi) = \tc_p(\lambda, \phi,\psi,\bSigma_0)  + \|\fNL\|_{L_2}^2\tv_p(\lambda, \phi,\psi,\bSigma_0),  
\end{align*}
and the non-negative constants $\tc_p(\lambda, \phi,\psi,\bSigma_0)$ and $\tv_p(\lambda, \phi,\psi,\bSigma_0)$ are defined through the following equations:
\begin{align}
    \frac{1}{v_p(\lambda, \psi)} &= \lambda+\psi \int\frac{r}{1+v_p(\lambda, \psi)r }\rd H_p(r), \label{eq:v}\\
    \tv_p(\lambda, \phi,\psi,\bSigma_0) &= \ddfrac{\phi \otr[\bSigma_0\bSigma(v_p(\lambda, \psi)\bSigma+\bI)^{-2}]}{v_p(\lambda, \psi)^{-2}-\phi \int\frac{r^2}{(1+v_p(\lambda, \psi)r)^2}\rd H_p(r)}, \label{eq:tv_Sigma0}\\
    \tc_p(\lambda, \phi,\psi,\bSigma_0) &=
    \bbeta^{\top}(v_p(\lambda, \psi)\bSigma+\bI)^{-1}(\tv_p(\lambda, \phi,\psi,\bSigma_0)\bSigma+\bSigma_0) (v_p(\lambda, \psi)\bSigma+\bI)^{-1}\bbeta . \notag
\end{align}

For the second term of \eqref{eq:R-decom}, we also have
\begin{align*}
    (\hbeta^\lambda - \bbeta)^\top \bSigma_0 (\bbeta - {\bbeta}_0) &= \bfNL^\top \frac{\bX}{n}(\hSigma + \lambda\bI)^{-1}  \bSigma_0 (\bbeta - {\bbeta}_0) - 
    \bbeta ^\top \lambda(\hSigma + \lambda\bI)^{-1} \bSigma_0 (\bbeta - {\bbeta}_0),
\end{align*}
where $\bfNL$ is defined in the proof of \Cref{thm:monotonicty-gen}.
The first term vanishes from \citet[Lemma D.2]{patil2023generalized} because $\bfNL$ and $\bX$ are uncorrelated, and the remaining part has operator norm tending to zero.
It follows that
\begin{align}
    (\hbeta^\lambda - \bbeta)^\top \bSigma_0 (\bbeta - {\bbeta}_0) \asympequi - \bbeta^\top (v_p(\lambda, \phi)\bSigma + \bI)^{-1} \bSigma_0 (\bbeta - {\bbeta}_0). \label{eq:R-decom-term-2}
\end{align}

Combining \eqref{eq:R-decom}, \eqref{eq:R-decom-term-1}, and \eqref{eq:R-decom-term-2} yields
\begin{align*}
    R(\hat{\bbeta}^{\lambda}) 
    &\asympequi \sR_p(\lambda, \phi,\phi) \\
    &:=\sQ_p(\lambda, \phi,\phi)  - \bbeta^\top (v_p(\lambda, \phi)\bSigma + \bI)^{-1} \bSigma_0 (\bbeta - {\bbeta}_0) + \{ (\bbeta_0 - \bbeta)^\top \bSigma_0 (\bbeta_0 - {\bbeta}) + \sigma_0^2 \},
\end{align*}
which completes the proof.
\end{proof}

\section{Proofs in \Cref{sec:optimal_regularization_signs}}
\label{sec:proofs-sec:optimal_regularization_signs}

Before presenting the proof of our main results, we introduce some notation to facilitate the upcoming presentation.
Define
\begin{align}
    q_b(\bSigma_0, \bB) &= \mu^2 \tr[(\bSigma + \mu \bI)^{-1} \bSigma_0 (\bSigma + \mu \bI)^{-1} \bB] / p  \label{eq:def_q_b} \\
    q_v(\bSigma_0, \bB) &= \phi \tr[(\bSigma + \mu \bI)^{-1} \bSigma_0 (\bSigma + \mu \bI)^{-1} \bB] / p \label{eq:def_q_v}\\
    l(\bSigma_0,\bB) &= \tr[(\bSigma + \mu \bI)^{-1} \bSigma_0 \bB] / p.
\end{align}
We denote the derivative with respect to $\mu$ by:
\begin{align*}
    q_b'(\bSigma_0, \bB) &:= \frac{\partial q_b(\bSigma_0, \bB)}{\partial \mu}, \\
    q_v'(\bSigma_0, \bB) &:= \frac{\partial q_v(\bSigma_0, \bB)}{\partial \mu}.
\end{align*}
We also let $\bB = \bbeta\bbeta^{\top}$ and $\bB_0 = \bbeta_0\bbeta^{\top}$.%

\subsection{Proof of \Cref{thm:stationary-point-no-shift}}
\label{sec:thm:stationary-point-no-shift-proof}

\bigskip

\ThmStationaryPointNoShift*

\begin{proof}
    When $(\bSigma_0,\bbeta_0) = (\bSigma,\bbeta)$, the excess bias term $\sE(\lambda,\phi)=0$.
    Therefore, only the bias and variance terms contribute to the risk.
    We next split the proof into two cases.
    
    \paragraph{Part (1) Underparameterized regime.}   
    When $\phi<1$, from \Cref{lem:fixed-point-v-lambda-properties}, $\mu_p(0,\phi) = 0$, and thus, we have that
    \begin{align}
        0=\sB(0,\phi)\leq \min_{\lambda\in[ \lambda_{\min}(\phi),0]} \sB(\lambda,\phi).\label{eq:stationary-eq-2}
    \end{align}

     On the other hand, when $\phi< 1$, because $\sV'(\lambda,\phi)< 0$, the variance is strictly decreasing over $\lambda \in (\lambda_{\min}(\phi),0)$.
    We have
    \begin{align}
        \sV(0,\phi)< \min_{\lambda\in( \lambda_{\min}(\phi),0)} \sV(\lambda,\phi)\label{eq:stationary-eq-3}
    \end{align}
    It follows that
    \[\min_{\lambda\in (\lambda_{\min}(\phi),0]}\sR(\lambda,\phi) \geq \sR(0,\phi)  \]
    and 
    \[\min_{\lambda\geq\lambda_{\min}}\sR(\lambda,\phi) = \min_{\lambda\geq0}\sR(\lambda,\phi) .\]
    Equivalently, we have $\lambda^* \geq 0$.

    \paragraph{Part (2) Overparameterized regime.}
    When $\phi>1$, we begin by deriving the formula of the derivative of the bias term.
    When $(\bSigma_0,\bbeta_0) = (\bSigma,\bbeta)$, the bias term reduces to
    \begin{align*}
        \sB(\lambda,\phi)   &= \frac{q_v(\bSigma, \bSigma)}{1 - q_v(\bSigma, \bSigma)} q_b(\bSigma, \bB) + q_b(\bSigma, \bB) = \frac{q_b(\bSigma, \bB)}{1 - q_v(\bSigma, \bSigma)} .
    \end{align*}
    From \Cref{lem:ood-risk-deri}, its derivative satisfies that
    \begin{align}
        (1 - q_v(\bSigma, \bSigma))^2\sB'(\lambda, \phi)&= (1 - q_v(\bSigma, \bSigma))^2 h_2'(\mu)\notag\\
        &= q_b'(\bSigma, \bB) (1-q_v(\bSigma, \bSigma)) + q_b(\bSigma, \bB) q_v'(\bSigma, \bSigma) \label{eq:bias-deri-ind}\\
        &= \frac{\lambda}{\mu}q_b'(\bSigma, \bB) + \mu q_b'(\bSigma, \bB)q_v(\bSigma, \bI)  + q_b(\bSigma, \bB) q_v'(\bSigma, \bSigma). \notag
    \end{align}
    where the last equality is from the identity
    \begin{align}
        1 - q_v(\bSigma, \bSigma) & = \frac{\lambda}{\mu} + \mu q_v(\bSigma, \bI) \label{eq:identity-qv}
    \end{align}
    The in-distribution excess term vanishes, i.e., $\cE(\lambda,\phi) = 0$.
    Therefore, the derivative of the risk with respect to $\mu$ satisfies that
    \begin{align}
        &(1 - q_v(\bSigma, \bSigma))^2\sB'(\lambda,\phi) \\
        &= \frac{\lambda}{\mu}q_b'(\bSigma, \bB) + \mu q_b'(\bSigma, \bB)q_v(\bSigma, \bI)  + q_b(\bSigma, \bB) q_v'(\bSigma, \bSigma)\notag\\
        &= \frac{\lambda}{\mu}q_b'(\bSigma, \bB) + \mu q_b'(\bSigma, \bB)q_v(\bSigma, \bI)  + q_b(\bSigma, \bB) q_v'(\bSigma, \bSigma) \notag\\
        &= \frac{\lambda}{\mu}\big[
        \mu\otr[\bSigma_{\bbeta} (\bSigma + \mu \bI)^{-2}] - \mu^2\otr[\bSigma_{\bbeta} (\bSigma + \mu \bI)^{-3}]\big] \notag\\
        &\quad + 2\big[
        \mu\otr[\bSigma_{\bbeta} (\bSigma + \mu \bI)^{-2}] - \mu^2\otr[\bSigma_{\bbeta} (\bSigma + \mu \bI)^{-3}]\big]
        \cdot
        \mu \phi \otr[\bSigma (\bSigma + \mu \bI)^{-2}] \notag\\
        &\quad  - 2\mu\otr[\bSigma_{\bbeta} (\bSigma + \mu \bI)^{-2}]  \cdot \mu \phi  \otr[\bSigma^2 (\bSigma + \mu \bI)^{-3}] \notag\\
        &= 2\lambda \big[
        \otr[\bSigma_{\bbeta} (\bSigma + \mu \bI)^{-3} (\bSigma + \mu \bI)] - \mu\otr[\bSigma_{\bbeta} (\bSigma + \mu \bI)^{-3}] \big] \notag\\
        &\quad + 2 \big[
        \mu\otr[\bSigma_{\bbeta} (\bSigma + \mu \bI)^{-2}] - \mu^2\otr[\bSigma_{\bbeta} (\bSigma + \mu \bI)^{-3}] \big]
        \cdot
        \mu \phi \otr[\bSigma (\bSigma + \mu \bI)^{-2}] \notag\\
        &\quad  - 2\mu\otr[\bSigma_{\bbeta} (\bSigma + \mu \bI)^{-2}]  \cdot \mu \phi  \otr[\bSigma^2 (\bSigma + \mu \bI)^{-3}] \notag\\
        &= 2\lambda \otr[\bSigma_{\bbeta} \bSigma (\bSigma + \mu \bI)^{-3}] \notag\\
        &\quad +2\mu^2 \phi\otr[\bSigma_{\bbeta} (\bSigma + \mu \bI)^{-2}]
        \cdot
        ( \otr[\bSigma (\bSigma + \mu \bI)^{-2}] 
        -  \otr[\bSigma (\bSigma + \mu \bI)^{-3}]) \notag\\
        &\quad
        - 2\mu^3 \phi \otr[\bSigma_{\bbeta} (\bSigma + \mu \bI)^{-3}]
        \cdot
         \otr[\bSigma^2 (\bSigma + \mu \bI)^{-2}] \notag\\
         &= 2\lambda \otr[\bSigma_{\bbeta} \bSigma (\bSigma + \mu \bI)^{-3}] \notag\\
        &\quad +2\mu^3 \phi\otr[\bSigma_{\bbeta} (\bSigma + \mu \bI)^{-2}]
        \cdot \otr[\bSigma (\bSigma + \mu \bI)^{-3} ] \notag\\
        &\quad
        - 2\mu^3 \phi \otr[\bSigma_{\bbeta} (\bSigma + \mu \bI)^{-3}] 
        \cdot
         \otr[\bSigma (\bSigma + \mu \bI)^{-2}].\label{eq:stationary-eq-0}
    \end{align}
    When $\lambda=0$, it follows that
    \begin{align*}
        &\sB'(\lambda,\phi) \\
        &= \frac{2\mu^3 \phi }{(1 - q_v(\bSigma, \bSigma))^2} \big(\otr[\bSigma_{\bbeta} (\bSigma + \mu \bI)^{-2}]
        \cdot \otr[\bSigma (\bSigma + \mu \bI)^{-3} ] - \otr[\bSigma_{\bbeta} (\bSigma + \mu \bI)^{-3}] 
        \cdot
         \otr[\bSigma (\bSigma + \mu \bI)^{-2}]\big) ,
    \end{align*}
    and 
    \begin{align}
        &\frac{\partial \sB(\lambda,\phi)}{\partial\lambda} \notag \\
        &= \frac{2\mu^3 \phi }{(1 - q_v(\bSigma, \bSigma))^2} \big(\otr[\bSigma_{\bbeta} (\bSigma + \mu \bI)^{-2}]
        \cdot \otr[\bSigma (\bSigma + \mu \bI)^{-3} ] - \otr[\bSigma_{\bbeta} (\bSigma + \mu \bI)^{-3}] 
        \cdot
         \otr[\bSigma (\bSigma + \mu \bI)^{-2}] \big) \notag\\
         &\quad \cdot \frac{\partial \mu(\lambda,\phi)}{\partial \lambda}, \label{eq:stationary-eq-deri}
    \end{align}
    where $\partial \mu(\lambda,\phi)/\partial \lambda >0$ from \Cref{lem:fixed-point-v-lambda-properties}.
    
    From \Cref{lem:fixed-point-v-lambda-properties} and \eqref{eq:stationary-eq-deri}, since $\partial \mu(\lambda,\phi)/\partial \lambda >0$, we further have that
    \begin{align*}
        \frac{\partial \sB(\lambda,\phi)}{\partial\lambda} &\propto \otr[\bSigma_{\bbeta} (\bSigma + \mu \bI)^{-2}]
        \cdot \otr[\bSigma (\bSigma + \mu \bI)^{-3} ] - \otr[\bSigma_{\bbeta} (\bSigma + \mu \bI)^{-3}] 
        \cdot
         \otr[\bSigma (\bSigma + \mu \bI)^{-2}],
    \end{align*}
    which finishes the proof of the first conclusion.
    
    To obtain the sign of the optimal ridge penalty, note that
    \begin{align*}
        &(1 - q_v(\bSigma, \bSigma))^2\sB'(\lambda,\phi)\\
        &= 2\lambda \otr[\bSigma_{\bbeta} \bSigma (\bSigma + \mu \bI)^{-3}]\\
        &\qquad + 2\mu^3 \phi(\otr[\bSigma_{\bbeta} (\bSigma + \mu \bI)^{-2}]
        \cdot \otr[\bSigma (\bSigma + \mu \bI)^{-3} ]-\otr[\bSigma_{\bbeta} (\bSigma + \mu \bI)^{-3}] 
        \cdot
         \otr[\bSigma (\bSigma + \mu \bI)^{-2}]) \\
         &=T_1 + T_2.
    \end{align*}
    When $\phi>1$, because $\mu\geq 0 $, we have that $T_1<0$ when $\lambda<0$ and $T_1>0$ when $\lambda>0$.
    Also, under the assumption that $\otr[\bSigma_{\bbeta} (\bSigma + \mu \bI)^{-2}] \cdot \otr[\bSigma (\bSigma + \mu \bI)^{-3} ]-\otr[\bSigma_{\bbeta} (\bSigma + \mu \bI)^{-3}] \cdot \otr[\bSigma (\bSigma + \mu \bI)^{-2}]> 0$, we have that $T_2>0$ for all $\lambda> \lambda_{\min}(\phi)$, with equality holds only when $\lambda=0$. 
    Thus, $\sB(\lambda,\phi)$ is minimized at $\lambda<0$.

    For the variance term, from \Cref{lem:ood-risk-deri}, we have
    \begin{align}
        \sV'(\lambda,\phi) &= \frac{ \sigma^2 q_v'(\bSigma, \bSigma)}{(1 - q_v(\bSigma, \bSigma))^2} \notag\\
        &= -2\sigma^2\frac{\phi \otr[\bSigma^2 (\bSigma + \mu \bI)^{-3}]}{(1 - q_v(\bSigma, \bSigma))^2} \notag\\
        &= 2\phi\sigma^2\frac{ \mu \otr[\bSigma (\bSigma + \mu \bI)^{-3}]-\otr[\bSigma (\bSigma + \mu \bI)^{-2}]   }{(1 - q_v(\bSigma, \bSigma))^2} \notag\\
        &:= \frac{ T_3 + T_4  }{(1 - q_v(\bSigma, \bSigma))^2} \notag\\
        &= \frac{- 2\sigma^2 \phi \otr[\bSigma^2 (\bSigma + \mu \bI)^{-3}] }{(1 - q_v(\bSigma, \bSigma))^2}. \label{eq:var-ind}
    \end{align}
    Note that $- 2\sigma^2 \phi \otr[\bSigma^2 (\bSigma + \mu \bI)^{-3}] < 0$ and strictly increasing over $\lambda\geq 0$.
    Thus,
    \begin{align*}
        &(1 - q_v(\bSigma, \bSigma))^2\sR'(\lambda,\phi) \\
        &= (1 - q_v(\bSigma, \bSigma))^2[\sB'(\lambda,\phi) + \sV'(\lambda,\phi)]\\
        &= 2\lambda \otr[\bSigma_{\bbeta} \bSigma (\bSigma + \mu \bI)^{-3}]\\
        &\qquad + 2\mu^3 \phi\big(\otr[\bSigma_{\bbeta} (\bSigma + \mu \bI)^{-2}]
        \cdot \otr[\bSigma (\bSigma + \mu \bI)^{-3} ]-\otr[\bSigma_{\bbeta} (\bSigma + \mu \bI)^{-3}] 
        \cdot
         \otr[\bSigma (\bSigma + \mu \bI)^{-2}]\big)\\
         &\qquad + 2\phi\sigma^2 \big( \mu \otr[\bSigma (\bSigma + \mu \bI)^{-3}]-\otr[\bSigma (\bSigma + \mu \bI)^{-2}] \big) \\
         &= 2\lambda \otr[\bSigma_{\bbeta} \bSigma (\bSigma + \mu \bI)^{-3}] \\
         &\qquad + 2\phi\big(\mu^3 \otr[\bSigma_{\bbeta} (\bSigma + \mu \bI)^{-2}]
         + \mu\sigma^2\big) \cdot \otr[\bSigma (\bSigma + \mu \bI)^{-3} ] -2\phi\big(\mu^3\otr[\bSigma_{\bbeta} (\bSigma + \mu \bI)^{-3}] 
         + \sigma^2\big) \cdot
          \otr[\bSigma (\bSigma + \mu \bI)^{-2}] .
    \end{align*}
    Under the condition that 
    \begin{align*}
        \frac{\otr[\bSigma \{ \mu^3 (\bSigma + \mu \bI)^{-3} \} ]}{\otr[\bSigma \{ \mu^2 (\bSigma + \mu \bI)^{-2} \}]} > \frac{\otr[\bSigma_{\bbeta} \{ \mu^3 (\bSigma + \mu \bI)^{-3} \}] 
        + \sigma^2}{\otr[\bSigma_{\bbeta} \{ \mu^{2} (\bSigma + \mu \bI)^{-2} \}]
        + \sigma^2},
    \end{align*}
    it follows that for all $\lambda\geq 0$,
    \begin{align*}
        (1 - q_v(\bSigma, \bSigma))^2\sR'(\lambda,\phi) > 0.
    \end{align*}
    This implies that $\sR(\lambda,\phi)$ is minimized at $\lambda<0$, which finishes the proof.
\end{proof}  

\subsection{General Alignment Condition of \Cref{thm:stationary-point-no-shift} under Special Cases}
\label{sec:alignment-special-cases}

\begin{remark}[\Cref{thm:stationary-point-no-shift} under isotropic features or signals]\label{rmk:thm:stationary-point-no-shift-iso}
    When $\phi>1$ and $\bSigma=\bI$, the condition above is never satisfied because
    \[
        \frac{\otr[\bSigma_{\bbeta} \{ \mu^3 (\bSigma + \mu \bI)^{-3} \}] 
        + \sigma^2}{\otr[\bSigma_{\bbeta} \{ \mu^{2} (\bSigma + \mu \bI)^{-2} \}]
        + \sigma^2}
        =
        \frac{\left(\frac{\mu}{1+\mu}\right)^3\otr[\bB] 
        + \sigma^2}{\left(\frac{\mu}{1+\mu}\right)^2\otr[\bB ]
        + \sigma^2}
        > 
        \frac{\left(\frac{\mu}{1+\mu}\right)^3}{\left(\frac{\mu}{1+\mu}\right)^2}
        = 
        \frac{\otr[\bSigma \{ \mu^3 (\bSigma + \mu \bI)^{-3} \}]}{\otr[\bSigma \{ \mu^{2} (\bSigma + \mu \bI)^{-2} \}]}.
    \]
    
    Similarly, when $\bbeta$ is isotropic, the condition above is never satisfied because
    \[
        \frac{\otr[\bSigma_{\bbeta} \{ \mu^3 (\bSigma + \mu \bI)^{-3} \}] 
        + \sigma^2}{\otr[\bSigma_{\bbeta} \{ \mu^{2} (\bSigma + \mu \bI)^{-2} \}]
        + \sigma^2}
        =
        \frac{\otr[\bSigma \{ \mu^3 (\bSigma + \mu \bI)^{-3} \}] 
        + \sigma^2}{\otr[\bSigma \{ \mu^{2} (\bSigma + \mu \bI)^{-2} \}]
        + \sigma^2}
        >
        \frac{\otr[\bSigma \{ \mu^3 (\bSigma + \mu \bI)^{-3} \}]}{\otr[\bSigma \{ \mu^{2} (\bSigma + \mu \bI)^{-2} \}]}
    \]
    since $\otr[\bSigma \{ \mu^3 (\bSigma + \mu \bI)^{-3} \}] < \otr[\bSigma \{ \mu^{2} (\bSigma + \mu \bI)^{-2} \}]$.
    To see this, note that $\mu>0$ when $\phi>1$, from \Cref{lem:fixed-point-v-lambda-properties}.
    Thus,
    \begin{align*}
        \otr[\bSigma \{ \mu^3 (\bSigma + \mu \bI)^{-3} \}]
        &= \otr[\bSigma \{ \mu^2 (\bSigma + \mu \bI)^{-2} \} \{ \mu (\bSigma + \mu \bI)^{-1} \}] \\
        &\le \otr[\bSigma \{ \mu^2 (\bSigma + \mu \bI)^{-2} \}]
        \cdot
        \| \mu (\bSigma + \mu \bI)^{-1} \|_{\mathrm{op}} \\
        & \le \otr[\bSigma \{ \mu^2 (\bSigma + \mu \bI)^{-2} \}].
    \end{align*}
\end{remark}

\begin{proposition}[\Cref{thm:stationary-point-no-shift} under strict alignment conditions]\label{prop:thm:stationary-point-no-shift-strict}
    Assuming random signals, zero noise under the strict alignment conditions of \citet{wu_xu_2020}, the general alignment condition \eqref{eq:align-cond} is satisfied.
\end{proposition}
\begin{proof}

Let $h$ be a random variable following the empirical measure of the eigenvalues of $\bSigma=\bU\bLambda\bU$ (i.e., $H_p$) and $g$ be a random variable following the empirical measure $\diag(\bU\EE[\bbeta^{\top}\bbeta]\bU)$.
When the joint distribution of $(h,g)$ exists, there exists a function $f$ such that $g$ has mass $r\mapsto f(r) \rd H_p(r)$.
The strict alignment condition from \citet{wu_xu_2020} imposes $f$ to be either strictly increasing or decreasing.
This also implies that
\begin{align}
    \otr[\bSigma_{\bbeta}]
    = \otr[\bSigma \bB]
    \ge \otr[\bSigma] \otr[\bB]. \label{eq:strict-align-imp}
\end{align}
This holds because of the following Chebyshev's ``other'' inequality \citep{fink1984chebyshev}.\footnote{This is the second (less well-known) inequality, and not the first (more well-known) tail inequality that may come to the reader's mind!}
\begin{fact}[Positive/negative correlations of monotone functions; see, e.g., Appendix 9.9 of \citet{ross2022simulation}]
    Let $f$ and $g$ be two real-valued functions of the same monotonicity.
    Let $H$ be a probability measure on the real line.
    Then the following inequality holds:
    \begin{equation}
        \label{eq:corr_monotone_functions}
        \int f(r) g(r) \, \mathrm{d}H(r)
        \ge \int f(r) \, \mathrm{d}H(r)
        \cdot \int g(r) \, \mathrm{d}H(r).
    \end{equation}
    On the other hand, if $f$ is non-decreasing and $g$ is non-increasing, then the inequality in \eqref{eq:corr_monotone_functions} is reversed. 
\end{fact}

Now we will verify that this condition indeed implies \eqref{eq:align-cond} when $\sigma^2 = 0$.
Define $f(r) = \sum_{i=1}^p(\bbeta^{\top}\bw_i)^2\ind\{r=r_i\}$.
$H_p(r)=p^{-1}\sum_{i=1}^p\ind\{r=r_i\}$ and the transformed measure $\tilde{H}_p(r)=r(r+\mu)^{-2} p^{-1}\sum_{i=1}^p\ind\{r=r_i\} / \int r(r+\mu)^{-2} \rd H_p(r)$.

Because $f(r)$ is increasing and $1/({r+\mu})$ is decreasing in $r$, it follows that
\begin{align*}
    \int  f(r)\frac{r}{r+\mu}\rd\tilde{H}_p(r) &\leq \int  f(r)\rd\tilde{H}_p(r) \int  \frac{r}{r+\mu}\rd\tilde{H}_p(r).
\end{align*}
Transforming back to $H_p(r)$ yields that
\begin{align*}
    \int f(r)\frac{r}{(r+\mu)^3} \rd H_p(r) \cdot \int \frac{r}{(r+\mu)^2} \rd H_p(r) &\leq \int f(r)\frac{r}{(r+\mu)^2} \rd H_p(r) \cdot \int \frac{r}{(r+\mu)^3} \rd H_p(r).
\end{align*}
Equivalently,
\[
    \frac{\otr[\bSigma_{\bbeta} \{ \mu^{2} (\bSigma + \mu \bI)^{-2} \}]
    }{\otr[\bSigma_{\bbeta} \{ \mu^3 (\bSigma + \mu \bI)^{-3} \}] 
    }
    \geq
    \frac{\otr[\bSigma \{ \mu^{2} (\bSigma + \mu \bI)^{-2} \}]}{\otr[\bSigma \{ \mu^3 (\bSigma + \mu \bI)^{-3} \}]},
\]
with equality holds if and only if $f(r)$ is $H_p$-almost surely constant.
This finishes the proof.
\end{proof}

\subsection{Proof of \Cref{thm:stationary-point-cov-shift-iso}}
\label{sec:thm:stationary-point-cov-shift-iso-proof}

\bigskip

\ThmStationaryPointCovShiftIso*

\begin{proof}
    With slight abuse of notations, we use $\sR$, $\sB$, $\sV$, and $\sS$ to denote $\EE_{\bbeta}[\sR]$, $\EE_{\bbeta}[\sB]$, $\EE_{\bbeta}[\sV]$, and $\EE_{\bbeta}[\sS]$ as well.
    We split the proof into two parts.
    
    \paragraph{Part (1) Asymptotic risk.}
    When $\bbeta_0=\bbeta$, the extra bias term is zero, that is, $\sE(\lambda,\phi)=0$.
    So, only the bias and variance components contribute to the risk.

    We begin by deriving the formula for the derivative of the bias.
    For isotropic signals, from \Cref{lem:ood-risk-deri} and \eqref{eq:identity-qv}, we have
    \begin{align*}
        &(1 - q_v(\bSigma, \bSigma))^2\sB'(\lambda, \phi) \\
        &= \alpha^2 [q_v'(\bSigma_0, \bSigma) q_b(\bSigma, \bI) (1 - q_v(\bSigma, \bSigma))\\
        &\quad+ q_v(\bSigma_0, \bSigma) q_b'(\bSigma, \bI) (1-q_v(\bSigma, \bSigma)) + q_v(\bSigma_0, \bSigma) q_b(\bSigma, \bI) q_v'(\bSigma, \bSigma) \\
        &\quad + q_b'(\bSigma_0, \bI)(1 - q_v(\bSigma, \bSigma))^2 ] \\
        &= \alpha^2 \bigg[q_v'(\bSigma_0, \bSigma) q_b(\bSigma, \bI) \left(\frac{\lambda}{\mu} + \mu q_v(\bSigma, \bI) \right) \\
        &\quad+ q_v(\bSigma_0, \bSigma) q_b'(\bSigma, \bI) \left(\frac{\lambda}{\mu} + \mu q_v(\bSigma, \bI) \right) + q_v(\bSigma_0, \bSigma) q_b(\bSigma, \bI) q_v'(\bSigma, \bSigma) \\
        &\quad + q_b'(\bSigma_0, \bI)\left(\frac{\lambda}{\mu} + \mu q_v(\bSigma, \bI) \right)^2 \bigg] \\
        &= \frac{\alpha^2\lambda }{\mu}\left(q_v'(\bSigma_0, \bSigma) q_b(\bSigma, \bI) + q_v(\bSigma_0, \bSigma) q_b'(\bSigma, \bI) + \frac{\lambda}{\mu}q_b'(\bSigma_0, \bI) + 2\mu q_b'(\bSigma_0, \bI)q_v(\bSigma, \bI)\right)\\
        &\quad + \alpha^2 \Big[ \mu q_v'(\bSigma_0, \bSigma) q_b(\bSigma, \bI)   q_v(\bSigma, \bI) \\
        &\quad+ \mu q_v(\bSigma_0, \bSigma) q_b'(\bSigma, \bI)  q_v(\bSigma, \bI) + q_v(\bSigma_0, \bSigma) q_b(\bSigma, \bI) q_v'(\bSigma, \bSigma) \\
        &\quad + \mu^2 q_b'(\bSigma_0, \bI) q_v(\bSigma, \bI)^2 \Big].
    \end{align*}
    Notice that
    \begin{align*}
        q_b(\bSigma_0, \bI) &= \mu^2\otr[\bSigma_0(\bSigma+\mu\bI)^{-2}] \\
        &= \frac{\mu^2}{\phi} q_v(\bSigma_0,\bI)\\
        q_b'(\bSigma_0, \bI) &= 2\mu\otr[\bSigma_0(\bSigma+\mu\bI)^{-2}] - 2 \mu^2\otr[\bSigma_0(\bSigma+\mu\bI)^{-3}]\\
        &= 2\mu \otr[\bSigma_0\bSigma(\bSigma+\mu\bI)^{-3}] \\
        &= - \frac{\mu}{\phi}q_v'(\bSigma_0,\bSigma),
    \end{align*}
    thus, it follows that
    \begin{align*}
        q_v'(\bSigma_0, \bSigma) q_b(\bSigma, \bI) + \mu q_b'(\bSigma_0, \bI) q_v(\bSigma, \bI) &= q_v'(\bSigma_0,\bSigma)\left[q_b(\bSigma, \bI) - \frac{\mu^2}{\phi} q_v(\bSigma, \bI)\right] = 0\\
        \mu q_b'(\bSigma, \bI)  q_v(\bSigma, \bI) + q_b(\bSigma, \bI) q_v'(\bSigma, \bSigma) &= \frac{\mu^2}{\phi}  \left[-  q_v'(\bSigma, \bSigma) q_v(\bSigma, \bI) + q_v'(\bSigma, \bSigma) q_v(\bSigma, \bI)  \right] =0.
    \end{align*}
    Therefore, we have
    \begin{align*}
        &(1 - q_v(\bSigma, \bSigma))^2\sB'(\lambda, \phi) \\
        &= \frac{\alpha^2\lambda }{\mu}\left( q_v(\bSigma_0, \bSigma) q_b'(\bSigma, \bI) + \frac{\lambda}{\mu}q_b'(\bSigma_0, \bI) + \mu q_b'(\bSigma_0,\bI)q_v(\bSigma,\bI)\right) \notag\\
        &= - \frac{\alpha^2\lambda }{\phi}q_v(\bSigma_0, \bSigma) q_v'(\bSigma, \bSigma) - \frac{\alpha^2\lambda^2}{\mu\phi}q_v'(\bSigma_0, \bSigma) - \frac{\alpha^2\lambda\mu}{\phi}q_v'(\bSigma_0, \bSigma) q_v(\bSigma,\bI)\\
        &= - \frac{\alpha^2\lambda }{\phi}[q_v(\bSigma_0, \bSigma) q_v'(\bSigma, \bSigma) +\mu q_v'(\bSigma_0, \bSigma) q_v(\bSigma,\bI)] - \frac{\alpha^2\lambda^2}{\mu\phi}q_v'(\bSigma_0, \bSigma) .
    \end{align*}

    From \eqref{eq:identity-qv}, we further have that
    \begin{align*}
        \sB'(\lambda, \phi)  &= - \frac{\alpha^2\lambda}{\phi (1 - q_v(\bSigma, \bSigma))^2}[q_v(\bSigma_0, \bSigma) q_v'(\bSigma, \bSigma) - q_v'(\bSigma_0, \bSigma) q_v(\bSigma, \bSigma) + q_v'(\bSigma_0, \bSigma) ]\\
        &= - \frac{\alpha^2\lambda}{\phi} \frac{\partial}{\partial\mu} \frac{q_v(\bSigma_0, \bSigma) }{1-q_v(\bSigma, \bSigma)}.
    \end{align*}
    From \Cref{lem:ood-risk-deri}, we know that
    \[\frac{\partial}{\partial\mu} \frac{q_v(\bSigma_0, \bSigma) }{ 1-q_v(\bSigma, \bSigma)}<0\]
    for all $\lambda\in(-\lambda_{\min}(\phi),+\infty)$.
    Furthermore, we have
    \begin{align*}
        \sR'(\lambda, \phi) &= \sB'(\lambda, \phi)  + \sV'(\lambda, \phi) \\
        &=  - \frac{\alpha^2\lambda}{\phi} \frac{\partial}{\partial\mu} \frac{q_v(\bSigma_0, \bSigma) }{1-q_v(\bSigma, \bSigma)} + \sigma^2 \frac{\partial}{\partial\mu} \frac{q_v(\bSigma_0, \bSigma) }{1-q_v(\bSigma, \bSigma)}.
    \end{align*}
    Setting the above to zero gives $\lambda^* = \phi \sigma^2/\alpha^2$. 
    Also note that $\sR'(\lambda,\phi)$ is negative when $\lambda<\lambda^*$ and positive when $\lambda>\lambda^*$, as well as $\partial \mu/\partial\lambda>0$ from \Cref{lem:fixed-point-v-lambda-properties} for all $\lambda\geq \lambda_{\min}(\phi)$.
    Thus, $\lambda^*$ gives the optimal risk.

    \paragraph{Part (2) Finite-sample risk.}
    Recall the finite-sample risk is given by \eqref{eq:ridge_prederr}.
    Under \Cref{asm:train-test} and isotropic signals, it follows that
    \begin{align*}
        R(\hbeta^\lambda) &= \lambda^2 \alpha^2 \otr[(\hSigma + \lambda \bI)^{-2} \bSigma_0]
            + \sigma^2 \phi \otr[(\hSigma + \lambda \bI)^{-2} \hSigma \bSigma_0].
    \end{align*}    
    Taking the derivative with respect to $\lambda$ yields that
    \begin{align*}
        \frac{\partial R(\hbeta^\lambda)}{\partial\lambda} & = 2\lambda \alpha^2 \otr[(\hSigma + \lambda \bI)^{-2} \bSigma_0]
        -2\lambda^2 \alpha^2 \otr[(\hSigma + \lambda \bI)^{-3} \bSigma_0] -2 \sigma^2 \phi \otr[(\hSigma + \lambda \bI)^{-3} \hSigma \bSigma_0] \\
        &= 2\lambda \alpha^2 \otr[(\hSigma + \lambda \bI)^{-3} \hSigma \bSigma_0] - 2 \sigma^2 \phi\otr[(\hSigma + \lambda \bI)^{-3} \hSigma \bSigma_0]\\
        &= 2(\lambda \alpha^2  - \sigma^2\phi )\otr[(\hSigma + \lambda \bI)^{-2} \hSigma\bSigma_0].
    \end{align*}
    Setting the above to zero gives $\lambda_p^* = \phi_n/\SNR$.
\end{proof}

\subsection{Proof of \Cref{thm:stationary-point-cov-shift}}
\label{sec:thm:stationary-point-cov-shift-proof}

\bigskip

\ThmStationaryPointCovShift*

\begin{proof}

    When $\bbeta_0=\bbeta$, the extra bias term is zero, that is, $\sE(\lambda,\phi)=0$.
    So, only bias and variance contribute to the risk.
    We next split the proof into two cases.

    \paragraph{Part (1) Underparameterized regime.}
    When $\phi< 1$, from \Cref{lem:ridge-fixed-point-v-properties} we have $\mu(0,\phi) = 0$ and the bias defined in \Cref{prop:ood-risk-asymptotics} becomes zero, i.e., $\sB(0,\phi) = 0$.
    From the fixed-point equation \eqref{eq:basic-ridge-equivalence-mu-fixed-point}, we also have $\lim_{\lambda\rightarrow0^+}\lambda/\mu(\lambda,\phi)=0$.
    From \Cref{lem:ood-risk-deri} and \eqref{eq:identity-qv}, we have
    \begin{align*}
        &(1 - q_v(\bSigma, \bSigma))^2\sB'(\lambda, \phi) \\
        &=  q_v'(\bSigma_0, \bSigma) q_b(\bSigma, \bB) (1 - q_v(\bSigma, \bSigma))\\
        &\quad+ q_v(\bSigma_0, \bSigma) q_b'(\bSigma, \bB) (1-q_v(\bSigma, \bSigma)) + q_v(\bSigma_0, \bSigma) q_b(\bSigma, \bB) q_v'(\bSigma, \bSigma) \\
        &\quad + q_b'(\bSigma_0, \bB)(1 - q_v(\bSigma, \bSigma))^2  \\
        &= q_v'(\bSigma_0, \bSigma) q_b(\bSigma, \bB) \left(\frac{\lambda}{\mu} + \mu q_v(\bSigma, \bI) \right) \\
        &\quad+ q_v(\bSigma_0, \bSigma) q_b'(\bSigma, \bB) \left(\frac{\lambda}{\mu} + \mu q_v(\bSigma, \bI) \right) + q_v(\bSigma_0, \bSigma) q_b(\bSigma, \bB) q_v'(\bSigma, \bSigma) \\
        &\quad  + q_b'(\bSigma_0, \bB)\left(\frac{\lambda}{\mu} + \mu q_v(\bSigma, \bI) \right)^2  \\
        &= \frac{\lambda }{\mu}\left(q_v'(\bSigma_0, \bSigma) q_b(\bSigma, \bB) + q_v(\bSigma_0, \bSigma) q_b'(\bSigma, \bB) + \frac{\lambda}{\mu}q_b'(\bSigma_0, \bB) + 2\mu q_b'(\bSigma_0, \bB)q_v(\bSigma, \bI)\right)\\
        &\quad + \Big[\mu q_v'(\bSigma_0, \bSigma) q_b(\bSigma, \bB)   q_v(\bSigma, \bI) \\
        &\quad+ \mu q_v(\bSigma_0, \bSigma) q_b'(\bSigma, \bB)  q_v(\bSigma, \bI) + q_v(\bSigma_0, \bSigma) q_b(\bSigma, \bB) q_v'(\bSigma, \bSigma) \\
        &\quad + \mu^2 q_b'(\bSigma_0, \bB) q_v(\bSigma, \bI)^2\Big].
    \end{align*}
    Thus, the derivative of the bias term becomes zero, i.e., $\sB'(0,\phi) = 0$.
    Note that the variance term is strictly decreasing over $(\lambda_{\min}(\phi),+\infty)$ from \Cref{lem:ood-risk-deri}.
    Similarly to part (1) of the proof of \Cref{thm:stationary-point-no-shift}, it follows that $\lambda^*\geq 0$.

    \paragraph{Part (2) Overparameterized regime and $\bSigma=\bI$.}
    When $\bSigma=\bI$, the above derivative in Part (1) becomes
    \begin{align*}
        &(1 - q_v(\bI, \bI))^2\sB'(\lambda, \phi) \\
        &= \frac{\lambda }{\mu}\left(q_v'(\bSigma_0, \bI) q_b(\bI, \bB) + q_v(\bSigma_0, \bI) q_b'(\bI, \bB) + \frac{\lambda}{\mu}q_b'(\bSigma_0, \bB) + 2\mu q_b'(\bSigma_0, \bB)q_v(\bI, \bI)\right)\\
        &\quad + \Big[\mu q_v'(\bSigma_0, \bI) q_b(\bI, \bB)   q_v(\bI, \bI) \\
        &\quad+ \mu q_v(\bSigma_0, \bI) q_b'(\bI, \bB)  q_v(\bI, \bI) + q_v(\bSigma_0, \bI) q_b(\bI, \bB) q_v'(\bI, \bI) \\
        &\quad + \mu^2 q_b'(\bSigma_0, \bB) q_v(\bI, \bI)^2\Big].
    \end{align*}
    Note from \eqref{eq:def_q_b} and \eqref{eq:def_q_v}, we have that
    \begin{align*}        
        q_b(\bSigma_0, \bB) &= \frac{\mu^2}{(1+\mu)^{2}} \otr[\bSigma_0\bB] \\
        q_b'(\bSigma_0, \bB) &= \left( -\frac{2\mu^2}{(1+\mu)^3} + \frac{2\mu}{(1+\mu)^2} \right) \otr[\bSigma_0\bB] = \frac{2\mu}{(1+\mu)^3} \otr[\bSigma_0\bB] = \frac{2}{\mu(1+\mu)}q_b(\bSigma_0,\bB),\\
        q_v(\bI,\bI) &= \frac{\phi}{(1+\mu)^2}\\
        q_v'(\bI,\bI) &= -\frac{2\phi}{(1+\mu)^3} = - \frac{2}{1+\mu} q_v(\bI,\bI) \\
        q_v(\bSigma_0, \bI) &= \frac{\phi}{(1+\mu)^2} \otr[\bSigma_0] \\
        q_v'(\bSigma_0, \bI) &= - \frac{2\phi}{(1+\mu)^3} \otr[\bSigma_0] = -\frac{2}{1+\mu}q_v(\bSigma_0, \bI).
    \end{align*}
    We further have
    \begin{align*}
        &(1 - q_v(\bI, \bI))^2\sB'(\lambda, \phi) \\
        &=\frac{\lambda }{\mu}\left(- \frac{2}{1+\mu}q_v(\bSigma_0, \bI) q_b(\bI, \bB) + \frac{2}{\mu(1+\mu)} q_v(\bSigma_0, \bI) q_b(\bI, \bB) + \frac{2\lambda}{\mu^2(1+\mu)}q_b(\bSigma_0, \bB) + \frac{4}{1+\mu} q_b(\bSigma_0, \bB)q_v(\bI, \bI)\right)\\
        &\quad +  [-\frac{2\mu}{1+\mu}q_v(\bSigma_0, \bI) q_b(\bI, \bB)   q_v(\bI, \bI) \\
        &\quad  + \frac{2}{1+\mu} q_v(\bSigma_0, \bI) q_b(\bI, \bB)  q_v(\bI, \bI) - \frac{2}{1+\mu} q_v(\bSigma_0, \bI) q_b(\bI, \bB) q_v(\bI, \bI) \\
        &\quad + \frac{2\mu}{1+\mu} q_b(\bSigma_0, \bB) q_v(\bI, \bI)^2 ]\\
        &= \frac{\lambda }{\mu}\left( \frac{2\phi\mu(1-\mu)}{(1+\mu)^5} \otr[\bSigma_0]\otr[\bB]+ \frac{2\lambda}{(1+\mu)^3}\otr[\bSigma_0\bB] + \frac{4\phi\mu^2}{(1+\mu)^5}\otr[\bSigma_0\bB]\right)\\
        &\quad -\frac{2\mu}{1+\mu}q_v(\bSigma_0, \bI) q_b(\bI, \bB)   q_v(\bI, \bI) \\
        &\quad + \frac{2\mu}{1+\mu} q_b(\bSigma_0, \bB) q_v(\bI, \bI)^2 \\
        &= \frac{\lambda }{\mu}\left(\frac{2\phi\mu(1-\mu)}{(1+\mu)^5} \otr[\bSigma_0]\otr[\bB]+ \frac{2\lambda}{(1+\mu)^3}\otr[\bSigma_0\bB] + \frac{4\phi\mu^2}{(1+\mu)^5}\otr[\bSigma_0\bB]\right)\\
        &\quad + \frac{2\phi^2\mu^3}{(1+\mu)^7} (  \otr[\bSigma_0\bB] -\otr[\bSigma_0]\otr[\bB] )\\
        &= \frac{\lambda }{\mu}\left(\frac{2\phi\mu(1-\mu)}{(1+\mu)^5} \otr[\bSigma_0]\otr[\bB]+ \frac{2[\mu(1+\mu)^2 - \phi\mu(1-\mu)]}{(1+\mu)^5}\otr[\bSigma_0\bB]\right)\\
        &\quad + \frac{2\phi^2\mu^3}{(1+\mu)^7} (  \otr[\bSigma_0\bB] -\otr[\bSigma_0]\otr[\bB] )\\
        &= \frac{\lambda }{\mu}\left(\frac{2\phi\mu(1-\mu)}{(1+\mu)^5} (\otr[\bSigma_0]\otr[\bB]-\otr[\bSigma_0\bB])+ \frac{2\mu}{(1+\mu)^3}\otr[\bSigma_0\bB]\right)\\
        &\quad + \frac{2\phi^2\mu^3}{(1+\mu)^7} (  \otr[\bSigma_0\bB] -\otr[\bSigma_0]\otr[\bB] ).
    \end{align*}

    For the variance term, from \Cref{lem:ood-risk-deri}, we have
    \begin{align*}
        (1 - q_v(\bI, \bI))^2\sV'(\lambda,\phi) %
        &= \sigma^2[ q_v'(\bSigma_0, \bI)(1 - q_v(\bI, \bI)) + q_v(\bSigma_0, \bI) q_v'(\bI, \bI)]\\
        &= \sigma^2\left[ q_v'(\bSigma_0, \bI) \left(\frac{\lambda}{\mu} + \mu q_v(\bI, \bI) \right) + q_v(\bSigma_0, \bI) q_v'(\bI, \bI)\right]\\
        &= -\sigma^2\frac{2\lambda}{\mu(1+\mu)}q_v(\bSigma_0, \bI) - \sigma^2 \frac{2}{1+\mu}\left[ \mu q_v(\bSigma_0, \bI) q_v(\bI, \bI) + q_v(\bSigma_0, \bI) q_v(\bI, \bI)\right]\\
        &= -\sigma^2\frac{2\lambda\phi}{\mu(1+\mu)^3} \otr[\bSigma_0] - \sigma^2 \frac{2\phi^2}{(1+\mu)^4}\otr[\bSigma_0] .
    \end{align*}
    From the fixed-point equation \eqref{eq:basic-ridge-equivalence-mu-fixed-point} we have $\mu(1+\mu)-\phi\mu =  \lambda(1+\mu)$ and thus,
    \begin{align*}
        (1 - q_v(\bI, \bI))^2\sV'(\lambda,\phi)& = -\sigma^2\frac{2\phi(\mu(1+\mu)-\phi\mu)}{\mu(1+\mu)^4} \otr[\bSigma_0] - \sigma^2 \frac{2\phi^2}{(1+\mu)^4}\otr[\bSigma_0]\\
        &= -\sigma^2 \frac{2\phi}{(1+\mu)^3}\otr[\bSigma_0],
    \end{align*}
    which is strictly increasing in $\lambda\geq 0$.

    Then we have 
    \begin{align}
        (1 - q_v(\bI, \bI))^2\sR'(\lambda,\phi) \notag
        &= (1 - q_v(\bSigma, \bSigma))^2[\sB'(\lambda,\phi) + \sV'(\lambda,\phi)]\notag\\
         &\geq  (1 - q_v(\bSigma, \bSigma))^2|_{\lambda=0} [\sB'(0;\phi) + \sV'(0;\phi)]\notag\\
         &= 2\lambda\left(\frac{\phi(1-\mu)}{(1+\mu)^5} (\otr[\bSigma_0]\otr[\bB]-\otr[\bSigma_0\bB])+ \frac{1}{(1+\mu)^3}\otr[\bSigma_0\bB] \right)\notag \\
         &\qquad + \frac{2\phi^2\mu^3}{(1+\mu)^7} (  \otr[\bSigma_0\bB] -\otr[\bSigma_0]\otr[\bB] ) - \sigma^2 \frac{2\phi}{(1+\mu)^3}\otr[\bSigma_0]\notag  \\
         &=  \frac{2\lambda\phi}{(1+\mu)^5}
         \left( \frac{(1+\mu)^2 -\phi}{\phi}\otr[\bSigma_0\bB] + \otr[\bSigma_0]\otr[\bB]\right) \label{eq:thm:stationary-point-cov-shift-eq-1}\\
         &\qquad + \left(\frac{2\lambda\phi\mu}{(1+\mu)^5} - \frac{2\lambda\phi\mu^2}{(1+\mu)^6} \right)\left(  \otr[\bSigma_0\bB] -\otr[\bSigma_0]\otr[\bB] \right)  \label{eq:thm:stationary-point-cov-shift-eq-2}\\
         &\qquad + \frac{2\phi\mu^3}{(1+\mu)^6} \left(  \otr[\bSigma_0\bB] -\otr[\bSigma_0]\left(\otr[\bB] + \frac{(1+\mu)^3}{\mu^3}\sigma^2\right)\right) \label{eq:thm:stationary-point-cov-shift-eq-3}.
    \end{align}
    Note that when $\phi>1$, $\mu(\lambda,\phi)>0$ from \Cref{lem:fixed-point-v-lambda-properties}.
    Under the alignment condition, we have 
    \begin{align*}
        \otr[\bSigma_0\bB] > \otr[\bSigma_0]\left(\otr[\bB] + \frac{(1+\mu(0,\phi))^3}{\mu(0,\phi)^3}\sigma^2\right) \geq \otr[\bSigma_0]\otr[\bB].
    \end{align*}    
    Thus, for $\lambda\geq 0$, the second term \eqref{eq:thm:stationary-point-cov-shift-eq-2} is non-negative, and the third term \eqref{eq:thm:stationary-point-cov-shift-eq-3} is strictly positive.
    When $\lambda\geq 0$, from the fixed-point equation \eqref{eq:basic-ridge-equivalence-mu-fixed-point} we have $\phi = (1+\mu) - \lambda(1+\mu)/\mu\leq 1+\mu$ and $(1+\mu)^2-\phi\geq \phi (1+\mu) - \phi=\phi\mu$.
    Then, we know that the first term \eqref{eq:thm:stationary-point-cov-shift-eq-1} is non-negative.
    Therefore, it follows that for all $\lambda\geq 0$,
    \begin{align*}
        (1 - q_v(\bSigma, \bSigma))^2\sR'(\lambda,\phi) > 0.
    \end{align*}
    This implies that $\sR(\lambda,\phi)$ is minimized at $\lambda<0$.

    \paragraph{Part (3) Overparameterized regime and $\bSigma_0=\bI$.}
    When $\bSigma_0=\bI$, the above derivative in Part (1) becomes
    \begin{align}
        &(1 - q_v(\bSigma, \bSigma))^2\sB'(\lambda, \phi) \notag \\
        &= \frac{\lambda }{\mu}\left(q_v'(\bI, \bSigma) q_b(\bSigma, \bB) + q_v(\bI, \bSigma) q_b'(\bSigma, \bB) +  \frac{\lambda}{\mu} q_b'(\bI, \bB) + 2\mu q_b'(\bI, \bB)q_v(\bSigma, \bI)\right) \notag \\
        &\quad+ [\mu q_v'(\bI, \bSigma) q_b(\bSigma, \bB) q_v(\bSigma, \bI) + \mu q_v(\bI, \bSigma) q_b'(\bSigma, \bB)  q_v(\bSigma, \bI) \label{eq:bias_derivative_estim_risk_main_term_1} \\
        &\qquad + q_v(\bI, \bSigma) q_b(\bSigma, \bB) q_v'(\bSigma, \bSigma) + \mu^2 q_b'(\bI, \bB) q_v(\bSigma, \bI)^2 ] \label{eq:bias_derivative_estim_risk_main_term_2}
    \end{align}
    Recalling the definitions of $q_b(\cdot, \cdot)$ and $q_v(\cdot, \cdot)$ from \eqref{eq:def_q_b} and \eqref{eq:def_q_v}, observe that 
    \begin{align*}        
        q_b(\bI, \bB) &= \mu^2 \otr[(\bSigma + \mu \bI)^{-2}\bB] \\
        q_b'(\bI, \bB) &= 2\mu \otr[(\bSigma + \mu \bI)^{-2}\bB] - 2\mu^2\otr[(\bSigma + \mu \bI)^{-3}\bB] = \frac{2}{\mu} q_b(\bI, \bB) - 2 \mu^2 \otr[(\bSigma + \mu \bI)^{-3} \bB]\\
        &= 2\mu\otr[(\bSigma+\mu\bI)^{-3}\bSigma\bB] ,\\
        q_v(\bI, \bSigma) &= \phi \otr[(\bSigma + \mu \bI)^{-2} \bSigma]  \\
        q_v'(\bI, \bSigma) &= - 2\phi \otr[(\bSigma + \mu \bI)^{-3}\bSigma] .
    \end{align*}
    Next, we work on terms \eqref{eq:bias_derivative_estim_risk_main_term_1} and \eqref{eq:bias_derivative_estim_risk_main_term_2} that do not involve a factor of $\lambda / \mu$.
    \begin{align*}
        & \mu q_v'(\bI, \bSigma) q_b(\bSigma, \bB)   q_v(\bSigma, \bI) \\
        &\quad+ \mu q_v(\bI, \bSigma) q_b'(\bSigma, \bB)  q_v(\bSigma, \bI) + q_v(\bI, \bSigma) q_b(\bSigma, \bB) q_v'(\bSigma, \bSigma) \\
        &\quad + \mu^2 q_b'(\bI, \bB) q_v(\bSigma, \bI)^2  \\
        &= -2\mu\phi \otr[(\bSigma+\mu \bI)^{-3}\bSigma]  q_b(\bSigma, \bB)   q_v(\bSigma, \bI) \\
        & \quad + \mu\left(\frac{2}{\mu}q_b(\bSigma,\bB) - 2\mu^2\otr[(\bSigma+\mu\bI)^{-3}\bSigma\bB]\right)  q_v(\bI, \bSigma)  q_v(\bSigma, \bI)  \\
        &\quad - 2\phi\otr[(\bSigma+\mu\bI)^{-3}\bSigma^2] q_v(\bI, \bSigma) q_b(\bSigma, \bB)  \\
        &\quad + \mu^2 \left(\frac{2}{\mu}q_b(\bI,\bB) - 2\mu^2\otr[(\bSigma+\mu\bI)^{-3}\bB]\right) q_v(\bSigma, \bI)^2\\
        &=2\mu\phi\left( -\otr[(\bSigma+\mu \bI)^{-3}\bSigma] + \frac{1}{\mu}\otr[(\bSigma+\mu \bI)^{-2}\bSigma] - \otr[(\bSigma+\mu\bI)^{-3}\bSigma^2] \right) q_b(\bSigma, \bB)   q_v(\bSigma, \bSigma) \\
        &\quad + 2\mu^2\left(\frac{1}{\mu}q_b(\bI,\bB) - \mu^2\otr[(\bSigma+\mu\bI)^{-3}\bB]
        - \mu\otr[(\bSigma+\mu\bI)^{-3}\bSigma\bB] 
        \right) q_v(\bSigma, \bI)^2 \\
        &= 0.
    \end{align*}
    This implies that $\sB'(0,\phi) = 0$ and
    \begin{align*}
        &(1 - q_v(\bSigma, \bSigma))^2\sB'(\lambda, \phi) \notag \\
        &= \frac{\lambda }{\mu}\left(q_v'(\bI, \bSigma) q_b(\bSigma, \bB) + q_v(\bI, \bSigma) q_b'(\bSigma, \bB)  + 2\mu q_b'(\bI, \bB)q_v(\bSigma, \bI)\right) +  \frac{\lambda^2}{\mu^2} q_b'(\bI, \bB)\\
        &= \frac{\lambda }{\mu}\left(q_v'(\bI, \bSigma) q_b(\bSigma, \bB) + q_v(\bI, \bSigma) q_b'(\bSigma, \bB) + 2\mu q_b'(\bI, \bB)q_v(\bSigma, \bI)\right) +  \frac{\lambda^2}{\mu^2} q_b'(\bI, \bB) \\
        &= \frac{\lambda }{\mu}\bigg(
        \frac{2}{\mu} \left(-\mu\phi \otr[(\bSigma+\mu \bI)^{-3}\bSigma]     + q_v(\bSigma,\bI )\right) q_b(\bSigma,\bB) 
         -2\mu^2\otr[(\bSigma+\mu\bI)^{-3}\bSigma\bB] q_v(\bSigma, \bI) \\
        &\qquad  +4\mu^2\otr[(\bSigma+\mu\bI)^{-3}\bSigma\bB]  q_v(\bSigma, \bI)\bigg) +  \frac{\lambda^2}{\mu^2} q_b'(\bI, \bB) \\
        &= \frac{\lambda }{\mu}\left(
        \frac{2}{\mu}\phi \otr[(\bSigma+\mu \bI)^{-3}\bSigma^2] 
         q_b(\bSigma,\bB) 
         + 2\mu^2\otr[(\bSigma+\mu\bI)^{-3}\bSigma\bB] q_v(\bSigma, \bI) \right) +  \frac{\lambda^2}{\mu^2} q_b'(\bI, \bB)  \\
        &= 2\phi\lambda \left(
         \otr[(\bSigma+\mu \bI)^{-3}\bSigma^2] \otr[(\bSigma+\mu \bI)^{-2}\bSigma\bB]
         + \mu\otr[(\bSigma+\mu\bI)^{-3}\bSigma\bB] \otr[(\bSigma+\mu\bI)^{-2}\bSigma] \right)\\
         &\qquad +  2\lambda\left(1 - \phi\otr[(\bSigma+\mu\bI)^{-1}\bSigma]\right)\otr[(\bSigma+\mu\bI)^{-3}\bSigma\bB] \\
         &= 2\phi\lambda \left(
         \otr[(\bSigma+\mu \bI)^{-3}\bSigma^2] \otr[(\bSigma+\mu \bI)^{-2}\bSigma\bB]
         -\otr[(\bSigma+\mu\bI)^{-3}\bSigma\bB] \otr[(\bSigma+\mu\bI)^{-2}\bSigma^2] \right)\\
         &\qquad +  2\lambda \otr[(\bSigma+\mu\bI)^{-3}\bSigma\bB] \\
         &= 2\phi\lambda 
         \otr[(\bSigma+\mu \bI)^{-3}\bSigma^2] \otr[(\bSigma+\mu \bI)^{-2}\bSigma\bB]
         \\
         &\qquad +  2\lambda \otr[(\bSigma+\mu\bI)^{-3}\bSigma\bB] (1 - \phi\otr[(\bSigma+\mu\bI)^{-2}\bSigma^2]) .
    \end{align*}
    From \Cref{lem:ridge-fixed-point-v-properties} \ref{lem:ridge-fixed-point-v-properties-item-vv-properties}, we know that $1 - \phi\otr[(\bSigma+\mu\bI)^{-2}\bSigma^2]\geq 0$.
    When $\phi>1$, it then follows that $\sB'$ is strictly negative for all $\lambda\in[\lambda_{\min}(\phi),0)$ and strictly positive for all $\lambda> 0$ because $\mu(\lambda,\phi)> 0$.
        
    For the variance term, from \Cref{lem:ood-risk-deri}, we have
    \begin{align*}
        &(1 - q_v(\bSigma, \bSigma))^2\sV'(\lambda,\phi) \\
        &= \sigma^2[ q_v'(\bI, \bSigma)(1 - q_v(\bSigma, \bSigma)) + q_v(\bI, \bSigma) q_v'(\bSigma, \bSigma)]\\
        &= \sigma^2\left[ q_v'(\bI,\bSigma) \left(\frac{\lambda}{\mu} + \mu q_v(\bSigma, \bI) \right) + q_v(\bI, \bSigma) q_v'(\bSigma, \bSigma)\right]\\
        &= - 2\phi\sigma^2\left[
          \otr[(\bSigma + \mu \bI)^{-3}\bSigma] \left(\frac{\lambda}{\mu} + \mu q_v(\bSigma, \bI) \right) + \otr[(\bSigma + \mu \bI)^{-3}\bSigma^2] q_v(\bI, \bSigma) 
         \right]\\
         &= - 2\phi\sigma^2\left[
          \frac{\lambda}{\mu}\otr[(\bSigma + \mu \bI)^{-3}\bSigma] +
           \otr[(\bSigma + \mu \bI)^{-2}\bSigma] q_v(\bSigma,\bI) 
         \right] \\
         &= - 2\phi\sigma^2\left[
          \left(1 - \phi\otr[(\bSigma+\mu\bI)^{-1}\bSigma]\right)\otr[(\bSigma + \mu \bI)^{-3}\bSigma] +
           \otr[(\bSigma + \mu \bI)^{-2}\bSigma] q_v(\bSigma,\bI) 
         \right]\\
         &= - 2\phi\sigma^2\left[
         \otr[(\bSigma + \mu \bI)^{-3}\bSigma] + \phi(\otr[(\bSigma + \mu \bI)^{-2}\bSigma]^2 -         \otr[(\bSigma+\mu\bI)^{-1}\bSigma] \otr[(\bSigma + \mu \bI)^{-3}\bSigma]) 
         \right],
    \end{align*}
    which is strictly negative for all $\lambda\geq \lambda_{\min}(\phi)$.

    Combining the above two derivatives, we conclude that $\lambda^* > 0$ in this case.
\end{proof}

\subsection{Proof of \Cref{thm:stationary-point-sig-shift}}
\label{sec:thm:stationary-point-sig-shift-proof}

\bigskip

\ThmStationaryPointSigShift*

\begin{proof}
    We split the proof into two parts.
    
    \paragraph{Part (1) Underparameterized regime.}
    From the proof of \Cref{thm:stationary-point-no-shift}, we know that when $\bSigma_0=\bSigma$, the bias term satisfies that 
    \[\min_{\lambda\in[\lambda_{\min}(\phi),0]}\sB(\lambda,\phi) \geq \sB(0,\phi) = 0,\]
    from \eqref{eq:stationary-eq-2}.
    From \Cref{lem:ood-risk-deri}, the excess bias term has the following derivative:
    \begin{align*}
        \sE'(\lambda, \phi) 
        &= - 2  \bbeta^\top \bSigma(\bSigma + \mu\bI)^{-2} \bSigma ({\bbeta}_0-\bbeta)= - 2  \bbeta^\top \bSigma(\bSigma + \mu\bI)^{-2} \bSigma (\bbeta - {\bbeta}_0) ,
    \end{align*}
    which is zero when $\lambda=0$ because $\mu(0,\phi)=0$ when $\phi<1$, from \Cref{lem:fixed-point-v-lambda-properties}.
    Under the condition that $\bbeta^\top \bSigma(\bSigma + \mu\bI)^{-2} \bSigma (\bbeta - {\bbeta}_0) < 0$ for all $\mu\geq 0$, we have that $\sE(\lambda,\phi)$ is strictly increasing in $\lambda\geq0$.

    From \eqref{eq:var-ind}, we have
    \begin{align*}
         \sV'(\lambda,\phi) &= \frac{- 2\sigma^2 \phi \otr[\bSigma^2 (\bSigma + \mu \bI)^{-3}] }{(1 - q_v(\bSigma, \bSigma))^2}.
    \end{align*}   
    Also, since $\sB\geq 0$ with equality holds if $\lambda=0$.
    Then we have, if
    \begin{align}
        \sS'(\lambda,\phi) + \sV'(\lambda,\phi) = - 2  \bbeta^\top \bSigma(\bSigma + \mu\bI)^{-2} \bSigma (\bbeta - {\bbeta}_0)  + \frac{- 2\sigma^2 \phi \otr[\bSigma^2 (\bSigma + \mu \bI)^{-3}] }{(1 - q_v(\bSigma, \bSigma))^2} > 0 \label{eq:sig-shift-gen-cond}
    \end{align}
    for all $\mu\geq \mu(0,\phi)$, then $\lambda^*<0$.
    Note that when $\lambda\rightarrow+\infty$, $\mu\rightarrow+\infty$ and the denominator of the second term tends to one, so we have $\sV'(\lambda,\phi) \asymp \mu^{-3}$.
    On the other hand, the first term scales as $\sS'(\lambda,\phi) \asymp\mu^{-2}$.
    Eventually, the first term dominates.
    Thus, the condition \eqref{eq:sig-shift-gen-cond} could hold when $\sS'(\lambda,\phi)$ is positive and large enough.
     
    Especially, when $\sigma^2=o(1)$ and the assumed alignment condition is met, it follows that $\sR(\lambda,\phi) = \sB(\lambda,\phi) + \sV(\lambda,\phi) + \sE(\lambda,\phi)$ is minimized over $\lambda<0$.
     
    \paragraph{Part (2) Overparameterized regime.}
    From the proof of \Cref{thm:stationary-point-no-shift}, we know that when \eqref{eq:align-cond}, $\sB(\lambda,\phi)+\sV(\lambda,\phi)$ is minimized at $\lambda<0$.
    Under the condition that $\bbeta^\top \bSigma(\bSigma + \mu\bI)^{-2} \bSigma (\bbeta - {\bbeta}_0) \leq 0$ for all $\mu>0$, we have $\sE'(\lambda, \phi)\geq 0$ over $\lambda \in (\lambda_{\min}(\phi),+\infty)$.
    This implies that $\sE(\lambda, \phi)$ is increasing over $\lambda \in (\lambda_{\min}(\phi),+\infty)$.
    Combining the two results, we further see that the risk $\sR(\lambda,\phi)$ is minimized at $\lambda<0$.
\end{proof}

\subsection{Helper Lemmas}
\label{sec:proofs-sec:optimal_regularization_signs-helper-lemmas}

\begin{lemma}[Out-of-distribution risk derivatives]\label{lem:ood-risk-deri}
    Under the same conditions as in \Cref{prop:ood-risk-asymptotics}, we have
    \begin{equation}
        \frac{\partial \sR(\lambda, \phi)}{\partial \lambda}
        = 
        \left( \sB'(\lambda, \phi) + \sV'(\lambda, \phi) + \sE'(\lambda, \phi) \right)
        \frac{\partial \mu}{\partial \lambda},
    \end{equation}
    where
    \begin{align*}
        \sB'(\lambda, \phi)
        &:= \frac{\partial \sB(\lambda, \phi)}{\partial \mu}
        = \frac{1}{(1 - q_v(\bSigma, \bSigma))^2}\Big[q_v'(\bSigma_0, \bSigma) q_b(\bSigma, \bB) (1 - q_v(\bSigma, \bSigma))\\
        &\qquad \qquad \qquad \quad + q_v(\bSigma_0, \bSigma) q_b'(\bSigma, \bB) (1-q_v(\bSigma, \bSigma)) + q_v(\bSigma_0, \bSigma) q_b(\bSigma, \bB) q_v'(\bSigma, \bSigma) \\
        &\qquad \qquad \qquad \qquad + q_b'(\bSigma_0, \bB)(1 - q_v(\bSigma, \bSigma))^2 \Big]\\
        \sV'(\lambda, \phi)
        &:= \frac{\partial \sV(\lambda, \phi)}{\partial \mu} 
        = \sigma^2 \frac{q'_v(\bSigma_0, \bSigma) - q'_v(\bSigma_0, \bSigma) q_v(\bSigma, \bSigma) + q'_v(\bSigma, \bSigma) q_v(\bSigma_0, \bSigma)}{(1 - q_v(\bSigma, \bSigma))^2} \\
        \sE'(\lambda, \phi) 
        &:= \frac{\partial \sE(\lambda, \phi)}{\partial \mu}
        = -\frac{2}{\mu^2}q_b(\bSigma_0,(\bB-\bB_0)\bSigma).
    \end{align*}
    and $\mu = 1/v_p(\lambda,\phi)$.
    Furthermore, we have $\sV'(\lambda, \phi)<0$ for all $\lambda\in(\lambda_{\min}(\phi),+\infty)$.
\end{lemma}
\begin{proof}
    We split the proof into different parts.

    \paragraph{Part (1) Bias term.}
    Recall the expression for bias for out-of-distribution squared risk:
    \begin{align*}
        \sB(\lambda,\phi)  &= \frac{q_v(\bSigma_0, \bSigma)}{1 - q_v(\bSigma, \bSigma)} q_b(\bSigma, \bB) + q_b(\bSigma_0, \bB) 
    \end{align*}

    Define
    \begin{align}
        h_1(\mu) &=  q_v(\bSigma_0, \bSigma) ,\qquad 
        h_2(\mu) = \frac{q_b(\bSigma, \bB)}{1 - q_v(\bSigma, \bSigma)} ,\qquad h_3(\mu) = q_b(\bSigma_0, \bB) .
    \end{align}
    Then we have
    \begin{align*}
        \sB(\lambda, \phi) &= h_1(\mu) h_2(\mu) + h_3(\mu),
    \end{align*}
    and
    \begin{align*}
        \sB'(\lambda, \phi) 
        &= h_1'(\mu) h_2(\mu) + h_1(\mu) h_2'(\mu) + h_3'(\mu)\\
        &= q_v'(\bSigma_0, \bSigma)\frac{q_b(\bSigma, \bB)}{1 - q_v(\bSigma, \bSigma)}  + q_v(\bSigma_0, \bSigma)\frac{q_b'(\bSigma, \bB)(1 - q_v(\bSigma, \bSigma)) + q_b(\bSigma, \bB) q_v'(\bSigma, \bSigma) }{(1 - q_v(\bSigma, \bSigma))^2} + q_b'(\bSigma_0, \bB)\\
        &= \frac{1}{(1 - q_v(\bSigma, \bSigma))^2} \Big[ q_v'(\bSigma_0, \bSigma) q_b(\bSigma, \bB) (1 - q_v(\bSigma, \bSigma))\\
        &\quad+ q_v(\bSigma_0, \bSigma) q_b'(\bSigma, \bB) (1-q_v(\bSigma, \bSigma)) + q_v(\bSigma_0, \bSigma) q_b(\bSigma, \bB) q_v'(\bSigma, \bSigma) \\
        &\quad + q_b'(\bSigma_0, \bB)(1 - q_v(\bSigma, \bSigma))^2 \Big].
    \end{align*}
    
    \paragraph{Part (2) Variance term.} Recall that the variance term is given by:
    \[
        \sV(\lambda,\phi) = \sigma^2 \frac{q_v(\bSigma_0, \bSigma)}{1 - q_v(\bSigma, \bSigma)}.
    \]
    The derivative in $\mu$ is:
    \begin{align*}
        \sV'
        &= \sigma^2
        \frac{1}{(1 - q_v(\bSigma, \bSigma))^2}
        \{ q'_v(\bSigma_0, \bSigma) (1 - q_v(\bSigma, \bSigma)) + q_v(\bSigma_0, \bSigma) q'_v(\bSigma, \bSigma)  \}.
    \end{align*}
    Note that
    \begin{align*}
        q'_v(\bSigma, \bSigma) & = -2\phi\otr[\bSigma^2(\bSigma+\mu\bI)^{-2}] < 0\\
        q_v'(\bSigma_0,\bSigma) &= -2\phi\otr[\bSigma_0\bSigma(\bSigma+\mu\bI)^{-2}] = -2\phi\otr[(\bSigma+\mu\bI)^{-1}\bSigma^{1/2}\bSigma_0\bSigma^{1/2}(\bSigma+\mu\bI)^{-1}]<0.
    \end{align*}
    From \Cref{lem:ridge-fixed-point-v-properties}, we also have $q_v(\bSigma,\bSigma) > 0$ and $1-q_v(\bSigma,\bSigma) > 0$ when $\lambda\in(\lambda_{\min}(\phi),+\infty)$.
    Therefore, it holds that
    \begin{align*}
        \sV'(\lambda,\phi) <0.
    \end{align*}

    \paragraph{Part (3) Extra bias term.} Recall that the extra bias term is given by
    \[\sE(\lambda, \phi) = - 2\bbeta^\top (v\bSigma + \bI)^{-1} \bSigma_0 (\bbeta - {\bbeta}_0) = - 2\mu\bbeta^\top (\bSigma + \mu\bI)^{-1} \bSigma_0 (\bbeta - {\bbeta}_0) = l(\bSigma_0,(\bB-\bB_0)\bSigma). \]
    Then, the derivative is given by:
    \begin{align*}
        \sE'(\lambda, \phi) &= -2 \bbeta^\top \bSigma(\bSigma + \mu\bI)^{-2} \bSigma_0 (\bbeta - {\bbeta}_0) 
        = -\frac{2}{\mu^2}q_b(\bSigma_0,(\bB-\bB_0)\bSigma).  
    \end{align*}
\end{proof}

\section{Proofs in \Cref{sec:optimal_risk_monotonicity}}
\label{sec:proofs-sec:optimal_risk_monotonicity}

\subsection{Proof of \Cref{lem:optimal-risk-iso-signal}}
\label{sec:lem:optimal-risk-iso-signal-proof}

\bigskip

\LemOptimalRiskIsoSignal*

\begin{proof}
    When $\bbeta=\bbeta_0$ and $\bbeta\bbeta^{\top}\asympequi\alpha^2\bI$, from \eqref{eq:identity-qv}, we have that
    \begin{align*}
        \sB(\lambda,\phi)  &= \alpha^2\frac{q_v(\bSigma_0, \bSigma)}{1 - q_v(\bSigma, \bSigma)} q_b(\bSigma, \bI) + \alpha^2q_b(\bSigma_0, \bI)  \\
        &= \alpha^2\frac{\mu^2}{\phi} \left(\frac{q_v(\bSigma_0, \bSigma)}{1 - q_v(\bSigma, \bSigma)} q_v(\bSigma, \bI) + q_v(\bSigma_0, \bI) \right)\\
        &= \alpha^2\frac{\mu}{\phi} \left(\frac{q_v(\bSigma_0, \bSigma)}{1 - q_v(\bSigma, \bSigma)} \left(1 - q_v(\bSigma, \bSigma) - \frac{\lambda}{\mu} \right) + \mu q_v(\bSigma_0, \bI) \right) \\
        &= \alpha^2\frac{\mu}{\phi} \big[q_v(\bSigma_0, \bSigma)  + \mu q_v(\bSigma_0, \bI)\big] - \alpha^2\frac{\lambda}{\phi} \frac{q_v(\bSigma_0, \bSigma)}{1 - q_v(\bSigma, \bSigma)} \\
        &=\alpha^2 \mu\otr[\bSigma_0(\bSigma+\mu\bI)^{-1}] - \alpha^2\frac{\lambda}{\phi} \frac{q_v(\bSigma_0, \bSigma)}{1 - q_v(\bSigma, \bSigma)}. 
    \end{align*}
    Therefore, from \Cref{prop:ood-risk-asymptotics} and \Cref{thm:stationary-point-sig-shift}, the optimal risk is given by
    \begin{align*}
        R(\hat{\bbeta}^{\lambda^*}) &\asympequi\sB_p(\lambda^*, \phi) + \sV_p(\lambda^*, \phi) + \sigma_0^2\\
        &= \alpha^2 \mu^*\otr[\bSigma_0(\bSigma+\mu^*\bI)^{-1}] + \left(\sigma^2- \frac{\alpha^2\lambda^*}{\phi}\right)\frac{q_v(\bSigma_0, \bSigma) }{1-q_v(\bSigma, \bSigma)} + \sigma_0^2\\
        &=  \alpha^2 \mu^*\otr[\bSigma_0(\bSigma+\mu^*\bI)^{-1}] + \sigma_0^2\\
        &=  \alpha^2 \otr[\bSigma_0(v^*\bSigma+\bI)^{-1}] + \sigma_0^2,
    \end{align*}
    where $\mu^*=\mu(\lambda^*,\phi)$ and $v^*=v(\lambda^*,\phi)$.

    Note that when $\sigma^2$ and $\sigma_0^2$ are fixed, $\sR(\lambda^*,\phi) = \phi\sigma^2 \cdot (\mu^*/\lambda^*) \cdot \otr[\bSigma_0(\bSigma+\mu^*\bI)^{-1}]+\sigma_0^2$ is simply a function of $\lambda^*$.
    From \Cref{lem:fixed-point-v-lambda-properties} \ref{lem:fixed-point-v-lambda-properties-item-monotonicity-lam-mu}, we have that $\mu^*/\lambda^*$ is strictly decreasing in $\lambda^*$.
    Also note that $\otr[\bSigma_0(\bSigma+\mu^*\bI)^{-1}]$ is strictly decreasing in $\lambda^*$.
    Thus, we know that $\sR(\lambda^*,\phi) $ is strictly decreasing in $\lambda^*$ when $\phi,\sigma,\sigma_0$ are fixed.
    Because $\lambda^*$ is strictly decreasing in $\SNR$, we further find that $\sR(\lambda^*,\phi) $ is strictly increasing in $\SNR$.
\end{proof}

\subsection{Proof of \Cref{thm:monotonicty}}
\label{sec:thm:monotonicity-proof}

\bigskip

\ThmMonotonicty*

\begin{proof}
    We split the proof into different parts.

    \paragraph{Part (1) Risk characterization and equivalence.}
    From the proof of \Cref{thm:negative-lam}, we have
     \begin{align*}
        R(\hat{\bbeta}^{\lambda}) &\asympequi \sR_p(\lambda, \phi,\phi)
    \end{align*}
    where $\sR_p$ is defined in \eqref{eq:risk-det-equiv-ensemble}.
    Furthermore, for any $\bar{\psi}\in[\phi,+\infty]$, there exists a unique $\bar{\lambda}\geq \lambda_{\min}(\phi)$ defined through \eqref{eq:lambda_bar}.
    For any pair of $(\lambda_1, \psi_1)$ and $(\lambda_2, \psi_2)$ on the path
    \smash{$\cP(\bar{\lambda}; \phi, \bar{\psi})$} as defined in \eqref{eq:path}, we have that:
    \[\sR_p(\lambda_1;\phi,\psi_1) = \sR_p(\lambda_2;\phi,\psi_2). \]
    
    \paragraph{Part (2) Risk monotonicity.}
    From \Cref{lem:ridge-fixed-point-v-properties} (3), we know that the denominator of $\tv_p$ defined in \eqref{eq:tv_Sigma0} is non-negative:
    \[v_p(\lambda, \psi)^{-2}-\phi \int r^2 (1+v_p(\lambda, \psi)r)^{-2}\rd H_p(r)\geq v_p(\lambda, \psi)^{-2}-\psi \int r^2 (1+v_p(\lambda, \psi)r)^{-2}\rd H_p(r)\geq 0\] 
    for $\lambda\geq \lambda_{\min}(\phi)$.
    Therefore, $\sR_p(\lambda, \phi,\psi)$ is increasing in $\phi$ for any fixed $(\lambda,\psi)$.
    Furthermore, since $\sR_p(\lambda, \phi,\psi)$ is a continuous function of $\phi$ and $v(\lambda;\psi)$, it follows that for $0<\phi_1\leq \phi_2<\infty$,
    \begin{align*}
        \min_{\psi \ge \phi_1}
        \sR_p(\lambda, \phi_1,\psi)
        &\le
        \min_{\psi \ge \phi_2}
        \sR_p(\lambda, \phi_1,\psi) \le
        \min_{\psi \ge \phi_2}
        \sR_p(\lambda, \phi_2,\psi),
    \end{align*}
    where the first inequality follows because $\{ \psi: \psi \ge \phi_1 \} \supseteq \{ \psi : \psi \ge \phi_2 \}$, and the second inequality follows because $\sR_p(\lambda, \phi,\psi)$ is increasing in $\phi$ for a fixed $\psi$.
    Thus, $\min_{\psi \ge \phi}\sR_p(\lambda, \phi,\psi)$ is a continuous and monotonically increasing function in $\phi$.

    \paragraph{Part (3) Optimal subsampling and regularization.}
    Similar to the proof of Part (3) in \Cref{thm:monotonicty-gen}, we have
    \begin{align*}
        \min_{\psi\geq\phi} \sR_p(0;\phi,\psi) \asympequi \min_{\lambda\geq\lambda_{\min}(\phi)+\epsilon}R(\hat{\bbeta}^{\lambda}).
    \end{align*}
    From Part (3), we know that the former is continuous and monotonically increasing in $\phi$, which finishes the proof.  

    \paragraph{Part (4) Monotonicity in signal-to-noise ratio.}
    When $\bbeta_0=\bbeta$, the extra bias term $\sS$ is zero.
    When $\sigma^2$ is fixed, note that $\sB$ and $\kappa$ are strictly increasing in $\alpha^2$ while $\sV$ does not depend on $\alpha^2$.
    Thus, we know that $\sR_p$ is strictly increasing in $\SNR$.
    Consequently, $\min_{\psi\geq\phi} \sR_p(0;\phi,\psi)$ is strictly increasing in $\SNR$.
\end{proof}

\subsection{Proof of \Cref{thm:nonmonotonicity}}
\label{sec:thm:nonmonotonicity-proof}

\bigskip

\ThmNonmonotonicity*

\begin{proof}
    For isotropic features $\bSigma=\bI$, analogous to the proof of \Cref{thm:monotonicty}, the excess prediction risk is given by
    \begin{align*}
        R\big(\hbeta^{\lambda}\big) \asympequi R_p(\lambda, \phi,\phi)
    \end{align*}
    where $R_p(\lambda, \phi,\psi)$ is defined in \eqref{eq:risk-det-equiv}.
    When $\bSigma=\bI$, the non-negative constants $\tc_p(\lambda, \psi)$ and $\tv_p(\lambda, \phi,\psi)$ are defined through the following equations:
    \begin{align*}
        v_p(\lambda, \phi) &= \frac{\sqrt{(\phi+\lambda-1)^2+4 \lambda}-(\phi+\lambda-1)}{2\lambda},\\
        \tv_p(\lambda, \phi,\psi) &= \ddfrac{\phi \frac{1}{(1+v_p(\lambda, \psi))^2}}{v_p(\lambda, \psi)^{-2}-\phi \frac{1}{(1+v_p(\lambda, \psi))^2}},\\
         \tc_p(\lambda, \psi) &=
         (v_p(\lambda, \psi)\bSigma+\bI)^{-2} \alpha^2.
    \end{align*}
    
    From Theorem 1 of \citet{patil2023generalized}, for all $(\lambda, \phi)\in(0,\infty)^2$, there exists $\psi = \psi(\lambda,\phi)$ such that the prediction risk \eqref{eq:generalized-risk} of the full-ensemble estimator are asymptotically equivalent:
    \begin{align*}
        R_p(\lambda, \phi,\phi)~\asympequi~R_p(0;\phi,\psi).
    \end{align*}
    Note that the left-hand side is simply the risk of the ridge predictor with ridge penalty $\lambda$.    
    Furthermore, from \eqref{eq:lambda_bar}, it holds that
    \begin{align*}
        \lambda + \phi \int \frac{r}{1 + v_p(0;\psi) r} \rd H(r) = \psi \int \frac{r}{1 + v_p(0;\psi) r} \rd H(r).
    \end{align*}
    Taking the derivative with respect to $\phi$ on both sides yields that
    \begin{align}
        \frac{\partial \psi}{\partial \phi} & = 1 - \frac{\lambda v_p(0;\psi)}{1 - \lambda v_p(0;\psi)} \tv_p(0;\phi,\psi).
    \end{align}

    We consider three cases below.
    
    \paragraph{(1) $\alpha^2=0$ and $\sigma^2>0$.}
    In this case, the excess risk equals the variance component:
    \begin{align*}
        R_p'(\lambda, \phi,\phi) &= \sigma^2 \tv_p(\lambda, \phi,\phi) \\
        &= \sigma^2 \left( \ddfrac{1}{1-\phi \int\left(\frac{v_p(\lambda, \phi)r}{1+v_p(\lambda, \phi)r}\right)^2\rd H_p(r)} - 1\right).
    \end{align*}
    Let $f(\phi) := \phi \int\left({v_p(\lambda, \phi)r}/{(1+v_p(\lambda, \phi)r)}\right)^2\rd H_p(r)$. Then, the monotonicity of the above display in $\phi$ is the same as that of $f$ in $\phi$.    
    Note that
    \begin{align}
        f(\phi)
        &= \phi \left(\frac{\sqrt{(\phi+\lambda-1)^2+4 \lambda}-(\phi+\lambda-1)}{\sqrt{(\phi+\lambda-1)^2+4 \lambda}-(\phi-\lambda-1)}\right)^2.
    \end{align}
    Taking the derivative with respect to $\phi$, we have
    \begin{align*}
        f'(\phi)
        &=\frac{(-\phi+\lambda+1) \left(\sqrt{(\phi+\lambda-1)^2+4
        \lambda}-\phi-\lambda +1\right)^2}{\sqrt{(\phi+\lambda-1)^2+4 \lambda} \left(\sqrt{(\phi+\lambda-1)^2+4 \lambda}+\lambda-\phi+1\right)^2}.
    \end{align*}
    Since $f'(\phi)>0$ when $\phi\in(0,\lambda+1)$ and $f'(\phi)<0$ when $\phi\in(\lambda+1,+\infty)$, we know that $R_p$ is strictly increasing over $\phi\in(0,\lambda+1)$ and strictly decreasing over $\phi\in(\lambda+1,+\infty)$.
    Thus, the monotonicity of $R_p(\lambda, \phi,\phi)$ follows.
    
    \paragraph{(2) $\alpha^2>0$ and $\sigma^2=0$.}
    In this case, the excess risk equals the bias component:
    \begin{align*}
        R_p(\lambda, \phi,\phi) &= \tc_p (-\lambda ;\phi)( \tv_p(\lambda, \phi,\phi) + 1)\\
        &= \alpha^2 \ddfrac{\frac{1}{(1 + v_p(\lambda, \phi))^2}}{ 1 - \phi\frac{1}{(1 + v_p(\lambda, \phi))^2}} \\
        &= \frac{\alpha^2 }{(1-\phi )v_p(\lambda, \phi)^2 + 2v_p(\lambda, \phi) +1 }.
    \end{align*}
    Let $g(\phi) = (1-\phi )v_p(\lambda, \phi)^2 + 2v_p(\lambda, \phi) +1 $. 
    Taking the derivative with respect to $\phi$ yields
    \begin{align*}
        g'(\phi)  &= \frac{\left(\sqrt{(\phi+\lambda-1)^2+4 \lambda}-\lambda-\phi+1\right)}{4\lambda^2 \sqrt{(\lambda+\phi-1)^2+4 \lambda}} \\
        &\qquad\cdot \left(\lambda+3 (\phi-1) \sqrt{\lambda^2+2 \lambda (\phi+1)+\phi^2-2 \phi+1}-\lambda^2-4 \lambda (\phi+1)-3 (\phi-1)^2\right)\\
        &=: \frac{\left(\sqrt{(\phi+\lambda-1)^2+4 \lambda}-\lambda-\phi+1\right)}{4\lambda^2 \sqrt{(\lambda+\phi-1)^2+4 \lambda}} h(\phi).
    \end{align*}
    By simple calculations, one can show that $h'(\phi)<0$ and $h(\phi)\leq h(0)<-2\lambda$ when $\lambda>0$.
    Therefore, we have $g'(\phi)<0$ and $R_p$ is strictly increasing over $\phi\in(0,\infty)$.

    \paragraph{(3) General cases when $\alpha^2>0$.}
    Note that
    \begin{align*}
        R_p(\lambda, \phi,\phi) &= \sigma^2 \tv_p(\lambda, \phi,\psi) + \tc_p (-\lambda ;\phi)( \tv_p(\lambda, \phi,\phi) + 1) \\
        &=: f_1(\phi) + f_2(\phi),
    \end{align*}
    where $f_1$ first increases and then decreases in $\phi$, and $f_2$ is a strictly increasing function.
    Note that only $f_1$ depends on $\sigma^2$.
    Because for any $\lambda>0$ and $\phi\in(\lambda+1,\infty)$, $f_1'(\phi)<0$ and its scale is proportional to $\sigma^2$, we have that for all $\epsilon>0$, there exists $\sigma^2>0$ such that $-f_1'(\phi)>f_2'(\phi)+\epsilon$.
    This implies that
    \[\max_{\sigma^2,\phi\in(0,\infty)} \min_{\lambda\geq 0} \frac{\partial \sR(\lambda, \phi)}{\partial \phi} \leq -\epsilon,\]
    which completes the proof.
\end{proof}

\subsection{Proof of \Cref{thm:negative-lam}}
\label{sec:thm:negative-lam-proof}

\bigskip

\ThmNegativeLam*

\begin{proof}
    We split the proof into different parts.
    
    \paragraph{Part (1) Risk characterization.} From \Cref{lem:ood-risk-asymptotics-ensemble}, we have $R(\hat{\bbeta}^{\lambda}_{k,\infty}) \asympequi \sR_p(\lambda, \phi,\psi)$.

    \paragraph{Part (2) Risk equivalence.}
    From \Cref{lem:v-equiv-path}, for any $\bar{\psi}\in[\phi,+\infty]$, there exists a unique $\bar{\lambda}\geq \lambda_{\min}(\phi)$ such that
    \begin{equation}
        \frac{1}{v} = \bar{\lambda} + \phi\int\frac{r}{1 + v r} \rd H_p(r),
        \quad
        \text{and}
        \quad
         \frac{1}{v}  = \bar{\psi}\int\frac{r}{1 + v r}\rd H_p(r). \label{eq:lambda_bar}
    \end{equation}
    For any pair of $(\lambda_1, \psi_1)$ and $(\lambda_2, \psi_2)$ on the path
    \smash{$\cP(\bar{\lambda}; \phi, \bar{\psi})$} as defined in \eqref{eq:path}, we have that:
    \[\sR_p(\lambda_1;\phi,\psi_1) = \sR_p(\lambda_2;\phi,\psi_2). \]

    \paragraph{Part (3) Optimal risk.}
    From \eqref{eq:risk-equiv}, we have that for any $\lambda\geq \lambda_{\min}(\phi)$, there exists $\psi\geq 1$ such that $| \sR_p(\lambda_{\min}(\phi);\phi,\psi) - \sR_p(\lambda, \phi,\phi)| \asto 0$.
    From \Cref{lem:v0} and \Cref{lem:ridge-fixed-point-v-properties}, $1/v_p(-\lambda                       ;\psi) \in [-r_{\min},\infty]$ and 
    \[\lim_{\psi\rightarrow \phi}\sR_p(\lambda_{\min}(\phi);\phi,\psi) = \lim_{\lambda\rightarrow \lambda_{\min}(\phi)^+}\sR_p(\lambda, \phi,\phi) = +\infty.\]
    Similar to the proof of Lemma 28 from \citet{bellec2023corrected}, one can show that the sequence of functions $\{\sR_p(\lambda_{\min}(\phi);\phi,\psi(\lambda)) - \sR_p(\lambda, \phi,\phi)\}_{p\in\NN}$ is uniformly equicontinuous on $\lambda\in\Lambda=[\lambda_{\min}(\phi)+\epsilon,\infty]$ almost surely for some small $\epsilon>0$ such that $\sR_p(\lambda, \phi,\psi)$ is no larger than the null risk $\sR_p(+\infty;\phi,\phi)$ when $\lambda\in[\lambda_{\min}(\phi)+\epsilon,+\infty]$.
    From Theorem 21.8 of \citet{davidson1994stochastic}, it further follows that the sequences converge to zero uniformly over $\Lambda$ almost surely.
    This implies that
    \begin{align*}
        0 & = \limsup_{p} \max_{\lambda\geq\lambda_{\min}(\phi)+\epsilon}\left[ \sR_p(\lambda_{\min}(\phi);\phi,\psi(\lambda)) - \sR_p(\lambda, \phi,\phi) \right]\\
        & \geq \min_{\psi\geq\phi} \limsup_{p} \max_{\lambda\geq\lambda_{\min}(\phi)+\epsilon}\left[ \sR_p(\lambda_{\min}(\phi);\phi,\psi) - \sR_p(\lambda, \phi,\phi) \right]\\
        &\geq \limsup_{p} \min_{\psi\geq\phi}\max_{\lambda\geq\lambda_{\min}(\phi)+\epsilon}\left[ \sR_p(\lambda_{\min}(\phi);\phi,\psi) - \sR_p(\lambda, \phi,\phi) \right] \\
        &=\limsup_{p} \left[\min_{\psi\geq\phi} \sR_p(\lambda_{\min}(\phi);\phi,\psi) - \min_{\lambda\geq\lambda_{\min}(\phi)+\epsilon}\sR_p(\lambda, \phi,\phi) \right].
    \end{align*}
    
    Conversely, since for any $\psi\geq \psi(\lambda_{\min}(\phi)+\epsilon)$, there exists $\lambda\geq \lambda_{\min}(\phi)+\epsilon$ such that $| \sR_p(\lambda_{\min}(\phi);\phi,\psi) - \sR_p(\lambda(\psi);\phi,\phi)| \asto 0$.
    Similarly, we can show that $\{\sR_p(\lambda_{\min}(\phi);\phi,\psi) - \sR_p(\lambda(\psi);\phi,\phi)\}_{p\in\NN}$ is uniformly equicontinuous on $\psi\in\Psi=[\psi(\lambda_{\min}(\phi)+\epsilon),\infty]$ almost surely.
    This also implies that
    \begin{align*}
        0 &= \liminf_{p} \min_{\psi\geq\psi(\lambda_{\min}(\phi)+\epsilon)}[\sR_p(\lambda_{\min}(\phi);\phi,\psi) - \sR_p(\lambda(\psi);\phi,\phi)]\\
        &\leq \max_{\lambda\geq\lambda_{\min}(\phi)} \liminf_{p} \min_{\psi\geq\psi(\lambda_{\min}(\phi)+\epsilon)}[\sR_p(\lambda_{\min}(\phi);\phi,\psi) - \sR_p(\lambda, \phi,\phi) ] \\
        &\leq \liminf_{p} \max_{\lambda\geq\lambda_{\min}(\phi)} \min_{\psi\geq\psi(\lambda_{\min}(\phi)+\epsilon)} [\sR_p(\lambda_{\min}(\phi);\phi,\psi) - \sR_p(\lambda, \phi,\phi) ] \\
        &= \liminf_{p} \left[\min_{\psi\geq\psi(\lambda_{\min}(\phi)+\epsilon)} \sR_p(\lambda_{\min}(\phi);\phi,\psi) - \min_{\lambda\geq\lambda_{\min}(\phi)}\sR_p(\lambda, \phi,\phi)\right].
    \end{align*}
    Combining the previous two inequalities implies that
    \begin{align*}
        \min_{\lambda\geq\lambda_{\min}(\phi)+\epsilon}\sR_p(\lambda, \phi,\phi) \asympequi \min_{\psi\geq\psi(\lambda_{\min}(\phi)+\epsilon)} \sR_p(\lambda_{\min}(\phi);\phi,\psi) .
    \end{align*}
    Since
    \begin{align*}
         \min_{\lambda\geq\lambda_{\min}(\phi)+\epsilon}\sR_p(\lambda, \phi,\phi)  &= \min_{\lambda\geq\lambda_{\min}(\phi)}\sR_p(\lambda, \phi,\phi) \\
        \min_{\psi\geq\psi(\lambda_{\min}(\phi)+\epsilon)} \sR_p(\lambda_{\min}(\phi);\phi,\psi) &= \min_{\psi\geq\phi} \sR_p(\lambda_{\min}(\phi);\phi,\psi),
    \end{align*}
    we further have
    \begin{align*}
        \min_{\lambda\geq\lambda_{\min}(\phi)}\sR_p(\lambda, \phi,\phi) = \min_{\psi\geq\phi} \sR_p(\lambda_{\min}(\phi);\phi,\psi),
    \end{align*}
    which holds for $\phi\in(0,\infty)$.
    This finishes the proof of the second conclusion.

    \paragraph{Part (4) Optimal risk when $\phi< 1$.}
    When $\bbeta_0=\bbeta$, the excess bias term $\sS\equiv0$.
    From \Cref{lem:ood-risk-deri}, we have that the bias component $\tc_p(\lambda, \phi,\phi)$ of the risk equivalent is minimized at $\lambda=0$ when $\phi<1$.
    Since $\tv_p(\lambda, \phi,\phi)$ is a strictly increasing function in $v_p(\lambda, \phi)$ and $v_p(\lambda, \phi)$ is a strictly decreasing function in $\lambda$, we see that $\tv_p(\lambda, \phi,\phi)$ is a strictly decreasing function in $\lambda$.
    Thus, we have that
    \[\min_{\lambda\in[\lambda_{\min},0]}\sR_p(\lambda, \phi,\phi) \geq \sR_p(0;\phi,\phi).\]
    This implies that
    \[\min_{\lambda\geq\lambda_{\min}}\sR_p(\lambda, \phi,\phi) \geq \min_{\lambda\geq 0}\sR_p(0;\phi,\phi),\]
    which finishes the proof of the first conclusion.
\end{proof}

\subsection{Helper Lemmas}
\label{sec:thm:negative-lam-proof-helper-lemmas}

\begin{lemma}[Monotonicity of generalized prediction risk with optimal ridge regularization]\label{thm:monotonicty-gen}
    Suppose \Cref{asm:train-test} hold.
    Define the generalized mean squared~risk for a estimator $\hat{\bbeta}$ as:
    \begin{equation}
        \label{eq:generalized-risk}
        R(\hbeta;\bA, \bb, \bbeta_0) = 
        \|L_{\bA,\bb}(\hbeta - \bbeta_0)\|_2^2,
    \end{equation}
    where \smash{$L_{\bA,\bb}(\bbeta)=\bA\bbeta + \bb$} is a linear functional and \smash{$(\bA,\bb)$} is independent of \smash{$(\bX,\by)$} such that \smash{$\|\bA\|_{\oper}$} and \smash{$\|\bb\|_2$} are almost surely bounded.   
    Then, as $k,n,p\rightarrow\infty$ such that $p/n\rightarrow\phi\in(0,\infty)$ and $p/k\rightarrow\psi\in[\phi,\infty]$, there exists a sequence of random variables $\{\sQ_p(\lambda, \phi,\psi)\}_{p=1}^{\infty}$ that is asymptotically equivalent to the risk of the full-ensemble ridge predictor, 
    \begin{align}
        R(\hbeta^{\lambda}_{k,\infty};\bA, \bb, \bbeta_0) \asympequi \sQ_p(\lambda, \phi,\psi):= \tc_p(\lambda, \phi,\psi,\bA^{\top}\bA)  + \|\fNL\|_{L_2}^2\tv_p(\lambda, \phi,\psi,\bA^{\top}\bA),  \label{eq:R_p}
    \end{align}
    where the non-negative constants $\tc_p(\lambda, \phi,\psi,\bA^{\top}\bA)$ and $\tv_p(\lambda, \phi,\psi,\bA^{\top}\bA)$ are defined through the following equations:
    \begin{align*}
        \frac{1}{v_p(\lambda, \psi)} &= \lambda+\psi \int\frac{r}{1+v_p(\lambda, \psi)r }\rd H_p(r),\\
        \tv_p(\lambda, \phi,\psi,\bA^{\top}\bA) &= \ddfrac{\phi \otr[\bA^{\top}\bA\bSigma(v_p(\lambda, \psi)\bSigma+\bI)^{-2}]}{v_p(\lambda, \psi)^{-2}-\phi \int\frac{r^2}{(1+v_p(\lambda, \psi)r)^2}\rd H_p(r)},\\
        \tc_p(\lambda, \phi,\psi,\bA^{\top}\bA) &=
        \bbeta_0^{\top}(v_p(\lambda, \psi)\bSigma+\bI)^{-1}(\tv_p(\lambda, \phi,\psi,\bA^{\top}\bA)\bSigma+\bA^{\top}\bA) (v_p(\lambda, \psi)\bSigma+\bI)^{-1}\bbeta_0.
    \end{align*}
    For the ridge predictor when $k=n$ and $\lambda_{\min}(\phi)$ defined in \eqref{eq:lam_min}, the optimal risk equivalence $\min_{\lambda\geq \lambda_{\min}(\phi)}\sQ_p(\lambda, \phi,\phi)$ is monotonically increasing in $\phi$.
\end{lemma}
\begin{proof}
    Given an observation $(\bx, y)$, recall the decomposition $y = \fLI(\bx) + \fNL(\bx)$ explained in \Cref{sec:preliminaries}.
    For $n$ i.i.d.\ samples from the same distribution as $(\bx,y)$, we define analogously the vector decomposition:
    \begin{align}
        \label{eq:li-nl-decomposition}
        \by = \bfLI +\bfNL,
    \end{align}
    where $\bfLI=\bX\bbeta_0$ and $\bfNL = [\fNL(\bx_i)]_{i\in[n]}$.

    Note that
    \begin{align*}
        R(\hbeta_{k,\infty}^{\lambda};\bA,\bb, \bbeta_0)
        &= R(\hbeta_{k,\infty}^{\lambda};\bA,\zero, \bbeta_0) +  2\bb^{\top}\bA(\hbeta_{k,\infty}^{\lambda} - \bbeta_0) + \|\bb\|_2^2 .
    \end{align*}
    By Theorem 3 of \citet{patil2023generalized}, the cross term vanishes, that is, $\bb^{\top}\bA(\hbeta_{k,\infty}^{\lambda} - \bbeta_0)\asto 0$.
    We then have
    \begin{align*}
        |R(\hbeta_{k_1,\infty}^{\lambda_1};\bA, \bb, \bbeta_0)  - R(\hbeta_{k_2,\infty}^{\lambda_2};\bA, \bb, \bbeta_0)|  &\asto  |R(\hbeta_{k_1,\infty}^{\lambda_1};\bA, \zero, \bbeta_0)  - R(\hbeta_{k_2,\infty}^{\lambda_2};\bA, \zero, \bbeta_0)|.
    \end{align*}
    Thus, it suffices to analyze $R(\hbeta_{k,\infty}^{\lambda};\bA,\zero, \bbeta_0)$.
    To simplify the notation, we define 
    \begin{align}
        R_p(\lambda, \phi,\psi) := R\big(\hbeta_{\lfloor p/\psi\rfloor,\infty}^{\lambda_1};\bA, \zero_p, \bbeta_0\big) \label{eq:risk-det}
    \end{align}
    to indicate the dependency solely on $p$ and $(\lambda,\phi,\psi)$.
    We split the proof into different parts.

    \paragraph{Part (1) Risk characterization.}
    Under \Cref{asm:train-test}, from Equation (11) of \citet{patil2023generalized}, we have that for $\lambda\geq 0$,
    \[R(\hbeta^{\lambda}_{k,\infty};\bA, \bb, \bbeta_0) \asympequi \sQ_p(\lambda, \phi,\psi)\]
    where $\sQ_p$ is defined in \eqref{eq:R_p}.
    Note that for $\lambda\in[\lambda_{\min},0)$, the fixed-point solution $v_p(\lambda, \psi)$ satisfies the same fixed-point equation as the one for $\lambda\geq 0$.
    Since the above deterministic equivalent depends on $\lambda$ only through $v_p(\lambda, \psi)$, it also applies to $\lambda\in[\lambda_{\min},0)$.
    
    \paragraph{Part (2) Risk equivalence.}
    From \Cref{lem:v-equiv-path}, we have that, for any $\bar{\psi}\in[\phi,+\infty]$, there exists $\bar{\lambda}$ uniquely defined through \eqref{eq:lambda_bar}.
    For any pair of $(\lambda_1, \psi_1)$ and $(\lambda_2, \psi_2)$ on the path
    \smash{$\cP(\bar{\lambda}; \phi, \bar{\psi})$} as defined in \eqref{eq:path}, the generalized risk functionals \eqref{eq:generalized-risk} of the full-ensemble estimator are asymptotically equivalent:
    \begin{equation}\label{eq:risk-equiv}
        R_p(\lambda_1;\phi,\psi_1) \asympequi \sQ_p(\lambda_1;\phi,\psi_1) = \sQ_p(\lambda_2;\phi,\psi_2) \asympequi R_p(\lambda_2;\phi,\psi_2).
    \end{equation}
    
    \paragraph{Part (3) Risk monotonicity.}
    Note that from Lemma F.10 (3) and Lemma F.11 (3) \citet{du2023gcv}, $v_p(\lambda, \psi)^{-2}-\phi \int r^2 (1+v_p(\lambda, \psi)r)^{-2}\rd H_p(r)$ is non-negative.
    Then we have
    \begin{align*}
        \tv_p(\lambda, \phi,\psi,\bA^{\top}\bA) &= \ddfrac{\otr[\bA^{\top}\bA\bSigma(v_p(\lambda, \psi)\bSigma+\bI)^{-2}]}{\phi^{-1}v_p(\lambda, \psi)^{-2}- \int\frac{r^2}{(1+v_p(\lambda, \psi)r)^2}\rd H_p(r)}
    \end{align*}
    is increasing in $\phi\in(0,\psi]$ for any fixed $\psi$.
    Thus, $\sQ_p(\lambda, \phi,\psi)$ is increasing in $\phi$ for any fixed $(\lambda,\psi)$.
    Furthermore, since $\sQ_p(\lambda, \phi,\psi)$ is a continuous function of $\phi$ and $v(\lambda;\psi)$, it follows that for $0<\phi_1\leq \phi_2<\infty$,
    \begin{align*}
        \min_{\psi \ge \phi_1}
        \sQ_p(\lambda, \phi_1,\psi)
        &\le
        \min_{\psi \ge \phi_2}
        \sQ_p(\lambda, \phi_1,\psi) \le
        \min_{\psi \ge \phi_2}
        \sQ_p(\lambda, \phi_2,\psi),
    \end{align*}
    where the first inequality follows because $\{ \psi: \psi \ge \phi_1 \} \supseteq \{ \psi : \psi \ge \phi_2 \}$, and the second inequality follows because $\sQ_p(\lambda, \phi,\psi)$ is increasing in $\phi$ for a fixed $\psi$.
    Thus, $\min_{\psi \ge \phi}\sQ_p(\lambda, \phi,\psi)$ is a continuous and monotonically increasing function in $\phi$.

    \paragraph{Part (4) Optimal subsampling and regularization.}
    From \eqref{eq:risk-equiv}, we have that for any $\lambda\geq \lambda_{\min}(\phi)$, there exists $\psi\geq 1$ such that $| \sQ_p(\lambda_{\min}(\phi);\phi,\psi) - R_p(\lambda, \phi,\phi)| \asto 0$.
    From \Cref{lem:v0} and \Cref{lem:ridge-fixed-point-v-properties}, $1/v_p(-\lambda;\psi) \in [-r_{\min},\infty]$ and $\lim_{\psi\rightarrow \phi}\sQ_p(\lambda_{\min}(\phi);\phi,\psi) = +\infty$.
    Similarly to the proof of Lemma 28 from \citet{bellec2023corrected}, one can show that the sequence of functions $\{\sQ_p(\lambda_{\min}(\phi);\phi,\psi(\lambda)) - R_p(\lambda, \phi,\phi)\}_{p\in\NN}$ is uniformly equicontinuous on $\lambda\in\Lambda=[\lambda_{\min}(\phi)+\epsilon,\infty]$ almost surely for some small $\epsilon>0$ such that $|\sQ_p(\lambda_{\min}(\phi);\phi,\psi(\lambda))|$ is not greater than the null risk $\sQ_p(+\infty;\phi,\phi)$ when $\lambda\in[\lambda_{\min}(\phi)+\epsilon,+\infty]$.
    From Theorem 21.8 of \citet{davidson1994stochastic}, it further follows that the sequences converge to zero uniformly over $\Lambda$ almost surely.
    This implies that
    \begin{align*}
        0 & = \limsup_{p} \max_{\lambda\geq\lambda_{\min}(\phi)+\epsilon}\left[ \sQ_p(\lambda_{\min}(\phi);\phi,\psi(\lambda)) - R_p(\lambda, \phi,\phi) \right] \\
        & \geq \min_{\psi\geq\phi} \limsup_{p} \max_{\lambda\geq\lambda_{\min}(\phi)+\epsilon}\left[ \sQ_p(\lambda_{\min}(\phi);\phi,\psi) - R_p(\lambda, \phi,\phi) \right] \\
        &\geq \limsup_{p} \min_{\psi\geq\phi}\max_{\lambda\geq\lambda_{\min}(\phi)+\epsilon}\left[ \sQ_p(\lambda_{\min}(\phi);\phi,\psi) - R_p(\lambda, \phi,\phi) \right] \\
        &=\limsup_{p} \left[\min_{\psi\geq\phi} \sQ_p(\lambda_{\min}(\phi);\phi,\psi) - \min_{\lambda\geq\lambda_{\min}(\phi)+\epsilon}R_p(\lambda, \phi,\phi) \right].
    \end{align*}
    
    Conversely, since for any $\psi\geq \psi(\lambda_{\min}(\phi)+\epsilon)$, there exists $\lambda\geq \lambda_{\min}(\phi)+\epsilon$ such that $| \sQ_p(\lambda_{\min}(\phi);\phi,\psi) - R_p(\lambda(\psi);\phi,\phi)| \asto 0$.
    Similarly, we can show that $\{\sQ_p(\lambda_{\min}(\phi);\phi,\psi) - R_p(\lambda(\psi);\phi,\phi)\}_{p\in\NN}$ is uniformly equicontinuous on $\psi\in\Psi=[\psi(\lambda_{\min}(\phi)+\epsilon),\infty]$ almost surely.
    This also implies that
    \begin{align*}
        0 &= \liminf_{p} \min_{\psi\geq\psi(\lambda_{\min}(\phi)+\epsilon)}[\sQ_p(\lambda_{\min}(\phi);\phi,\psi) - R_p(\lambda(\psi);\phi,\phi)]\\
        &\leq \max_{\lambda\geq\lambda_{\min}(\phi)} \liminf_{p} \min_{\psi\geq\psi(\lambda_{\min}(\phi)+\epsilon)}[\sQ_p(\lambda_{\min}(\phi);\phi,\psi) - R_p(\lambda, \phi,\phi) ] \\
        &\leq \liminf_{p} \max_{\lambda\geq\lambda_{\min}(\phi)} \min_{\psi\geq\psi(\lambda_{\min}(\phi)+\epsilon)} [\sQ_p(\lambda_{\min}(\phi);\phi,\psi) - R_p(\lambda, \phi,\phi) ] \\
        &= \liminf_{p} \left[\min_{\psi\geq\psi(\lambda_{\min}(\phi)+\epsilon)} \sQ_p(\lambda_{\min}(\phi);\phi,\psi) - \min_{\lambda\geq\lambda_{\min}(\phi)}R_p(\lambda, \phi,\phi)\right].
    \end{align*}
    Combining the previous two inequalities implies that
    \begin{align*}
        \min_{\lambda\geq\lambda_{\min}(\phi)+\epsilon}R_p(\lambda, \phi,\phi) \asympequi \min_{\psi\geq\psi(\lambda_{\min}(\phi)+\epsilon)} \sQ_p(\lambda_{\min}(\phi);\phi,\psi) .
    \end{align*}
    Since
    \begin{align*}
        \min_{\psi\geq\psi(\lambda_{\min}(\phi)+\epsilon)} \sQ_p(\lambda_{\min}(\phi);\phi,\psi) &= \min_{\psi\geq\psi(\lambda_{\min}(\phi))} \sQ_p(\lambda_{\min}(\phi);\phi,\psi) = \min_{\psi\geq\phi} \sQ_p(\lambda_{\min}(\phi);\phi,\psi),
    \end{align*}
    we further have
    \begin{align*}
        \min_{\lambda\geq\lambda_{\min}(\phi)+\epsilon}R_p(\lambda, \phi,\phi) \asympequi \min_{\psi\geq\phi} \sQ_p(\lambda_{\min}(\phi);\phi,\psi).
    \end{align*}
    From the second part, we know that the latter is continuous and monotonically increasing in $\phi$, which finishes the proof.        
\end{proof}   

\begin{lemma}[Out-of-distribution full-ensemble risk asymptotics]\label{lem:ood-risk-asymptotics-ensemble}
    Under \Cref{asm:train-test}, as $k,n,p\rightarrow\infty$ such that $p/n\rightarrow\phi\in(0,\infty)$ and $p/k\rightarrow\psi\in[\phi,\infty]$, for $\lambda\geq \lambda_{\min}(\phi)$, it holds that
    \begin{align}
        R(\hat{\bbeta}^{\lambda}_{k,\infty}) &\asympequi \sR_p(\lambda, \phi,\psi) \notag\\
        &:=\sQ_p(\lambda, \phi,\psi)  - 2\bbeta^\top (v_p(\lambda, \phi)\bSigma + \bI)^{-1} \bSigma_0 (\bbeta - {\bbeta}_0) + 
        [(\bbeta_0 - \bbeta)^\top \bSigma_0 (\bbeta_0 - {\bbeta}) 
        + \sigma_0^2] .  \label{eq:risk-det-equiv-ensemble}
    \end{align}
    where
    \begin{align*}
        \sQ_p(\lambda, \phi,\psi):= \tc_p(\lambda, \phi,\psi,\bSigma_0)  + \|\fNL\|_{L_2}^2\tv_p(\lambda, \phi,\psi,\bSigma_0),
    \end{align*}
    where the non-negative constants $\tc_p(\lambda, \phi,\psi,\bSigma_0)$ and $\tv_p(\lambda, \phi,\psi,\bSigma_0)$ are defined through the following equations:
    \begin{align*}
        \frac{1}{v_p(\lambda, \psi)} &= \lambda+\psi \int\frac{r}{1+v_p(\lambda, \psi)r }\rd H_p(r),\\
        \tv_p(\lambda, \phi,\psi,\bSigma_0) &= \ddfrac{\phi \otr[\bSigma_0\bSigma(v_p(\lambda, \psi)\bSigma+\bI)^{-2}]}{v_p(\lambda, \psi)^{-2}-\phi \int\frac{r^2}{(1+v_p(\lambda, \psi)r)^2}\rd H_p(r)},\\
        \tc_p(\lambda, \phi,\psi,\bSigma_0) &=
        \bbeta^{\top}(v_p(\lambda, \psi)\bSigma+\bI)^{-1}(\tv_p(\lambda, \phi,\psi,\bSigma_0)\bSigma+\bSigma_0) (v_p(\lambda, \psi)\bSigma+\bI)^{-1}\bbeta .
    \end{align*}
\end{lemma}
\begin{proof}
    Similar to the proof of \Cref{prop:ood-risk-asymptotics}, we have the decomposition
    \begin{align*}
        R(\hat{\bbeta}_{k,\infty}^{\lambda}) &= (\hbeta_{k,\infty}^\lambda - \bbeta_0)^\top \bSigma_0 (\hbeta_{k,\infty}^\lambda - {\bbeta}_0)
        + \sigma_0^2 \\
        &= (\hbeta_{k,\infty}^\lambda - \bbeta)^\top \bSigma_0 (\hbeta_{k,\infty}^\lambda - {\bbeta}) +
        2(\hbeta_{k,\infty}^\lambda - \bbeta)^\top \bSigma_0 (\bbeta - {\bbeta}_0) + 
        [(\bbeta_0 - \bbeta)^\top \bSigma_0 (\bbeta_0 - {\bbeta}) 
        + \sigma_0^2] \\
        &\asympequi \sQ_p(\lambda,\phi,\psi) + 2\EE_{I\sim\cI_k}[(\hbeta(I)- \bbeta)^\top \bSigma_0 (\bbeta - {\bbeta}_0)\mid \cD_n] +  [(\bbeta_0 - \bbeta)^\top \bSigma_0 (\bbeta_0 - {\bbeta}) 
        + \sigma_0^2]  \\
        &\asympequi \sQ_p(\lambda,\phi,\psi) - 2\bbeta^\top (v_p(\lambda, \psi)\bSigma + \bI)^{-1} \bSigma_0 (\bbeta - {\bbeta}_0) + [(\bbeta_0 - \bbeta)^\top \bSigma_0 (\bbeta_0 - {\bbeta}) 
        + \sigma_0^2],
    \end{align*}
    where the first asymptotic equivalent is from \Cref{thm:monotonicty-gen} and the second is from \citet[Lemma D.1 (1)]{patil2023generalized}.
    This finishes the proof.
\end{proof}

\section{Technical Lemmas}
\label{sec:analytic-properties-fp-sols}

\subsection{Fixed-Point Equations for Minimum Ridge Penalty}
\label{subsec:analytic-properties-fp-sols-min-lam}

Recall that under \Cref{asm:train-test}, the minimum ridge penalty $\lambda_{\min}=\lambda_{\min}(\phi)$ can be determined by the following equations:
    \begin{align*}
        1 &= \phi \int \left(\frac{v_0r}{1+v_0r}\right)^2 \rd P(r) ,\qquad -r_{\min} < v_0^{-1}\leq 0 \\
        \frac{1}{v_0} &= \lambda_{\min} + \phi \int \frac{r}{1+v_0r} \rd P(r),  
    \end{align*}
    or equivalently
    \begin{align*}
        1 &= \phi \int \left(\frac{r}{\mu_0+r}\right)^2 \rd P(r) ,\qquad -r_{\min} < \mu_0\leq 0 \\
        \mu_0 &= \lambda_{\min} + \phi \int \frac{\mu_0 r}{\mu_0+r} \rd P(r).
    \end{align*}
    We next analyze the properties of $\mu_0$ in $\phi$.
    
\begin{lemma}[Continuity properties with the minimum regularization parameter]\label{lem:v0}
    Let $a > 0$ and $b < \infty$ be real numbers.
    Let $P$ be a probability measure supported on $[a, b]$.
    Consider the function $v_0(\cdot) : \phi \mapsto \mu_0(\phi)$, over $(0, \infty)$, where $- a\leq \mu_0(\phi) \leq 0$ is the unique solution to the following fixed-point equation:
    \begin{align}
        1 &= \phi \int \left(\frac{r}{\mu_0(\phi)+r}\right)^2 \rd P(r) , \label{eq:v0}
    \end{align}
    Then the following properties hold:
    \begin{enumerate}[label={(\arabic*)},font=\normalfont]
        \item The range of $\mu_0(\phi)$ is $[-a,\infty)$.
        
        \item $\mu_0(\phi)$ is continuous and strictly increasing over $\phi\in(0,\infty)$.
        In addition, $\mu_0(\phi)$ has a root at $\phi=1$.

        \item $\lambda_{\min}(\phi) = \mu_0(\phi)(1 - \phi\int r/(\mu_0(\phi)+r) \rd P(r))$ is non-positive over $\phi\in(0,\infty)$ with zero obtained when $\phi=1$.
        Furthermore, it is strictly increasing over $\phi\in(0,1)$ and strictly decreasing over $\phi\in(1,\infty)$.
    \end{enumerate}
\end{lemma}
\begin{proof}
    The existence of the solution $v_0(\phi) = 1/\mu_0(\phi)$ to the fixed-point equation 
    \[\frac{1}{\phi} = \int \left(\frac{v_0(\phi)r}{1+v_0(\phi)r}\right)^2 \rd P(r) \]
    follows from Theorem 3.1 of \citet{lejeune2022asymptotics}.
    Next, we split the proof into different parts.

    \paragraph{Part (1).}
    Define $h(x) = \int (xr)^2 / (1 +xr)^2 \rd P(r)$.

    When $\phi< 1$, $h(v_0(\phi)) = 1/\phi > 1$, which implies that $1/v_0(\phi)\in [-a,0)$.
    When $\phi=+\infty$, $1/v_0(\phi)=-a$ if $P(a)>0$.
    
    When $\phi\geq 1$, $h(v_0(\phi)) = 1/\phi \leq 1$, which implies that $v_0(\phi)\in(0,\infty]$, with infinity obtained when $\phi=1$, or equivalently, $1/v_0(\phi)\in[0,\infty)$.

    \paragraph{Part (2).}
    Since $g(t)=h(t^{-1})^{-1}$ is positive, strictly increasing and continuous over $t\in [-a, \infty)$, by the continuous inverse theorem, we have that $1/v_0(\phi)=g^{-1}(\phi)$ is strictly increasing and continuous over $\phi\in (0,\infty)$.
    From Part (1), we also have $1/v_0(1)=0$, which finishes the proof.

    \paragraph{Part (3).} 
    Consider the function $f(x)=1 - \phi\int r / (x+r)\rd P(r)$.
    When $\phi<1$, we know that $f(\mu_0(\phi))>0$ because from \eqref{eq:v0},
    \begin{align*}
        1 &= \phi \int \left(\frac{r}{\mu_0(\phi)+r}\right)^2 \rd P(r) > \phi \int \frac{r}{\mu_0(\phi)+r} \rd P(r) ,
    \end{align*}
    which holds because $\mu_0(\phi)<0$ from Part (2).
    Analogously, when $\phi>1$, $f(\mu_0(\phi))<0$.    
    Therefore, $\lambda_{\min}(\phi)=\mu_0(\phi) f(\mu_0(\phi))\leq0$ with equality obtained when $\phi=1$.
    
    Because $f(x)$ is strictly decreasing in both $\phi$ and $x$, combining the sign properties, we further find that $\lambda_{\min}(\phi)=\mu_0(\phi) f(\mu_0(\phi))$ is strictly increasing over $\phi\in(0,1 )$ and strictly decreasing over $\phi\in(1,\infty )$.
\end{proof}

\subsection{Properties of Fixed-Point Equations under Negative Regularization}
\label{subsec:analytic-properties-fp-sols-v}

Our analysis involves $v(\lambda, \phi)$ as a unique solution to the fixed-point equation in 
\begin{align}
    \label{eq:basic-ridge-equivalence-v-fixed-point}
    \frac{1}{v(\lambda, \phi)}
    =
    \lambda
    + \phi \int \frac{r}{1+v(\lambda, \phi)r} \, \rd H_p(r).
\end{align}
Define $\mu(\lambda, \phi) = 1/v(\lambda, \phi)$. Equivalently, we have that $\mu(\lambda, \phi)$ is a unique solution to the following fixed-point equation:
\begin{align}
    \label{eq:basic-ridge-equivalence-mu-fixed-point}
    \mu(\lambda,\phi)
    =
    \lambda
    + \phi \int \frac{\mu(\lambda,\phi)r}{\mu(\lambda,\phi)+r} \, \rd H_p(r).
\end{align}

The analytic properties of the function $\lambda\mapsto v(\lambda, \phi)$ on $(\lambda_{\min}(\phi),+\infty)$ for $\phi\in(1,\infty)$ are detailed in \Cref{lem:fixed-point-v-lambda-properties}. 

\begin{lemma}[Analytic properties in the regularization parameter]
    \label{lem:fixed-point-v-lambda-properties}
    Let $0<a \leq b<\infty$ be real numbers.
    Let $P$ be a probability measure supported on $[a, b]$.
    Let $\phi>0$ be a real number.
    For $\lambda > \lambda_{\min}(\phi)$, let $v(\lambda, \phi) > 0$ denote the solution to the fixed-point equation
    \[
        \mu(\lambda, \phi)
        =  \lambda
        + \phi \int \frac{\mu(\lambda, \phi)r}{ r + \mu(\lambda, \phi)} \, \mathrm{d}P(r).
    \]
    Then the following properties hold:
    \begin{enumerate}[(1),font=\normalfont]
        \item\label{lem:fixed-point-v-lambda-properties-item-monotonicity} \emph{(Monotonicity)} For $\phi\in(0,\infty)$, the function $\lambda \mapsto \mu(\lambda, \phi)$ is strictly increasing in $\lambda\in(\lambda_{\min}(\phi),\infty)$.

        \item \emph{(Range)} For $\phi\in(0,1]$, $\lim_{\lambda\rightarrow\lambda_{\min}(\phi)^-}\mu(\lambda, \phi)=-\infty$ and $\mu(0,\phi)=0$.
        For $\phi\in(1,\infty)$, $\lim_{\lambda\rightarrow\lambda_{\min}(\phi)^-}\mu(\lambda, \phi)\in(0,\infty)$.
        For $\phi\in(0,\infty)$, $\lim_{\lambda\rightarrow+\infty}\mu(\lambda, \phi)=+\infty$.
        
        \item\label{lem:fixed-point-v-lambda-properties-item-differentiability} \emph{(Differentiability)} 
        For $\phi\in(0,\infty)$, the function $\lambda \mapsto \mu(\lambda, \phi)$ is differentiable over $\Lambda$.

        \item \label{lem:fixed-point-v-lambda-properties-item-monotonicity-lam-mu} The function $\lambda \mapsto \lambda/\mu_p(\lambda,\phi)$ is strictly increasing in $\lambda$ with $\lim_{\lambda\rightarrow0}\lambda /\mu_p(\lambda,\phi)=0$ and $\lim_{\lambda\rightarrow0}\lambda /\mu_p(\lambda,\phi)=1$.
    \end{enumerate}
\end{lemma}
\begin{proof}
    Note that 
    \[
        \lambda
        =  
        \mu(\lambda, \phi) - \phi \int \frac{\mu(\lambda, \phi)r}{ r + \mu(\lambda, \phi)} \, \mathrm{d}P(r).
    \]
    Define a function $f$ by
    \[f(x) = x - \phi\int\frac{xr}{x+r}\rd P(r).\]
    Observe that $\mu(\lambda,\phi)=f^{-1}(\lambda)$. The claim of differentiability of the function $\lambda\mapsto \mu(\lambda,\phi)$ follows from the differentiability and strict monotonicity of $f$, similar to \citet[Lemma S.6.14]{patil2022mitigating}.

    For the last property, from the definition of the fixed-point equation, we have
    \begin{align*}
        1 = \frac{\lambda}{\mu_p(\lambda,\phi)} + \phi \int \frac{r}{r + \mu(\lambda,\phi)} \rd P(r).
    \end{align*}
    Because $\phi \int \frac{r}{r + \mu(\lambda,\phi)} \rd P(r)$ is strictly decreasing in $\lambda$, we have that $\lambda/\mu_p(\lambda,\phi)$ is strictly increasing in $\lambda$.
    Because $\lim_{\lambda\rightarrow +\infty}\mu_p(\lambda,\phi) = +\infty$, we know that
    \begin{align*}
        \lim_{\lambda\rightarrow \infty}\phi \int \frac{r}{r + \mu(\lambda,\phi)} \rd P(r) = 0.
    \end{align*}
    and thus,
    \begin{align*}
        \lim_{\lambda\rightarrow \infty} \frac{\lambda}{\mu_p(\lambda,\phi)}  = 1.
    \end{align*}
    On the other hand, because $\lim_{\lambda\rightarrow 0}\mu_p(\lambda,\phi) = 0$ for $\phi\in(0,1]$ and $\lim_{\lambda\rightarrow 0}\mu_p(\lambda,\phi)<\infty$, it directly implies that
    \begin{align*}
        \lim_{\lambda\rightarrow 0} \frac{\lambda}{\mu_p(\lambda,\phi)}  = 0.
    \end{align*}
    Consequently, the conclusion follows.
\end{proof}

\begin{lemma}[Continuity properties in the aspect ratio for ridge regression, adapted from \citet{patil2022bagging}]
    \label{lem:ridge-fixed-point-v-properties}
    Let $a > 0$, $b < \infty$ and $\lambda\in \RR$ be real numbers.
    Let $P$ be a probability measure supported on $[a, b]$.
    For $\lambda\in \RR$, let $\Phi(\lambda) = \{\phi \in (0,\infty) \mid \lambda_{\min}(\phi) < \lambda\}$.
    Consider the function $\mu(\lambda, \cdot) : \phi \mapsto \mu(\lambda, \phi)$, over $\phi\in\Phi(\lambda)$ such that $\lambda_{\min}(\phi)\leq \lambda$, where $\mu(\lambda, \phi) \geq - a$ is the unique solution to the following fixed-point equation:
    \begin{equation}
        \label{eq:ridge-fixed-point-gen-phi}
        \mu(\lambda, \phi)
         = \lambda + \phi \int \frac{\mu(\lambda, \phi) r}{\mu(\lambda, \phi)+r} \rd P(r).
    \end{equation}
    Then the following properties hold:
    \begin{enumerate}[label={(\arabic*)},font=\normalfont]
        \item 
        \label{lem:ridge-fixed-point-v-properties-item-v-bound}        
        The range of the function $\mu(-\lambda; \cdot)$ is a subset of $[\lambda,\infty]$ when $\lambda\geq0$ and $[-a,\infty]$ when $\lambda<0$.
        
        \item 
        \label{lem:ridge-fixed-point-v-properties-item-v-properties}
        The function $\mu(-\lambda; \cdot)$ is continuous and strictly increasing over $\Phi(\lambda)$.
        Furthermore, $\lim_{\phi \to \infty} \mu(\lambda, \phi) = +\infty$, $\lim_{\phi \to 0^{+}} \mu(\lambda, \phi) = \lambda$ when $\lambda\geq 0$, and $\lim_{\phi \to 0^{+}} \mu(\lambda, \phi) = -\infty$ when $\lambda< 0$.

        \item
        \label{lem:ridge-fixed-point-v-properties-item-vv-properties}
        The function 
        $\tv_v(-\lambda; \cdot) : \phi \mapsto \tv_v(\lambda, \phi)$,
        where
        \[
           \tv_v(\lambda, \phi) =
           \left(
                \mu(\lambda, \phi)^{2}
                - \phi \int  (\mu(\lambda, \phi)r)^2(\mu(\lambda, \phi) + r )^{-2} 
                \, \mathrm{d}P(r)
           \right)^{-1},
        \]
        is positive and continuous over $\Phi(\lambda)$.
        
        \item
        \label{lem:ridge-fixed-point-v-properties-item-vb-properties}
        The function 
        $\tv_b(-\lambda; \cdot) : \phi \mapsto \tv_b(\lambda, \phi)$,
        where
        \[
            \tv_b(\lambda, \phi)
            = \tv_v(\lambda, \phi)
            \int
            \phi (\mu(\lambda, \phi)r)^2(\mu(\lambda, \phi) + r)^{-2}
            \, \mathrm{d}P(r),
        \]
        is positive and continuous over $\Phi(\lambda)$.
    \end{enumerate}
\end{lemma}
\begin{proof}
    Properties (1)-(4) follow similarly to \citet[Lemma S.6.15]{patil2022mitigating}.
\end{proof}

\subsection{Contours of Fixed-Point Solutions under Negative Regularization}
\label{subsec:analytic-properties-fp-sols-contour}

\begin{lemma}[Contours of fixed-point solutions]\label{lem:v-equiv-path}
    As $n,p\rightarrow\infty$ such that $p/n\rightarrow\phi\in(0,\infty)$, let $\lambda_{\min}:\phi\mapsto\lambda_{\min}(\phi)$ as defined in \eqref{eq:lam_min}.
    For any $\bar{\psi}\in[\phi,+\infty]$, there exists a unique value $\bar{\lambda} \geq \lambda_{\min}(\bar{\psi})$ (or conversely for $\bar{\lambda}\in[\lambda_{\min}(\phi),\infty]$, there exists a unique value $\bar{\psi}\in[\phi,\infty]$) such that for all $(\lambda, \psi)$ on the path (as in \Cref{fig:inv_v})
    \begin{align}
        \cP=\{(1 - \theta)\cdot(\bar{\lambda}, \phi) + \theta\cdot(\lambda_{\min}(\bar{\psi}), \bar{\psi})\mid \theta\in[0, 1]\}, \label{eq:path}
    \end{align}
    it holds that
    \begin{align*}
        \mu(\lambda,\psi) = \mu(\overline{\lambda}, \phi) = \mu(\lambda_{\min}(\bar{\psi}), \bar{\psi}),
    \end{align*}
    where $\mu(\lambda,\psi)$ is as defined in \eqref{eq:basic-ridge-equivalence-mu-fixed-point}.
    
    \begin{figure}[!ht]
        \centering
        \includegraphics[width=0.6\textwidth]{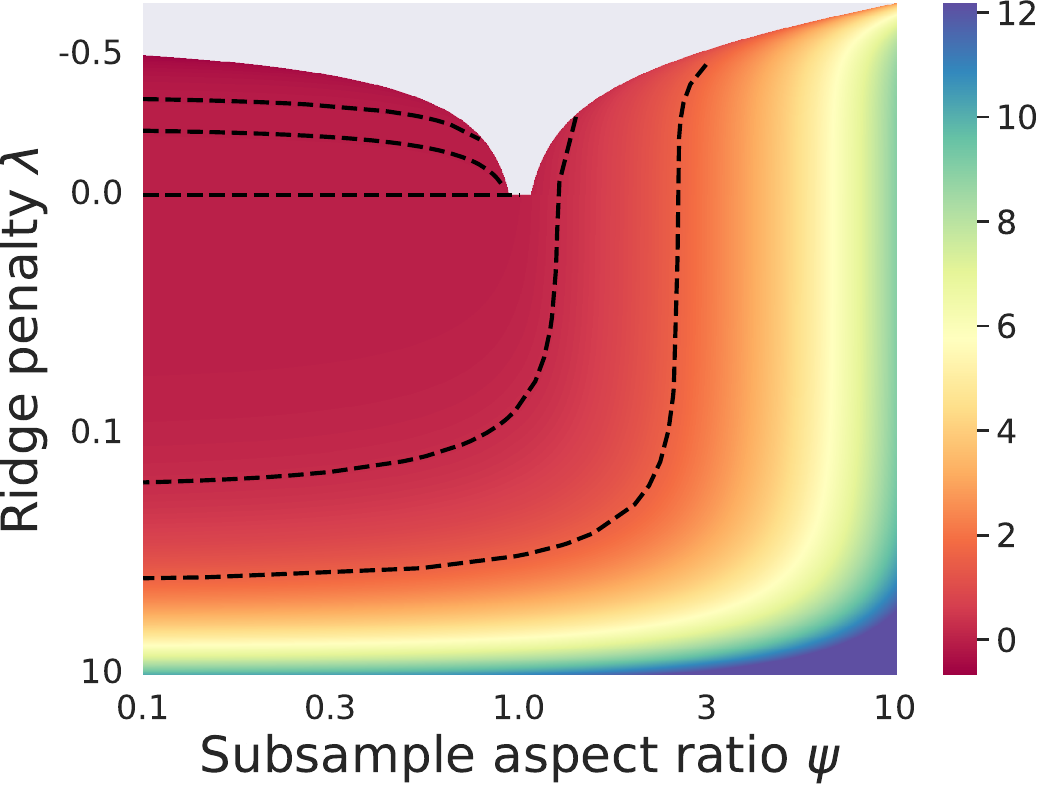}
        \caption{Heatmap of $\mu_p(\lambda, \psi)$ for isotropic covariance matrix $\bSigma=\bI$ in the symmetric log-log scale (where the logarithmic scale is applied symmetrically to both positive and negative values on the $x$-axis and $y$-axis).
        The 5 black dashed lines indicate 5 different equivalence paths.
        The boundary of negative ridge penalties is given by $-(1 - \sqrt{\psi})^2$.}
        \label{fig:inv_v}
    \end{figure}
    
\end{lemma}
\begin{proof}
    Note that $\mu(\lambda_{\min}(\phi),\phi) = \mu_0(\phi)$ which is the solution to the fixed-point equation \eqref{eq:v0}.
    From \Cref{lem:v0} (2), we see that the function $\psi\mapsto \mu(\lambda_{\min}(\psi), \psi)$ is strictly increasing over $\psi\in[\phi,\infty]$ with the range \begin{align*} \lim_ {\psi\rightarrow \phi} \mu(\lambda_{\min}(\psi), \psi) = 
        \mu_0(\phi), \quad \lim_{\psi\rightarrow +\infty} \mu(\lambda_{\min}(\psi), \psi) = 
        +\infty.
    \end{align*}
    From \Cref{lem:ridge-fixed-point-v-properties}, the function $\lambda\mapsto \mu(\lambda, \phi)$ is strictly increasing over $\lambda\in[\lambda_{\min},\infty]$ with range
    \begin{align*}
        \lim_{\lambda\rightarrow \lambda_{\min}} \mu(\lambda, \phi) = 
        \mu_0(\phi), \quad \lim_{\lambda\rightarrow +\infty} \mu(\lambda, \phi) = +\infty.
    \end{align*}

    For $\bar{\psi}\in[\phi,\infty]$, by the intermediate value theorem, there exists a unique $\bar{\lambda}\in[\lambda_{\min}(\phi),\infty]$ such that $v(\bar{\lambda}, \phi)=v(-\lambda_{\min}(\bar{\psi});\bar{\psi})$.
    Conversely, for $\bar{\lambda}\in[\lambda_{\min}(\phi),\infty]$, there also exists a unique $\bar{\psi}\in[\phi,\infty]$ such that $v(\bar{\lambda}, \phi)=v(-\lambda_{\min}(\bar{\psi});\bar{\psi})$. 

    Based on the definition of fixed-point solutions, it follows that
    \begin{align*}
        \mu(\bar{\lambda}, \phi) &= \bar{\lambda} + \phi\int\frac{\mu(\bar{\lambda}, \phi) r}{\mu(\bar{\lambda}, \phi) +  r}\rd H_p(r) = \lambda_{\min}(\phi) + \bar{\psi}\int\frac{\mu(-\lambda_{\min};\bar{\psi})r}{\mu(-\lambda_{\min};\bar{\psi}) +  r}\rd H_p(r) = \mu(-\lambda_{\min};\bar{\psi}) .
    \end{align*}
    Then, for any $(\lambda, \psi) = (1-\theta)(\bar{\lambda},\phi) + \theta (\lambda_{\min}(\bar{\psi}),\bar{\psi})$ on the path $\cP$, we have
    \begin{align*}
        \mu(\bar{\lambda}, \phi) &= (1-\theta)\mu(\bar{\lambda}, \phi) + \theta\mu(\lambda_{\min}(\bar{\psi}), \bar{\psi}) \\
        &= (1-\theta) \bar{\lambda} + (1-\theta)\phi\int\frac{\mu(\bar{\lambda}, \phi) r}{\mu(\bar{\lambda}, \phi) +  r}\rd H_p(r) + \theta\lambda_{\min}(\bar{\psi}) + \theta\bar{\psi}\int\frac{\mu(-\lambda_{\min}(\bar{\psi});\bar{\psi})r}{\mu(-\lambda_{\min}(\bar{\psi});\bar{\psi}) + r}\rd H_p(r)\\
        &= \lambda + \psi\int\frac{\mu(\bar{\lambda}, \phi)r}{\mu(\bar{\lambda}, \phi) + r}\rd H_p(r).
    \end{align*}
    Because $\mu(\lambda, \psi)$ is the unique solution to the fixed-point equation:    
    \begin{align*}
        \mu(\lambda, \psi) &= \lambda + \psi \int \frac{\mu(\lambda, \psi)r}{\mu(\lambda, \psi)+r} \rd H_p(r),
    \end{align*}
    it then follows that $\mu(\lambda, \psi) = \mu(\bar{\lambda}, \phi) = \mu(\lambda_{\min}(\bar{\psi}),\bar{\psi})$.
    This completes the proof.
\end{proof}

\clearpage
\section{Experimental Details and Additional Numerical Illustrations}
\label{sec:additional_numerical_illustrations}

The source code for reproducing the results of this paper can be found at the following location: \url{https://github.com/jaydu1/ood-ridge}.

\subsection{Additional Illustrations for \Cref{sec:optimal_regularization_signs}}
\label{sec:additional-illustrations-sec:optimal_regularization_signs}

\subsubsection[Numerical verification of the general alignment condition]{Numerical verification of \eqref{eq:align-cond} under generic alignment scenarios}

Under \Cref{asm:train-test}, suppose the data model has a covariance matrix $\bSigma$ to be $(\bSigma_{\mathrm{ar1}})_{ij}:= \rhoar^{|i-j|}$ with parameter $\rhoar\in(0,1)$, and a coefficient $\bbeta:=\frac{1}{2}(\bw_{(1)} + \bw_{(p)})$, where $\bw_{(j)}$ is the $j$th eigenvector of $\bSigma_{\mathrm{ar1}}$.

\begin{align*}
    \otr[\bSigma \bB] &= \frac{1}{4p}(\bw_{(1)} + \bw_{(p)})^{\top}\bSigma(\bw_{(1)} + \bw_{(p)}) =\frac{1}{4p}\bw_{(1)}^{\top}\bSigma\bw_{(1)} + \frac{1}{4p}\bw_{(p)}^{\top}\bSigma\bw_{(p)}=\frac{1}{4}(r_{(1)}+r_{(p)}).
\end{align*}
On the other hand, we also have
\begin{align*}
    \otr[\bSigma] \otr[\bB] &= \sum_{j=1}^p \frac{r_{(j)}}{p}  \cdot \frac{\|\bbeta\|_2^2}{p} = \frac{1}{2p^2}\sum_{j=1}^pr_{(j)}
\end{align*}
One can numerically verify that when $p=500 $ and $\rhoar=0.5$,
\begin{align*}
    r_{(1)}+r_{(p)} \approx 3.33 > 2 = \frac{2}{p}\sum_{j=1}^pr_{(j)}
\end{align*}
which contradicts the implication \eqref{eq:strict-align-imp} of the strict alignment condition.

\begin{figure}[!ht]
    \centering
    \includegraphics[width=0.6\textwidth]{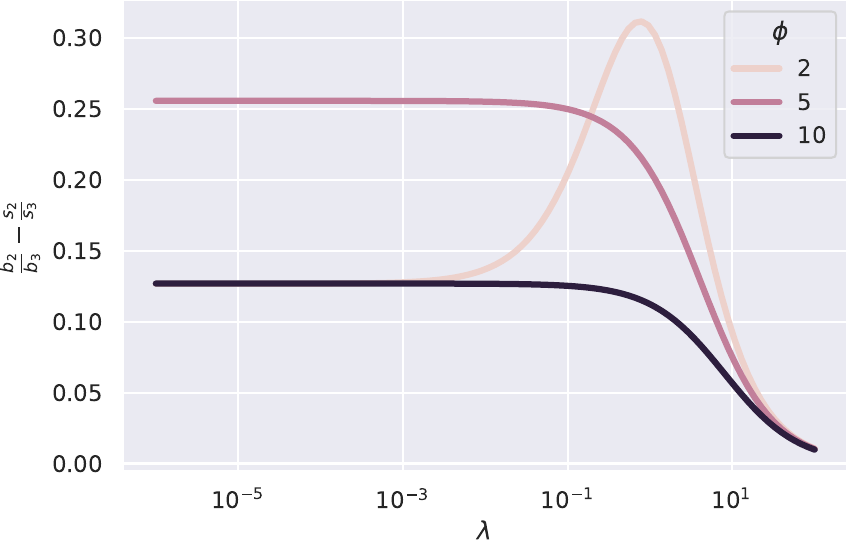}
    \caption{Numerical evaluation of $b_2/b_3-s_2/s_3$ under the same data model in \Cref{fig:negative_optimal_our_condition}, where $s_k=\otr[\bSigma \{ \mu^{k} (\bSigma + \mu \bI)^{-k} \}]$ and $b_k=\otr[\bSigma_{\bbeta} \{ \mu^{k} (\bSigma + \mu \bI)^{-k} \}]$ for $k=2,3$.}
    \label{fig:align-cond}
\end{figure}

On the other hand, from \Cref{fig:align-cond}, we see that the general alignment condition
\begin{align*}
    \frac{\otr[\bSigma_{\bbeta} \{ \mu^{2} (\bSigma + \mu \bI)^{-2} \}]
    }{\otr[\bSigma_{\bbeta} \{ \mu^3 (\bSigma + \mu \bI)^{-3} \}] 
    } &=  \frac{\otr[\bSigma \{ \mu^{2} (\bSigma + \mu \bI)^{-2} \}]
    }{\otr[\bSigma\{ \mu^3 (\bSigma + \mu \bI)^{-3} \}] 
    }
\end{align*}
holds for various data aspect ratios in the noiseless setting.

\subsubsection{Real Data Illustration}
\label{sec:real_data_illustration}

Following the approach of \citet{kobak_lomond_sanchez_2020}, we consider a similar setup using random Fourier features on MNIST. 

\paragraph{Feature generation.}
The pixel values are normalized to the range $[-1,1]$. 
The features are then mapped from the original dimension of $28\times 28$ to 1000 random Fourier features. 
This is achieved by multiplying the features with a $784\times500$ random matrix $\bW$, where the elements are independently drawn from a normal distribution with mean $0$ and standard deviation $0.2$. 
The real and imaginary parts of $\exp(-i \bX\bW)$ are taken as separate features, where $\bX\in\RR^{n\times784}$. 

\paragraph{Training details.}
For training, we randomly select $n=64$ images, and for testing, we hold out 10,000 images. 
The response variable $y$ represents the digit value ranging from $0$ to $9$. 
Our model includes an intercept term, which is not subject to penalization. 
To estimate the expected out-of-distribution risk, we average the risks across 100 random samples from the training set.

\paragraph{Distribution shift.}
To generate distribution shift, we gradually exclude samples with given labels in the test set: 
\begin{enumerate}[-,leftmargin=4mm]
    \item \texttt{Type 1}: $\varnothing$
    \item \texttt{Type 2}: $\{4\}$ 
    \item \texttt{Type 3}: $\{3,4\}$ 
    \item \texttt{Type 4}: $\{2,3,4\}$ 
    \item \texttt{Type 5}: $\{1,2,3,4\}$
\end{enumerate}
In other words, \texttt{Type} 1 represents no covariate shift, while \texttt{Type} 5 represents the case with potentially the most severe covariate shift. 
The results are summarized in \Cref{fig:MNIST}. 
We observe a clear pattern where the optimal ridge penalty shifts toward negative values, suggesting that \Cref{thm:stationary-point-cov-shift} may occur in real-world datasets. 
However, for \texttt{Type} 5, the optimal ridge penalty becomes positive again. This could be due to the removal of Class 1, which reduces the degree of alignment.
Also, observe that the minimum risks increase from \texttt{Type 1} to \texttt{5}.

\begin{figure*}[!ht]
    \centering
    \includegraphics[width=0.99\textwidth]{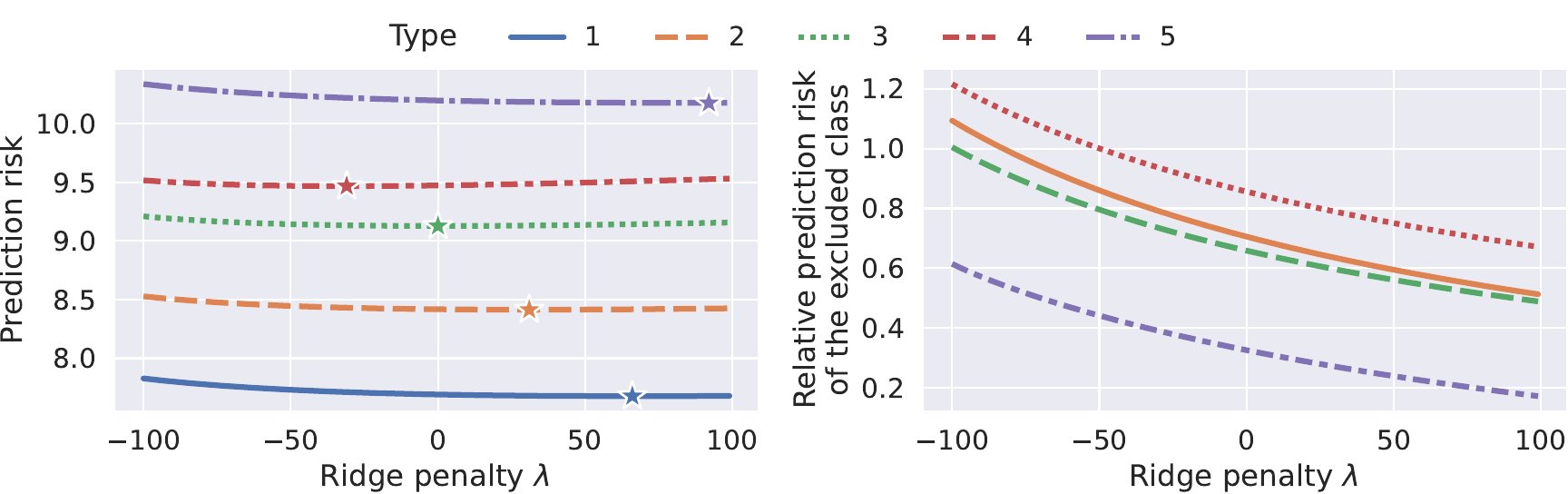}
    \caption{
        \textbf{Effect of distribution shift on the optimal ridge penalty on MNIST.}
        (a) The left panel illustrates the risk profile (against the regularization penalty) of ridge regression on the MNIST dataset when subjected to different types of distribution shifts.         
        Different colors represent different types of shift from less severe (\texttt{Type} 1) to more severe (\texttt{Type} 5).
        The y-axis represents the out-of-distribution prediction risk for the task of accurately predicting the digit value for unseen images. 
        The figure shows a clear pattern where the optimal ridge penalty shifts towards negative values, in the spirit of \Cref{thm:stationary-point-sig-shift}.
        The only exception seems to be \texttt{Type} 5 for which the optimal ridge penalty becomes positive again.
        It is likely due to the removal of Class 1, which reduces the degree of alignment.
        (b) The right panel shows the relative OOD prediction risk computed only on the excluded class, compared to the in-distribution prediction risk of the ridge predictor fitted only on the training data of the same class.
        The reason we compare the relative prediction risk is to compensate for the differences in the conditional variances in the different classes. 
        Observe that Class 1 (\texttt{Type} 5) has the lowest relative prediction risk, which partially explains the increase in optimal regularization of \texttt{Type} 5 in the left panel.
        }
    \label{fig:MNIST}
\end{figure*}

\clearpage
\subsection{Additional Illustrations for \Cref{sec:discussion} (Ridge versus Lasso Monotonicities)}
\label{sec:additional_illustrations-sec:discussion}

\subsubsection{Subsampled ridgeless versus full ridge}
\begin{figure*}[!ht]
    \centering
    \includegraphics[width=0.99\textwidth]{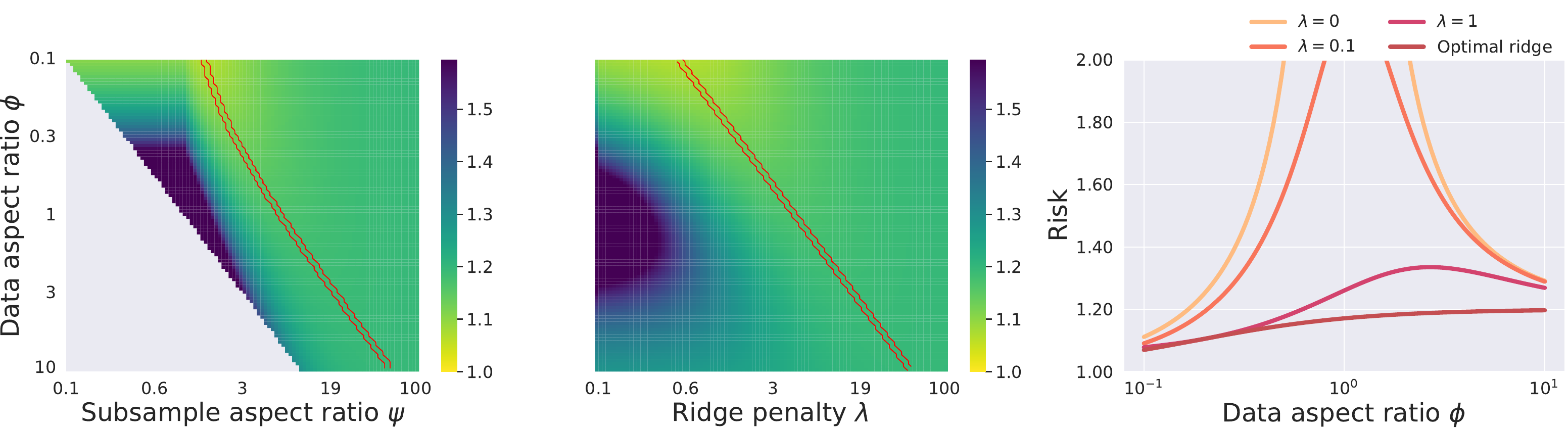}
    \caption{
    \textbf{Ridge optimal risk monotonicity and suboptimal risk non-monotonicity.}
    Optimal ridge has a monotonic risk in the data aspect ratio, but the risk for fixed ridge penalty $\lambda$ may not be monotonic in the data aspect ratio.
    The data model has $\bSigma=\bSigma_{\mathrm{ar1}}$ with $\rhoar=0.25$ (the covariance matrix of an auto-regressive process of order 1 (AR(1)) is given by $\bSigma_{\mathrm{ar1}}$, where $(\bSigma_{\mathrm{ar1}})_{ij} = \rhoar^{|i-j|}$ for some parameter $\rhoar\in(0,1)$), $\bbeta$ being the leading eigenvector of $\bSigma$ and $\sigma = 0.5$.
    The leftmost panel shows the limiting risk of the full-ensemble ridgeless regression at various data and subsample aspect ratios $(\phi, \psi)$.
    The middle panel shows the limiting risk of the ridge predictor (on the full data) at various data aspect ratios and regularization penalties $(\phi, \lambda)$.
    We highlight the optimal risks at a given data aspect ratio for the leftmost and middle panels using slender red lines.
    Observe that the optimal risk in both cases is increasing as a function of $\phi$.
    }
    \label{fig:ridge-monotonicity}
\end{figure*}

\subsubsection{Subsampled lassoless versus full lasso}

\begin{figure*}[!ht]
    \centering
    \includegraphics[width=0.99\textwidth]{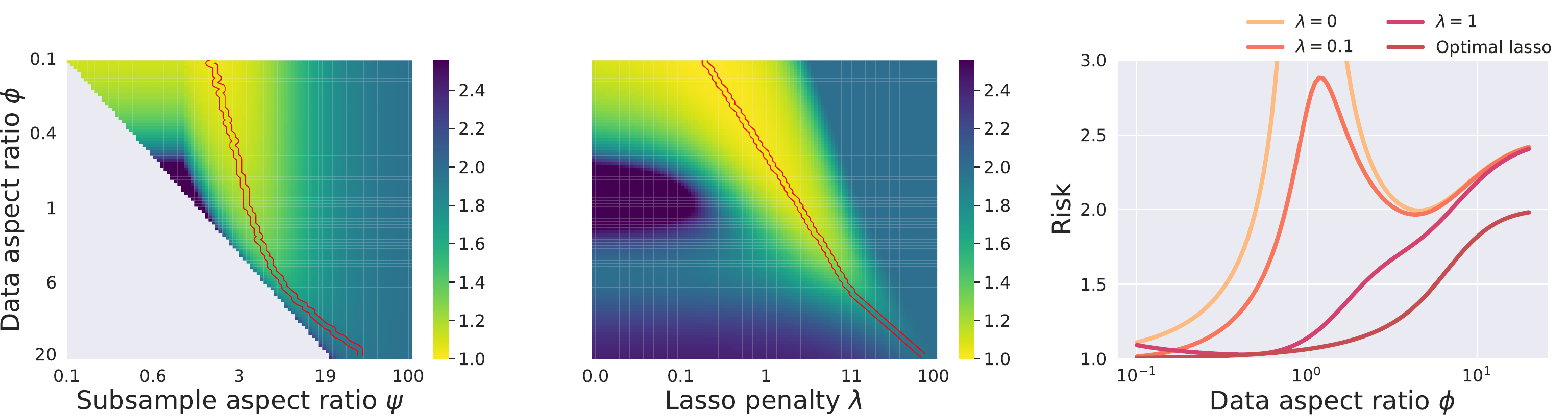}
    \caption{
    \textbf{Lasso optimal risk monotonicity and suboptimal risk non-monotonicity.}
    Similar to ridge regression, optimal lasso has a monotonic risk in the over-parameterization ratio, but the risk for fixed lasso penalty $\lambda$ may not be monotonic in the data aspect ratio.
    The data model has $\bSigma=\bI$ and $\beta_j\overset{\iid}{\sim} \epsilon P_{1/\sqrt{\phi\epsilon}} + (1-\epsilon)P_{0}$ where the sparsity level is $\epsilon=0.01$ and $\sigma^2=1$, such that $\SNR=1$.
    As for ridge regression in \Cref{fig:ridge-monotonicity}, optimal risks for each data aspect ratio are highlighted using slender red lines in the left and middle panels.
    Similarly to the ridge curves in \Cref{fig:ridge-monotonicity}, observe that the optimal risk in both cases is increasing as a function of $\phi$.
    }
    \label{fig:lasso-monotonicity}
\end{figure*}

\end{document}